\documentclass[11pt]{amsart}
\usepackage{amsmath,amsthm,amsfonts,amscd,amssymb,eucal,latexsym,mathrsfs}
\usepackage[all]{xy}

\newtheorem{theorem}{Theorem}[section]
\newtheorem{corollary}[theorem]{Corollary}
\newtheorem{lemma}[theorem]{Lemma}
\newtheorem{proposition}[theorem]{Proposition}

\theoremstyle{definition}
\newtheorem{definition}[theorem]{Definition}
\newtheorem{remark}[theorem]{Remark}

\newtheorem{example}[theorem]{Example}
\newtheorem{question}[theorem]{Question}

\newcommand{\rank}{{\rm rank}}
\newcommand{\CPA}{{\rm CPA}}
\newcommand{\rcp}{{\rm rcp}}

\newcommand{\spn}{{\rm span}}

\newcommand{\re}{{\rm re}}
\newcommand{\im}{{\rm im}}

\newcommand{\interior}{{\rm int}}

\newcommand{\rc}{{\rm rc}}

\newcommand{\id}{{\rm id}}

\newcommand{\cB}{{\mathcal B}}
\newcommand{\cH}{{\mathcal H}}

\newcommand{\cW}{{\mathcal W}}

\newcommand{\cE}{{\mathcal E}}

\newcommand{\cU}{{\mathcal U}}

\newcommand{\cQ}{{\mathcal Q}}
\newcommand{\cP}{{\mathcal P}}
\newcommand{\cR}{{\mathcal R}}
\newcommand{\cS}{{\mathcal S}}
\newcommand{\cZ}{{\mathcal Z}}

\newcommand{\cV}{{\mathcal V}}

\newcommand{\cY}{{\mathcal Y}}
\newcommand{\Cb}{{\mathbb C}}
\newcommand{\Zb}{{\mathbb Z}}

\newcommand{\Rb}{{\mathbb R}}
\newcommand{\Nb}{{\mathbb N}}
\newcommand{\sB}{{\mathscr B}}
\newcommand{\sA}{{\mathscr A}}
\newcommand{\sI}{{\mathscr I}}
\newcommand{\sP}{{\mathscr P}}
\newcommand{\sR}{{\mathscr R}}

\newcommand{\sX}{{\mathscr X}}
\newcommand{\sY}{{\mathscr Y}}
\newcommand{\sZ}{{\mathscr Z}}

\newcommand{\diam}{{\rm diam}}

\newcommand{\supp}{{\rm supp}}

\newcommand{\esssup}{{\rm ess\, sup}}
\newcommand{\IE}{{\rm IE}}

\newcommand{\IN}{{\rm IN}}

\newcommand{\oA}{{\boldsymbol{A}}}
\newcommand{\oB}{{\boldsymbol{B}}}
\newcommand{\oC}{{\boldsymbol{C}}}

\newcommand{\ox}{{\boldsymbol{x}}}

\newcommand{\oX}{{\boldsymbol{X}}}

\newcommand{\comp}{{\rm c}}
\newcommand{\comb}{{\rm c}}

\newcommand{\hcpa}{{\rm hcpa}}

\newcommand{\upind}{{\overline{\rm I}}}
\newcommand{\lwind}{{\underline{\rm I}}}
\newcommand{\hmeas}{{h}}

\newcommand{\Hmeas}{{H}}
\newcommand{\fs}{{\mathfrak{s}}}
\newcommand{\cpct}{{\rm cpct}}
\newcommand{\wm}{{\rm wm}}

\newcommand{\lwh}{{\underline{h}}}
\newcommand{\uph}{{\overline{h}}}

\newcommand{\LX}{L^\infty (X,\mu )}

\allowdisplaybreaks

\begin{document}

\title{Combinatorial independence in measurable dynamics}

\author{David Kerr}
\author{Hanfeng Li}
\address{\hskip-\parindent
David Kerr, Department of Mathematics, Texas A{\&}M University,
College Station TX 77843-3368, U.S.A.}
\email{kerr@math.tamu.edu}

\address{\hskip-\parindent
Hanfeng Li, Department of Mathematics, SUNY at Buffalo,
Buffalo NY 14260-2900, U.S.A.}
\email{hfli@math.buffalo.edu}

\date{May 5, 2007}

\begin{abstract}
We develop a fine-scale local analysis of measure entropy and measure
sequence entropy based on combinatorial independence.
The concepts of measure
IE-tuples and measure IN-tuples are introduced and studied in analogy with their
counterparts in topological dynamics.
Local characterizations of the Pinsker von Neumann algebra and its sequence entropy
analogue are given in terms of combinatorial independence, $\ell_1$ geometry, and
Voiculescu's completely positive approximation entropy.
Among the novel features of our local study is the treatment of general discrete
acting groups, with the structural assumption of amenability in 
the case of entropy.  
\end{abstract}

\maketitle

\section{Introduction}

Many of the fundamental concepts in measurable dynamics revolve around the notion
of probabilistic independence as an indicator of randomness or unpredictability.
Ergodicity, weak mixing, and mixing are all expressions of asymptotic independence,
whether in a mean or strict sense. At a stronger level, completely positive entropy
can be characterized by a type of uniform asymptotic independence (see \cite{ETJ}).

In topological dynamics the appropriate notion of independence is the combinatorial
(or set-theoretic) one, according to which a family of tuples of subsets of a set is
independent if when picking any one subset from each of finitely many tuples
one always ends up with a collection having nonempty intersection.
Combinatorial independence manifests itself dynamically in many
ways and has long played an important role in the topological theory, although it has not
received the same kind of systematic attention as probabilistic independence has in
measurable dynamics. 
In fact it has only been recently that precise relationships
have been established between independence and the properties of
nullness, tameness, and positive entropy \cite{LVRA,Ind}.
For example, a topological $\Zb$-system has uniformly positive entropy if and only if
the orbit of each pair of nonempty open subsets of the space is independent along a
positive density subset of $\Zb$ \cite{LVRA} (see \cite{Ind} for a combinatorial proof
that applies more generally to actions of discrete amenable groups).

The aim of this paper is to develop a theory of combinatorial independence in
measurable dynamics. That such a cross-pollination is at all possible
might be surprising, but it ends up providing, among other things, the missing link
for a geometric understanding of local entropy production in connection with 
Voiculescu's operator-algebraic notion of approximation entropy \cite{Voi}. 
One of our main motivations is to establish local combinatorial
and linear-geometric characterizations of positive entropy and positive sequence entropy.
For automorphisms of a Lebesgue space, the extreme situation of complete positive
entropy was characterized in terms
of combinatorial independence by Glasner and Weiss in Section~3 of \cite{GW}
using Karpovsky and Milman's generalization of the Sauer-Perles-Shelah lemma.
What we see in this case however is an essentially topological phenomenon
whereby independence over positive density subsets of iterates occurs for every
finite partition of the space into sets of positive measure
(cf.\ Theorem~\ref{T-cpe} in this paper). This does not help us much in the analysis
of entropy production for other kinds of systems, as it can easily happen
that combinatorial independence is present but not in a robust enough way
to be measure-theoretically meaningful (indeed every free ergodic
$\Zb$-system has a minimal topological model
with uniformly positive entropy \cite{SEUPEM}).
We seek moreover a fine-scale localization predicated not on
partitions but rather on tuples of subsets that together compose only a very small
fraction of the space, which the Glasner-Weiss result provides for $\Zb$-systems with
completely positive entropy but in the purely topological sense of \cite{Ind}.

It turns out that we should ask whether combinatorial
independence can be observed to the appropriate degree in orbits of tuples of
subsets whenever we hide from view a small portion of the
ambient space at each stage of the dynamics. Thus the recognition of positive entropy
or positive sequence entropy becomes a purely combinatorial issue, with 
the measure being relegated to the role of observational control device.
This way of counting sets appears in the global entropy formulas of Katok 
for metrizable topological $\Zb$-systems with an ergodic invariant measure \cite{Katok}, 
which rely on the Shannon-McMillan-Breiman
theorem for the uniformization of entropy measurement. Here we avoid
the Shannon-McMillan-Breiman theorem in our focus on local entropy production 
and its relation to independence for arbitrary systems. 
What is particularly important at the technical level is that we be able to make observations 
over finite sets of group elements in a nonuniform manner (see Subsection~\ref{SS-density}), 
as this will permit us to work with $L^2$ perturbations and thereby establish the link with
Voiculescu's approximation entropy. We will thus be developing
probabilistic arguments that will render the theory rather different from the
topological one, despite the obvious analogies in the statements of the main results,
although we will make critical use of the key combinatorial lemma from \cite{Ind}.

Our basic framework will be that of a discrete group acting on a
compact Hausdorff space with an invariant Borel
probability measure, with the structural assumption of amenability on the group in the
context of entropy. With a couple of exceptions, our results do not
require any restrictions of metrizability on the space or countability on the group. 
In analogy with topological IE-tuples and IN-tuples \cite{Ind},
we introduce the notions of measure IE-tuple (in the entropy context) and
measure IN-tuple (in the sequence entropy context) as tuples of points
in the space such that the orbit of every tuple of neigbourhoods of the respective
points exhibits independence with fixed density on certain finite subsets.
For IE-tuples these finite subsets will be required to be approximately invariant 
in the sense of the F{\o}lner characterization of amenability, while for IN-tuples we
will demand that they can be taken to be arbitrarily large.

Our main application of measure IE-tuples will be the derivation of a series of local
descriptions of the Pinsker $\sigma$-algebra (or maximal zero entropy factor)
in terms of combinatorial independence,
$\ell_1$ geometry, and Voiculescu's c.p.\ (completely positive) approximation entropy
(Theorem~\ref{T-Pinsker}).
These local descriptions are formulated as conditions on an $L^\infty$ function $f$
which are equivalent to the containment of $f$ in the Pinsker von Neumann algebra,
i.e., the von Neumann subalgebra corresponding to the Pinsker $\sigma$-algebra.
These conditions include:
\begin{enumerate}
\item there exist $\lambda\geq 1$ and $d>0$ such that
every $L^2$ perturbation of the orbit of $f$ exhibits 
$\lambda$-equivalence to the standard basis of $\ell_1$ over subsets 
of F{\o}lner sets with density at least $d$,

\item the local c.p.\ approximation entropy with respect to $f$ is positive.
\end{enumerate}
If the action is ergodic we can add:
\begin{enumerate}
\item[(3)] every $L^2$ perturbation of the orbit of $f$ contains a
subset of positive asymptotic density which is equivalent to the standard basis of
$\ell_1$.
\end{enumerate}
In the case that $f$ is continuous we can add:
\begin{enumerate}
\item[(4)] $f$ separates a measure IE-pair.
\end{enumerate}
This provides new geometric insight into the phenomenon of positive
c.p.\ approximation entropy, in parallel to what was done in the
topological setting for Voiculescu-Brown approximation entropy in \cite{EID,DEBS}.
In fact the only way to establish positive c.p.\ approximation entropy
until now has been by means of a comparison with Connes-Narnhofer-Thirring
entropy, whose definition is based on Abelian models (see Proposition~3.6 in \cite{Voi}).
We also do not require the Shannon-McMillan-Breiman theorem, which
factors crucially into Voiculescu's proof for $^*$-automorphisms in
the separable commutative ergodic case that c.p.\ approximation entropy
coincides with the underlying measure entropy \cite[Cor.\ 3.8]{Voi}.
One consequence of the characterization of elements in the Pinsker von Neumann algebra
given by condition (1) is a linear-geometric explanation for the well-known disjointness 
between zero entropy systems and systems with completely positive entropy,
as discussed at the end of Section~\ref{S-Pinsker}.  

The notion of measure entropy tuple
was introduced in \cite{BHMMR} in the pair case and in \cite{LVRA} in general
and has been a key tool in the local study of both measure entropy and 
topological entropy for $\Zb$-systems (see Section~19 of \cite{ETJ}).
We show in Theorem~\ref{T-IE E} that nondiagonal measure IE-tuples are
the same as measure entropy tuples. The argument depends in part on a
theorem of Huang and Ye for $\Zb$-systems from \cite{LVRA}, whose proof
involves taking powers of the generating automorphism and thus does not
extend as is to actions of amenable groups. For more general systems we
reduce to Huang and Ye's result by applying the orbit equivalence technique of
Rudolph and Weiss \cite{EMAGA}.
We point this out in particular because, with the
exception of the product formula of Theorem~\ref{T-IE product} 
and the characterizations of completely positive entropy in Theorem~\ref{T-cpe}, our 
study of measure IE-tuples and their relation to the topological theory
does not otherwise rely on orbit equivalence or any special treatment of the
integer action case, in contrast to what the measure 
entropy tuple approach in its present $\Zb$-system form seems to demand
(see \cite{mu,LVRA}). It is worth emphasizing however that we do need the relation 
with measure entropy tuples to establish the product formula for measure IE-tuples
(Theorem~\ref{T-IE product}), while the corresponding product formula
for topological IE-tuples as established in Theorem~3.15 of \cite{Ind} completely 
avoids the entropy tuple perspective, which would only serve to complicate matters
(compare the proof of the entropy pair product formula
for topological $\Zb$-systems in \cite{mu}). 
We also show (without the use of orbit equivalence) that the set of topological IE-tuples 
is the closure of the union of the sets 
of measure IE-tuples over all invariant Borel probability measures 
(Theorem~\ref{T-mu-IE IE closure}), and furthermore that when the space is metrizable
there exists an invariant Borel probability measure such that the sets of measure 
IE-tuples and topological IE-tuples coincide (Theorem~\ref{T-mu-IE IE}). In the 
$\Zb$-system setting, the latter result for entropy pairs was established in \cite{BGH} 
and more generally for entropy tuples in \cite{LVRA}.

One of the major advantages of the combinatorial viewpoint is the universal nature
of its application to entropy and independence density problems, as was demonstrated
in the topological-dynamical domain in \cite{Ind}. This means that many of 
the methods we develop for the study of measure IE-tuples apply equally well
to the sequence entropy context of measure IN-tuples. Accordingly, using measure IN-tuples
we are able to establish various local descriptions of the maximal null von Neumann algebra,
i.e., the sequence entropy analogue of the Pinsker von Neumann algebra (Theorem~\ref{T-null}).
We thus have the following types of conditions on a $L^\infty$ function $f$ characterizing
its containment in the maximal null von Neumann algebra:
\begin{enumerate}
\item there exist $\lambda\geq 1$ and $d>0$ such that every $L^2$ perturbation of the orbit 
of $f$ contains arbitrarily large finite subsets possessing subsets of density at least $d$ 
which are $\lambda$-equivalent to the standard basis
of $\ell_1$ in the corresponding dimension,

\item the local sequence
c.p.\ approximation entropy with respect to $f$ is positive for some sequence,
\end{enumerate}
and, in the case that $f$ is continuous,
\begin{enumerate}
\item[(3)] $f$ separates a measure IN-pair.
\end{enumerate}
Here, however, additional equivalent conditions
arise that have no counterpart on the entropy side, such as:
\begin{enumerate}
\item[(4)] every $L^2$ perturbation of the orbit of $f$ contains an infinite
subset which is equivalent to the standard basis of $\ell_1$, and

\item[(5)] every $L^2$ perturbation of the orbit of $f$ contains arbitrarily large finite
subsets which are $\lambda$-equivalent to the standard basis of $\ell_1$ for some
$\lambda > 0$.
\end{enumerate}
The presence of such conditions reflects the fact that there is a
strong dichotomy between nullness and nonnullness, which registers as
compactness vs.\ noncompactness for orbit closures in $L^2$ and is thus tied to weak mixing
and the issue of finite-dimensionality for group subrepresentations. In its probabilistic
manisfestation this dichotomy underlies Furstenberg's ergodic-theoretic approach
to Szemer{\'e}di's theorem \cite{Fur} and fits within a broader mathematical theme of
structure vs.\ randomness as discussed by Tao in \cite{Tao}.
Notice that the appearance of condition (4) indicates that the distinction
between tameness and nullness in topological dynamics collapses in the measurable setting.
In parallel with measure IE-tuples, it turns out (Theorem~\ref{T-IN SE}) that
nondiagonal measure IN-tuples are the same as measure sequence
entropy tuples as introduced in \cite{HMY}, which leads in particular to a simple product 
formula (Theorem~\ref{T-IN product}).

The main body of the paper is divided into four sections.
Section~\ref{S-IE} consists of four subsections. The first discusses measure independence
density for tuples of subsets, while in the second we define measure IE-tuples and establish
several basic properties. In the third subsection we address the problem of
realizing IE-tuples as measure IE-tuples. The fourth subsection contains the proof
that nondiagonal measure IE-tuples are the same as measure entropy tuples
and includes the product formula for measure IE-tuples.
Section~\ref{S-Pinsker} furnishes the local characterizations of the
Pinsker von Neumann algebra. In Section~\ref{S-IN} we define measure IN-tuples, record
their basic properties, show that nondiagonal measure IN-tuples are the same as
sequence measure entropy tuples, and derive the measure IN-tuple product formula. 
Finally, in Section~\ref{S-null} we establish
the local characterizations of the maximal null von Neumann algebra.

We now describe some of the basic concepts and notation used in the paper.
A collection $\{ (A_{i,1} , \dots , A_{i,k} ) : i\in I \}$ of $k$-tuples of subsets
of a given set is said to be {\it independent} if
$\bigcap_{i\in J} A_{i,\sigma (i)} = \emptyset$ for every finite set $J\subseteq I$
and $\sigma\in \{ 1,\dots ,k\}^J$.
The following definition captures a relativized version of this idea of combinatorial
independence in a group action context and forms the basis for our analysis of
measure-preserving dynamics. The relativized form is not necessary for topological dynamics
(cf.\ Definition~2.1 of \cite{Ind}) but becomes crucial in the measure-preserving case,
where we will need to consider independence relative to subsets of nearly full measure.

\begin{definition}\label{D-indset}
Let $G$ be a group acting on a set $X$. Let $\oA = (A_1 ,\dots ,A_k )$ be a
tuple of subsets of $X$. Let $D$ be a map from $G$ to the power set $2^X$
of $X$, with the image of $s\in G$ written as $D_s$. We say that a set $J\subseteq G$ is
an {\em independence set for $\oA$ relative to $D$} if for every nonempty
finite subset $F\subseteq J$ and map $\sigma : F\to \{ 1,\dots ,k \}$ we have
$\bigcap_{s\in F} (D_s \cap s^{-1} A_{\sigma (s)} ) \neq\emptyset$.
For a subset $D$ of $X$, we say that $J$ is an
{\em independence set for $\oA$ relative to $D$} if for every nonempty
finite subset $F\subseteq J$ and map $\sigma : F\to \{ 1,\dots ,k \}$ we have
$D\cap \bigcap_{s\in F} s^{-1} A_{\sigma (s)} \neq\emptyset$, i.e., if $J$
is an independence set for $\oA$ relative to the map $G\to 2^X$
with constant value $D$.
\end{definition}

By a {\it topological dynamical system} we mean a pair $(X,G)$ where $X$ is a
compact Hausdorff space and $G$ is a discrete group acting on $X$ by homeomorphisms.
We will also speak of a {\it topological $G$-system}.
In this context we will always use $\sB$ to denote the Borel $\sigma$-algebra of $X$.
%$G\times X\to X$ written $(s,x)\mapsto sx$.
Given a $G$-invariant Borel probability measure $\mu$ on $X$, we will invariably
write $\alpha$ for the induced action of $G$ on $\LX$ given by
$\alpha_s (f)(x) = f(s^{-1} x)$ for all $s\in G$, $f\in\LX$, and $x\in X$.
Given another topological $G$-system $(Y,G)$, a continuous surjective
$G$-equivariant map $X\to Y$ will be called a {\em topological $G$-factor map}.
In this situation we will regard $C(Y)$ as a unital $C^*$-subalgebra of $C(X)$.

By a {\it measure-preserving dynamical system} we mean a
quadruple $(X,\sX , \mu , G)$ where $(X,\sX , \mu )$ is a probability space and $G$ is a
discrete group acting on $(X,\sX , \mu )$ by $\mu$-preserving bimeasurable
transformations. 
%We will sometimes simply describe such a structure as an action of $G$ on $(X,\sX , \mu )$. 
The expression {\it measure-preserving $G$-system} will also be used.
%We will also use the expression {\it Lebesgue system} when $(X,\sX , \mu )$ is a Lebesgue space.
The action of $G$ is said to be {\it free} if for
every $s\in G\setminus \{ e \}$ the fixed-point set $\{ x\in X : sx = x \}$ has
measure zero. A {\it topological model} for $(X,\sX , \mu , G)$ is a 
measure-preserving $G$-system $(Y,\sY , \nu , G)$ isomorphic to $(X,\sX , \mu , G)$
such that $(Y,G)$ is a topological dynamical system.

We will actually work for the most part with an
invariant Borel probability measure for a topological dynamical system instead of an
abstract measure-preserving dynamical system, since the local study of independence properties
requires the specification of a topological model and such a specification entails no
essential loss of generality from the measure-theoretic viewpoint. So our basic
setting will consist of $(X,G)$ along with a $G$-invariant Borel probability measure $\mu$.
In Sections~\ref{S-IE} and \ref{S-Pinsker} we will also suppose $G$ to be amenable,
as the entropy context naturally requires.
%We write $\cB_X$ for the Borel $\sigma$-algebra of $X$.

%The context of our discussions on independence and entropy in Sections~\ref{S-IE} and
%\ref{S-Pinsker} will naturally require the acting group $G$ to be amenable.
For a finite $K\subseteq G$ and $\delta>0$ we write $M(K, \delta)$ for
the set of all nonempty finite subsets $F$ of $G$ which are $(K, \delta)$-invariant
in the sense that
\[ |\{s\in F: Ks\subseteq F\}|\ge (1-\delta )|F| . \]
The F{\o}lner characterization of amenability asserts that $M(K, \delta)$ is nonempty
for every finite set $K\subseteq G$ and $\delta > 0$. Given a real-valued function $\varphi$
on the finite subsets of $G$ we define the {\it limit supremum} 
and {\it limit infimum of $\varphi (F)/|F|$
as $F$ becomes more and more invariant} by
\[ \lim_{(K,\delta )} \sup_{F\in M(K,\delta )} \frac{\varphi (F)}{|F|}
\hspace*{7mm} \text{and} \hspace*{7mm}
\lim_{(K,\delta )} \inf_{F\in M(K,\delta )} \frac{\varphi (F)}{|F|} \]
respectively, where the net is constructed by stipulating that
$(K,\delta ) \succ (K' , \delta' )$ if $K\supseteq K'$ and $\delta\leq\delta'$.
These limits coincide under the following conditions:
\begin{enumerate}
\item $0\le \varphi(A)<+\infty$ and $\varphi(\emptyset)=0$,

\item $\varphi(A)\le \varphi(B)$ whenever $A\subseteq B$,

\item $\varphi(As)=\varphi(A)$ for all finite $A\subseteq G$ and
$s\in G$,

\item $\varphi(A\cup B)\le \varphi(A)+\varphi(B)$ if $A\cap
B=\emptyset$.
\end{enumerate}
See Section~6 of \cite{LW} and the last part of Section~3 in \cite{Ind}.
These conditions hold in the definition of measure entropy, 
which we recall next.

The entropy of a finite measurable partition $\cP$ of a probability space
$(X,\sX , \mu )$ is defined by $\Hmeas (\cP ) = \sum_{p\in\cP} - \mu (P) \ln \mu (P)$
(sometimes we write $\Hmeas_\mu (\cP )$ for precision).
Let $(X,\sX , \mu ,G)$ be a measure-preserving dynamical system. For a finite set
$F\subseteq G$, we abbreviate the join $\bigvee_{s\in F} s^{-1} \cP$ to $\cP^F$.
When $G$ is amenable, we write
$h_\mu (\cP )$ (or sometimes $h_\mu (X,\cP )$) for the limit of
$\frac{1}{|F|} \Hmeas (\cP^F )$ as $F$ becomes more and
more invariant, and we define the measure entropy $h_\mu (X)$ to be the supremum of
$h_\mu (\cP)$ over all finite Borel partitions $\cP$ of $X$. For general $G$,
given a sequence $\mathfrak{s}=\{s_j\}_{j\in \Nb}$ in $G$ we set
$h_\mu (\cP ; \fs ) = \limsup_{n\to\infty} \frac1n \Hmeas (\bigvee_{i=1}^n s_i^{-1} \cP)$
and define the measure sequence entropy $h_\mu (X ; \fs )$ to be the supremum of
$h_\mu (\cP ; \fs )$ over all finite measurable partitions $\cP$. The system
is said to be {\it null} if $h_\mu (X ; \fs ) = 0$ for all sequences $\fs$ in $G$.

The conditional entropy of a finite measurable partition $\cP = \{ P_1 , \dots , P_n \}$
with respect to a $\sigma$-subalgebra $\sA \subseteq\sX$ is defined by
\[ \Hmeas (\cP | \sA ) = \int I^\sA (\cP )(x)\, d\mu (x) \]
where $I^\sA (\cP )(x) = - \sum_{i=1}^n \boldsymbol{1}_{P_i} (x) \ln \mu (P_i | \sA ) (x)$
is the conditional information function.
For references on entropy see \cite{ETJ,Wal,OW}.

A unitary representation $\pi : G\to\cB (\cH )$ of a discrete group $G$ is
said to be {\it weakly mixing} if for all $\xi , \zeta\in\cH$ the function
$f_{\xi , \zeta } (s) = \langle \pi (s)\xi , \zeta \rangle$ on $G$ satisfies
$\mathfrak{m} (|f_{\xi , \zeta} |) = 0$, where $\mathfrak{m}$ is the unique invariant
mean on the space of weakly almost periodic bounded functions on $G$.
A subset $J$ of $G$ is {\it syndetic} if there is a finite set $F\subseteq G$
such that $FJ = G$ and {\it thickly syndetic} if for every finite set $F\subseteq G$
the set $\bigcap_{s\in F} sJ$ is syndetic.
Weak mixing is equivalent to each of the following conditions:
\begin{enumerate}
\item $\pi$ has no nonzero finite-dimensional subrepresentations,

\item for every finite set $F\subseteq\cH$ and $\varepsilon > 0$ there exists
an $s\in G$ such that $| \langle \pi (s)\xi , \zeta \rangle | < \varepsilon$ for
all $\xi , \zeta\in F$,

\item for all $\xi , \zeta \in\cH$ and $\varepsilon > 0$ the set of all
$s\in G$ such that $| \langle \pi (s)\xi , \zeta \rangle | < \varepsilon$
is thickly syndetic.
\end{enumerate}
We say that a measure-preserving dynamical system
$(X,\sX , \mu , G)$ is {\it weakly mixing} if the associated unitary representation
of $G$ on $L^2 (X,\mu ) \ominus \Cb\boldsymbol{1}$ is weakly mixing.
%where $\xi$ is the vector represented by the function with constant value $1$.
For references on weak mixing see \cite{BerRos,ETJ}.

For a probability space $(X , \sX , \mu )$ we write $\| \cdot \|_\mu$ for the
corresponding Hilbert space norm on elements of $L^\infty (X,\mu )$, i.e.,
$\| f \|_\mu = \mu (|f|^2 )^{1/2}$.

After this paper was completed we received a preprint by Huang, Ye, and Zhang \cite{LETCDAGA}
which uses orbit equivalence to establish a local variational principle 
for measure-preserving actions of countable discrete amenable groups on compact metrizable
spaces. For such systems they provide an entropy tuple variational relation 
(cf.\ Subsection~\ref{SS-realize} herein) and a positive answer to our Question~\ref{Q-open cover}.
They also obtained what appears here as Lemma~\ref{L-positive inf} 
\cite[Thm.\ 5.11]{LETCDAGA}.
\medskip

\noindent{\it Acknowledgements.} The first author was partially supported by NSF 
grant DMS-0600907. He is grateful to Bill Johnson and Gideon Schechtman for seminal
discussions and in particular for indicating the relevance of the Sauer-Perles-Shelah
lemma to the types of perturbation problems considered in the paper.

%%%%%%%%%%%%%%%%%%%%%%%%%%%%%%%%%%%%%%%%%%%%%%%%%%%%%%%%%%%%%%%%%%%%%%%%%%%%%%%%%%%%%%%%

\section{Measure IE-tuples}\label{S-IE}

Throughout this section $(X,G)$ is a topological dynamical system with $G$ amenable
and $\mu$ is a $G$-invariant Borel probability measure on $X$.

\subsection{Measure independence density for tuples of subsets}\label{SS-density}

Our concept of measure IE-tuple will be based on the following definition of
independence density for tuples of subsets, which is formulated in terms of the
notion of independence set from Definition~\ref{D-indset}.
For $\delta > 0$ denote by $\mathscr{B} (\mu , \delta )$ the collection of all
Borel subsets $D$ of $X$ such that $\mu (D) \geq 1 - \delta$, and
by $\mathscr{B}' (\mu , \delta )$ the collection of all maps
$D : G\to\mathscr{B} (X)$ such that $\inf_{s\in G} \mu (D_s ) \geq 1-\delta$.
Let $\oA = (A_1 , \dots , A_k )$ be a tuple of subsets of $X$ and let $\delta > 0$.
For every finite subset $F$ of $G$ we define
\begin{align*}
\varphi_{\oA ,\delta} (F) &= \min_{D\in \mathscr{B} (\mu , \delta )}
\max \big\{ |F\cap J| : J\text{ is an independence set for }\oA
\text{ relative to } D \big\} , \\
\varphi'_{\oA ,\delta} (F) &= \min_{D\in \mathscr{B}' (\mu , \delta )}
\max \big\{ |F\cap J| : J\text{ is an independence set for }\oA
\text{ relative to } D \big\} .
\end{align*}
Since the action of $G$ on $X$ is $\mu$-preserving, we have
$\varphi_{\oA ,\delta} (Fs) = \varphi_{\oA ,\delta} (F)$ and
$\varphi'_{\oA ,\delta} (Fs) = \varphi'_{\oA ,\delta} (F)$ for all finite sets
$F\subseteq G$ and $s\in G$. However, neither $\varphi_{\oA ,\delta}$ nor
$\varphi'_{\oA ,\delta}$ satisfy
the subadditivity condition in Proposition~3.22 of \cite{Ind}, so that the
limit of $\frac{1}{|F|} \varphi_{\oA ,\delta} (F)$ or
$\frac{1}{|F|} \varphi'_{\oA ,\delta} (F)$ as $F$ becomes more and more invariant
might not exist. We define $\upind_\mu (\oA , \delta )$ to be the limit supremum
of $\frac{1}{|F|} \varphi_{\oA ,\delta} (F)$ as $F$ becomes more and more invariant,
and $\lwind_\mu (\oA , \delta )$ to be the corresponding limit infimum.
Similarly, we define $\upind{}'_\mu (\oA , \delta )$ to be the limit supremum
of $\frac{1}{|F|} \varphi'_{\oA ,\delta} (F)$ as $F$ becomes more and more invariant,
and $\lwind{}'_\mu (\oA , \delta )$ to be the corresponding limit infimum.
Note that $\upind{}'_\mu (\oA , \delta ) \leq \upind_\mu (\oA , \delta )$ and
$\lwind{}'_\mu (\oA , \delta ) \leq \lwind_\mu (\oA , \delta )$.

%From the viewpoint of the Shannon-McMillan-Breiman theorem, the quantity
%$\lwind_\mu (\oA , \delta )$ is related to pointwise convergence while
%its primed version is related to $L^1$ convergence. See the proof of Lemma~\ref{L-two}.

\begin{definition}
We set
\[ \upind_\mu (\oA ) = \sup_{\delta > 0} \upind_\mu (\oA , \delta )
\hspace*{5mm}\text{and}\hspace*{5mm}
\lwind_\mu (\oA ) = \sup_{\delta > 0} \lwind_\mu (\oA , \delta ) \]
and refer to
these quantities respectively as the {\em upper $\mu$-independence density} and
{\em lower $\mu$-independence density} of $\oA$.
\end{definition}

In order to relate independence and c.p.\ approximation entropy in the local description
of the Pinsker von Neumann algebra (Theorem~\ref{T-Pinsker}), we will need
to know that in the definitions of $\upind_\mu (\oA )$ and $\lwind_\mu (\oA )$ the
quantitites $\upind{}_\mu (\oA , \delta )$ and $\lwind{}_\mu (\oA , \delta )$
can be replaced by their primed versions, i.e., if the subsets of $X$ of measure at
least $1-\delta$ relative to which independence is gauged are not required to be
uniform over $G$, then the resulting versions of upper and lower independence
density agree with the original ones. This is the content of
Proposition~\ref{P-alternative}, which we now aim to establish.

\begin{lemma}\label{L-density}
Let $k\geq 2$. Then for every $\lambda$ in the interval $(\log_k (k-1),1)$
there are\linebreak $a,b > 0$ such that
for every $n\in\Nb$ and $S\subseteq \{ 0,1,\dots ,k \}^{\{ 1,\dots ,n\}}$ with
$|S| \geq k^{\lambda n}$\linebreak and $\max_{\sigma\in S} |\sigma^{-1} (0)| \leq bn$
there exists an $I\subseteq \{ 1,\dots ,n \}$ with $|I|\geq an$ and\linebreak
$S|_{I} \supseteq \{ 1,\dots ,k \}^{\{ 1,\dots ,n\}}$. Moreover as
$\lambda\nearrow 1$ we may choose $a\nearrow 1$.
\end{lemma}

\begin{proof}
Let $\lambda\in (\log_k (k-1),1)$. Set
$f(\lambda ) = (1-\lambda )(\lambda - \log_k (k-1))$. Then the quantity
$\lambda - f(\lambda )$ lies in the interval $(\log_k (k-1),1)$ and
tends to one as $\lambda\nearrow 1$.
By Karpovsky and Milman's generalization of the Sauer-Perles-Shelah lemma
\cite{Sauer,Shelah,KM} there is an $a\in (0,1)$ such that for every $n\in\Nb$ and
$S\subseteq \{ 1,\dots ,k \}^{\{ 1,\dots ,n\}}$
with $|S| \geq k^{(\lambda - f(\lambda ))n}$ there exists an
$I\subseteq \{ 1,\dots ,n \}$ with $|I|\geq a^{\frac12} n$ and
$S|_{I} = \{ 1,\dots ,k \}^I$, and we may choose $a\nearrow 1$ as $\lambda\nearrow 1$.
By Stirling's formula there is a $c\in (0,1/2)$ such that
$cn\binom{n}{cn} \leq k^{f(\lambda )n}$ for all $n\in\Nb$.
Set $b = \min (c,1-a^{\frac12} )$.
Now suppose we are given an $n\in\Nb$ and
$S\subseteq \{ 0,1,\dots ,k \}^{\{ 1,\dots ,n\}}$
with $|S| \geq k^{\lambda n}$ and $\max_{\sigma\in S} |\sigma^{-1} (0)| \leq bn$.
Then we can find a $J\subseteq \{ 1,\dots ,n \}$ with
$|J| \geq (1-b)n \geq a^{\frac12} n$ such that the cardinality of the set
$\big\{ \sigma\in S : \sigma^{-1} \{ 1,\dots ,k \} = J \big\}$ is at least
$\frac{|S|}{cn\binom{n}{cn}} \geq k^{(\lambda - f(\lambda )) n}$.
Consequently there exists an
$I\subseteq J$ with $|I|\geq a^{\frac12} |J|\geq an$ and
$S|_{I} \supseteq \{ 1,\dots ,k \}^I$, as desired.
\end{proof}

\begin{lemma}\label{L-prime}
For every $\delta > 0$ there is a $\delta' > 0$ such that
$\frac{1}{|F|}\varphi'_{\oA ,\delta'} (F) \geq \frac{1}{|F|}\varphi_{\oA ,\delta} (F) -
\delta$ for all finite sets $F\subseteq G$.
\end{lemma}

\begin{proof}
Let $\delta > 0$, and $d$ be a positive number to be further specified
below as a function of $\delta$. Set $\delta' = \delta d$.
Let $F$ be a finite subset of $G$. To
establish the inequality in the proposition statement we may assume
that $\varphi_{\oA ,\delta} (F) \geq \delta |F|$.
Let $D$ be an element of $\mathscr{B}' (\mu , \delta' )$ such that
$\varphi'_{\oA ,\delta'} (F)$ is equal to
the maximum of $|F\cap J|$ over all independence sets $J$ for $\oA$
relative to $D$. Put
\[ E = \{ x\in X : | \{ s\in F : x\notin D_s \} | \leq d|F| \} . \]
Since $\mu (D_s ) \geq 1 - \delta'$ for each $s\in F$ we have
\[ \mu (E^\comp ) d|F| \leq \sum_{s\in F} \mu (D_s^\comp ) \leq |F|\delta' \]
and so $\mu (E) \geq 1 - \frac{\delta'}{d} = 1 - \delta$, that is,
$E\in \mathscr{B} (\mu , \delta )$. Hence
there exists an $I\subseteq F$ with $|I| = \varphi_{\oA ,\delta} (F)$ which is an
independence set for
$\oA$ relative to $E$. For each $\sigma\in \{ 1,\dots ,k \}^I$ we can find by the
definition of $E$ a set $I_\sigma \subseteq I$ with
$|I\setminus I_\sigma | \leq d|F|$ such that
$\bigcap_{s\in I_\sigma} (D_s \cap s^{-1} A_{\sigma (s)}) \neq\emptyset$,
%There then exists a $J\subseteq I$ with $|I\setminus J | \leq d|F|$ such that
%the number of $\sigma\in \{ 1,\dots ,k \}^I$ for which $I_\sigma = J$ is at
%least $k^{|I|} / \binom{|F|}{d|F|}$.
and we define $\rho_\sigma \in \{ 0,1,\dots ,k \}^I$ by
\[ \rho_\sigma (s) = \left\{ \begin{array}{l@{\hspace*{8mm}}l}
\sigma (s) &
\text{if } s\in I_\sigma , \\
0 &
\text{if } s\notin I_\sigma .
\end{array} \right. \]
Since for every $\rho\in \{ 0,1,\dots ,k \}^I$ the number of
$\sigma\in \{ 1,\dots ,k \}^I$ for which $\rho_\sigma = \rho$ is at most
$k^{d|F|}$, the set
$\cS = \big\{ \rho_\sigma : \sigma \in \{ 1,\dots ,k \}^I \big\}$ has
cardinality at least $k^{|I|} / k^{d|F|} \geq k^{(1 - d/\delta )|I|}$.
It follows by Lemma~\ref{L-density} that if $d$ is small enough as
a function of $\delta$ then there exists a
$J\subseteq I$ with $|J|\geq (1-\delta ) |I|$ such that
$\cS |_J \supseteq \{ 1,\dots ,k \}^J$. Such a $J$ is an independence
set for $\oA$ relative to $D$, and so we conclude that
$\frac{1}{|F|} \varphi'_{\oA ,\delta'} (F) \geq
\frac{1}{|F|} (1-\delta )\varphi_{\oA ,\delta} (F) \geq
\frac{1}{|F|} \varphi_{\oA ,\delta} (F) - \delta$. Since our choice of $\delta'$
does not depend on $F$ this completes the proof.
\end{proof}

It follows from Lemma~\ref{L-prime} that for every $\delta > 0$ there is a
$\delta' > 0$ such that
$\upind{}'_\mu (\oA ,\delta' ) \geq \upind_\mu (\oA , \delta ) - \delta$ and
$\lwind{}'_\mu (\oA ,\delta' ) \geq \lwind_\mu (\oA , \delta ) - \delta$.
We thus obtain the following alternative means of expressing
upper and lower $\mu$-independence density.

\begin{proposition}\label{P-alternative}
We have $\upind_\mu (\oA ) = \sup_{\delta > 0} \upind{}'_\mu (\oA , \delta )$
and $\lwind_\mu (\oA ) = \sup_{\delta > 0} \lwind{}'_\mu (\oA , \delta )$.
\end{proposition}

%%%%%%%%%%%%%%%%%%%%%%%%%%%%%%%%%%%%%%%%%%%%%%%%%%%%%%%%%%%%%%%%%%%%%%%%%%%%%

\subsection{Definition and basic properties of measure IE-tuples}

In \cite{Ind} we defined a tuple $\ox = (x_1 , \dots, x_k )\in X^k$
to be an IE-tuple (or an IE-pair in the case $k=2$) if for every product
neighbourhood $U_1 \times\cdots\times U_k$ of $\ox$ the $G$-orbit of the tuple
$(U_1, \dots, U_k)$ has an independent subcollection of positive density. The following
is the measure-theoretic analogue.

\begin{definition}
We call a tuple $\ox = (x_1 , \dots , x_k )\in X^k$ a {\em $\mu$-IE-tuple}
(or {\em $\mu$-IE-pair} in the case $k=2$) if
for every product neighbourhood $U_1 \times\cdots\times U_k$ of $\ox$ the tuple
$(U_1 , \dots , U_k )$ has positive upper $\mu$-independence density. We denote the set
of $\mu$-IE-tuples of length $k$ by $\IE^\mu_k (X)$.
\end{definition}

Evidently every $\mu$-IE-tuple is an IE-tuple. The problem
of realizing IE-tuples as $\mu$-IE-tuples for some $\mu$ will be addressed in
Subsection~\ref{SS-realize}.

We proceed now with a series of lemmas which will enable us to establish some
properties of $\mu$-IE-tuples as recorded in Proposition~\ref{P-IE basic}.

\begin{lemma}\label{L-split}
Let $\oA = (A_1 , \dots , A_k )$ be a tuple of subsets of $X$ which has
positive upper $\mu$-independence density. Suppose that $A_1 = A_{1,1} \cup A_{1,2}$.
Then at least one of the tuples $\oA_1 = (A_{1,1} , A_2 , \dots , A_k )$ and
$\oA_2 = (A_{1,2} , A_2 , \dots , A_k )$ has positive upper $\mu$-independence density.
\end{lemma}

\begin{proof}
By Lemma~3.6 of \cite{Ind} there is a constant
$c>0$ depending only on $k$ such that, for all $n\in\Nb$, if $S$ is a subset
of $(\{ (1,0),(1,1) \} \cup \{2,\dots ,k\} )^{\{ 1,\dots ,n\}}$ for which the
restriction $\Gamma_n |_S$ is bijective, where $\Gamma_n :
(\{ (1,0),(1,1) \} \cup \{2,\dots ,k\} )^{\{ 1,\dots ,n\}} \to
\{ 1,\dots ,k \}^{\{ 1,\dots ,n\}}$
converts the coordinate values $(1,0)$ and $(1,1)$ to $1$, then there is an
$I\subseteq \{ 1,\dots ,n \}$ with $|I|\geq cn$ and either
$S|_I \supseteq (\{ (1,0) \} \cup \{ 2,\dots ,k \} )^I$ or
$S|_I \supseteq (\{ (1,1) \} \cup \{ 2,\dots ,k \} )^I$.
Thus, given sets $D_1 , D_2 \subseteq X$,
any finite set $I\subseteq G$ which is an independence set for $\oA$ relative
to $D_1 \cap D_2$ has a
subset $J$ of cardinality at least $c|I|$  which is either an independence set for
$\oA_1$ relative to $D_1 \cap D_2$ (and hence relative to $D_1$) or an independence set for
$\oA_2$ relative to $D_1 \cap D_2$ (and hence relative to $D_2$). Given a $\delta > 0$,
we have $D_1 \cap D_2 \in\mathscr{B} (\mu , \delta )$
whenever $D_1 , D_2 \in\mathscr{B} (\mu , \delta /2)$ and
so we deduce that $\max \{ \upind_\mu (\oA_1 , \delta /2),\upind_\mu (\oA_2 , \delta /2) \}
\geq c\cdot \upind_\mu (\oA , \delta )$.
By hypothesis there is a $\delta > 0$ such that $\upind_\mu (\oA ,\delta ) > 0$, from
which we conclude that $\upind_\mu (\oA_j , \delta /2) > 0$ for at least one
$j\in \{ 0,1 \}$, yielding the proposition.
\end{proof}

\begin{lemma}\label{L-mind}
For every $d>0$ there exist $\delta >0$, $c>0$, and $M>0$ such
that if $F$ is a finite subset of $G$ with $|F|\ge M$, $D$ is in
$\mathscr{B}'(\mu, \delta)$, $\cP=\{P_1, P_2\}$ is a Borel
partition of $X$ with $\frac{\Hmeas (\cP^F)}{|F|}\ge d$,
%where $\cP^F:=\vee_{g\in F}g^{-1}\cP$,
and $A_1 \subseteq P_1$ and $A_2 \subseteq P_2$ are Borel sets
with $\mu(P_1\setminus A_1), \mu(P_2\setminus
A_2)<\delta$, then $(A_1, A_2)$ has a $\mu$-independence set
$I\subseteq F$ relative to $D$ with $|I|\ge c|F|$.
\end{lemma}

\begin{proof}
Let $d>0$. Given a finite set $F\subseteq G$, denote by $\cY$ the set of all
$Y\in \cP^F$ such that $\mu(Y)<e^{-\frac{d}{3}|F|}$ and by $\cZ$ the set of
all $Z\in \cP^F$ such that $\mu(Z)\geq e^{-\frac{d}{3}|F|}$.
Put $B=\bigcup\cY$. Since the function $f(x)=-x\ln x$
for $x\in [0, 1]$ is concave downward and has maximal value
$e^{-1}$, we have
\begin{align*}
\sum_{Y\in \cY} -\mu(Y)\ln \mu(Y)  &\le -\mu(B)\ln\frac{\mu(B)}{|\cY |} \\
&= \mu(B)\ln |\cY |-\mu(B) \ln \mu(B) \\
&\le (\mu(B)\cdot \ln 2)|F|+e^{-1}.
\end{align*}
We also have
\[ \sum_{Z \in \cZ} -\mu(Z)\ln
\mu(Z)\le \sum_{Z\in \cZ}\mu(Z) \ln e^{\frac{d}{3}|F|}\le
\frac{d}{3}|F|. \]
Thus
\begin{align*}
d \le \frac{\Hmeas (\cP^F)}{|F|}
&= \frac{\sum_{Y\in \cY} -\mu(Y)\ln \mu(Y) + \sum_{Z \in\cZ} -\mu(Z)\ln
\mu(Z)}{|F|}\\
&\le \frac{(\mu(B)\cdot \ln 2)|F|+e^{-1}+\frac{d}{3}|F|}{|F|} \\
&= \mu(B)\cdot \ln 2+\frac{d}{3}+\frac{e^{-1}}{|F|}.
\end{align*}
Choose an $M\ge \frac{3e^{-1}}{d(2-2\ln 2)}$ such that
$\frac{d}{3}e^{\frac{d}{6}M}\ge 1$.
We will suppose henceforth that $|F|\ge M$, in which case $\mu(B)\ge \frac{2}{3}d$.

%By Stirling's formula, there is a $b'>0$ (depending on
%$d$) such that for every nonempty finite set $K$ the number of
%subsets of $K$ with cardinality at most $b'|K|$ is at most
%$e^{\frac{d}{12}|K|}$.
By Lemma~\ref{L-density}, there are $b, c > 0$ (depending on $d$)
such that for every nonempty finite set $K$ and $S\subseteq \{
0,1,2 \}^{K}$ with $|S| \geq e^{\frac{d}{12}|K|}$ and
$\max_{\sigma\in S} |\sigma^{-1} (0) | \leq b |K|$ there exists an
$I\subseteq K$ with $|I| \geq c|K|$ and $S|_I \supseteq \{1,
2\}^I$. We may assume that $2^b\le e^{\frac{d}{12}}$.
%$b\le b'$.

Set $\delta=\frac{db}{9}$.
%Let $D\in \mathscr{B}' (\mu , \delta)$
%and let $A_1 \subseteq P_1$ and $A_2 \subseteq P_2$ be Borel sets
%with $\mu(P_1\setminus A_1), \mu(P_2\setminus A_2)<\delta$.
Then $\mu (X\setminus (D_s\cap s^{-1}(A_1\cup A_2)))\le 3\delta=
\frac{db}{3}$ for every $s\in G$. Set
\[ W = \big\{ x\in X : \big| \big\{ s\in F: x  \in
D_s \cap s^{-1} (A_1 \cup A_2 )\big\} \big| \geq (1-b)|F| \big\} , \]
which has measure at least
\[ 1 - \frac{1}{b |F|} \sum_{s\in F} \mu \big( X\setminus (D_s \cap
s^{-1} ( A_1 \cup A_2 ))\big) \geq 1 - \frac{1}{b|F|} \cdot |F| \frac{db}{3}
= 1-\frac{d}{3} . \]
Then $\mu (W\cap B) \geq \frac{d}{3}$. Thus the
set $\cY'$ of all $Y\in \cY$ for which $\mu (W\cap Y) > 0$ has
cardinality at least $\frac{d}{3}e^{\frac{d}{3}|F|}\ge
e^{\frac{d}{6}|F|}$. For each $Y\in \cY'$ pick an $x_Y\in W\cap Y$.
Define a map $\varphi: \cY' \to \{0, 1,2\}^{F}$ by
\[ \varphi (Y)(s) = \left\{ \begin{array}{l@{\hspace*{8mm}}l}
0 &
\text{if } x_Y \notin D_s\cap s^{-1}(A_1\cup A_2), \\
1 &
\text{if } x_Y \in D_s\cap s^{-1} A_1 , \\
2 & \text{if } x_Y \in D_s\cap s^{-1} A_2 ,
\end{array} \right. \]
for $Y\in \cY'$ and $s\in F$. If $\varphi (Y_1 )=\varphi (Y_2 )$, then $Y_1$
and $Y_2$ coincide on a subset of $F$ with cardinality at least
$(1-b)|F|$.
%\ge (1-b')|F|$.
Hence $|\varphi(\cY')|\ge
|\cY' |/2^{b|F|}\ge e^{\frac{d}{12}|F|}$. Therefore there
exists an $I\subseteq F$ such that $|I|\ge c|F|$ and
$\varphi(\cY' )|_I\supseteq \{1, 2\}^{I}$. Then $I$ is a
$\mu$-independence set for $(A_1, A_2)$ relative to $D$.
\end{proof}

\noindent We remark that the constants $\delta$, $c$, and $M$ specified
in the proof of Lemma~\ref{L-mind} do not depend on $(X,G)$ or $\mu$.

\begin{lemma}\label{L-pe to pd}
Let $\cP=\{P_1, P_2\}$ be a two-element Borel partition of $X$
such that $h_{\mu}(\cP)>0$. Then there
exists an $\varepsilon>0$ such that $\lwind_\mu (\oA )>0$ whenever
$\oA=(A_1, A_2)$ for Borel subsets $A_1\subseteq P_1$ and
$A_2\subseteq P_2$ with $\mu(P_1\setminus A_1), \mu (P_2\setminus
A_2)<\varepsilon$.
\end{lemma}

\begin{proof}
Apply Lemma~\ref{L-mind}.
\end{proof}

\begin{lemma}\label{L-supp}
Let $A$ be a Borel subset of $X$ with $\mu(A)>0$. Then there are
$d>0$ and $\delta>0$ such that for every finite subset
$F\subseteq G$ and $D\in \mathscr{B}(\mu, \delta)$
there is an $H\subseteq F$ with $|H|\ge d|F|$ and $D\cap
\big( \bigcap_{s\in H}s^{-1}A \big) \neq \emptyset$.
\end{lemma}

\begin{proof}
Choose a $d>0$ less than $\mu(A)$ and
set $E=\{x\in X: |\{g\in F:gx\in A\}|\geq d|F|\}$. Then
$(1-d)|F|1_{X\setminus E}\le \sum_{g\in F}1_{g^{-1}(X\setminus A)}$ so that
\[ (1-d)|F|(1-\mu(E))=\int (1-d)|F|1_{X\setminus E} \, d\mu
\leq \int \sum_{g\in F}1_{g^{-1}(X\setminus A)} \, d\mu = |F|(1-\mu(A)) \]
and hence $\mu(E)\geq 1-\frac{1-\mu(A)}{1-d}>0$. We
can thus take $\delta$ to be any strictly positive number less than
$1-\frac{1-\mu(A)}{1-d}$.
\end{proof}

In order to determine the behaviour of measure IE-tuples under taking factors
and to establish the main results of the next two subsections, 
we need to consider several auxiliary entropy
quantities. Let $\cU$ be a finite Borel cover of $X$. For a subset $D$ of $X$
denote by $N_D(\cU)$ the minimal number of members of $\cU$ needed to
cover $D$. For $\delta>0$ we set
$N_{\delta}(\cU) = \min_{D\in \mathscr{B} (\mu , \delta )} N_D(\cU)$
and write $\lwh_{\comb , \mu}(\cU, \delta)$ for the limit infimum of
$\frac{1}{|F|}\ln N_{\delta}(\cU^F)$ as $F$ becomes more and more invariant
and $\uph_{\comb , \mu}(\cU, \delta)$ for the limit supremum of
$\frac{1}{|F|} \ln N_{\delta}(\cU^F)$ as $F$ becomes more and more invariant.
We then define
\begin{align*}
\lwh_{\comb , \mu}(\cU) &= \sup_{\delta>0}\lwh_{\comb , \mu}(\cU, \delta), \\
\uph_{\comb , \mu}(\cU) &= \sup_{\delta>0}\uph_{\comb , \mu}(\cU, \delta) .
\end{align*}
The metric versions of $\lwh_{\comb , \mu}(\cU, \delta)$ and 
$\uph_{\comb , \mu}(\cU, \delta)$ in the ergodic $\Zb$-system case appear in the 
entropy formulas of Katok from \cite{Katok}.
Writing $\Hmeas (\cU)$ for the infimum of $\Hmeas (\cP)$ over all Borel
parititions $\cP$ of $X$ refining $\cU$, we define
$h_{\mu}^-(\cU)$ to be the limit of $\frac{1}{|F|}\Hmeas (\cU^F)$ as $F$ becomes
more and more invariant. Finally, we define $h_{\mu}^+(\cU)$ to be the
infimum of $h_{\mu}(\cP)$ over all Borel parititions $\cP$ of $X$ refining $\cU$.
The quantities $h_{\mu}^-(\cU)$ and $h_{\mu}^+(\cU)$ were introduced by
Romagnoli in the case $G=\Zb$ \cite{Roma}. We have the trivial inequalities
$\lwh_{\comb , \mu}(\cU)\le \uph_{\comb , \mu}(\cU)$ and $h_{\mu}^-(\cU)\leq h_{\mu}^+(\cU)$.
Huang, Ye, and Zhang observed in \cite{HYZ} that results in \cite{Interplay, HMRY, Roma}
can be combined to deduce that $h_{\mu}^-(\cU) = h_{\mu}^+(\cU)$ 
for all open covers $\cU$ when $X$ is metrizable and 
$G=\Zb$.

\begin{question}\label{Q-open cover}
Is it always the case that $h_{\mu}^-(\cU) = h_{\mu}^+(\cU)$ for an open cover
$\cU$?
\end{question}

%If the complements in $X$ of the members of the cover $\cU$ are
%pairwise disjoint and $\oA$ is a tuple consisting of these complements, then
%by \cite[Lemma 3.3]{Ind} and Lemma~\ref{L-converse} we have
%$\lwh_{\comb , \mu}(\cU)>0$ if and only if $\lwind_{\mu}(\oA)>0$ and
%$\uph_{\comb ,\mu}(\cU)>0$ if and only if $\upind_{\mu}(\oA)>0$.

The following fact was established by Romagnoli \cite[Eqn.\ (8)]{Roma}.

\begin{lemma}\label{L-Hlifting}
Let $\pi:X\to Y$ be a factor of $X$. Then
\[ H_{\mu}(\pi^{-1}\cU) = H_{\pi_*(\mu)}(\cU) \]
for every finite Borel cover $\cU$ of $Y$.
\end{lemma}

\noindent One direct consequence of Lemma~\ref{L-Hlifting} is the following,
which in the case $G=\Zb$ is recorded as Proposition~6 in \cite{Roma}.

\begin{lemma}\label{L-hlifting}
Let $\pi:X\to Y$ be a factor of $X$. Then
\[ h_{\mu}^-(\pi^{-1}\cU) = h_{\pi_*(\mu)}^-(\cU) \]
for every finite Borel cover $\cU$ of $Y$.
\end{lemma}

\begin{lemma}\label{L-local}
For a finite Borel cover $\cU$ of $X$ and $\delta > 0$ we have
\[ \delta\cdot \uph_{\comb , \mu}(\cU, \delta)\leq h_{\mu}^-(\cU)\leq
\lwh_{\comb , \mu}(\cU) . \]
\end{lemma}

\begin{proof} Let $\varepsilon>0$ and $\delta>0$. When a finite subset $F$ of
$G$ is sufficiently invariant, we have $ \frac{1}{|F|}H_{\mu}(\cU^F)\leq
h_{\mu}^-(\cU)+\varepsilon$. Then we can find a finite Borel
partition $\cP\succeq \cU^F$ with $\frac{1}{|F|}H_{\mu}(\cP)\leq
h_{\mu}^-(\cU)+2\varepsilon$. Consider the set $\cY$ consisting of
members of $\cP$ with $\mu$-measure at least
$e^{-|F|(h_{\mu}^-(\cU)+2\varepsilon)/\delta}$ and set $D=\bigcup \cY$.
Then $\mu(D^c)\leq \delta$. Thus $D\in \mathscr{B} (\mu ,\delta )$ and hence
$N_{\delta}(\cU^F)\leq |\cY|\leq e^{|F|(h_{\mu}^-(\cU)+2\varepsilon)/\delta}$.
Consequently, $\uph_{\comb , \mu}(\cU, \delta)\leq
(h_{\mu}^-(\cU)+2\varepsilon)/\delta$. Letting $\varepsilon\to 0$
we obtain $\delta\cdot \uph_{\comb , \mu}(\cU, \delta)\leq h_{\mu}^-(\cU)$.

For the second inequality, let $\varepsilon > 0$ and $\delta \in (0,e^{-1} )$.
Take a finite subset $F$ of $G$ sufficiently invariant so that
$\frac{1}{|F|}\ln N_{\delta}(\cU^F) < \lwh_{\comb , \mu}(\cU, \delta)+\varepsilon$.
Then we can find a $D\in\mathscr{B} (\mu, \delta)$ with
$\frac{1}{|F|}\ln N_D(\cU^F) < \lwh_{\comb , \mu}(\cU, \delta)+\varepsilon$.
Take a Borel partition $\cY$ of $D$ finer than the
restriction of $\cU^F$ to $D$ with cardinality $N_D(\cU^F)$
and a Borel partition $\cZ$ of $D^c$ finer than the
restriction of $\cU^F$ to $D^c$ with cardinality $N_{D^c}(\cU^F)$.
Since the function $x\mapsto -x\ln x$ is concave on $[0,1]$ and increasing on
$[0, e^{-1}]$ and decreasing on $[e^{-1},1]$, we have
\begin{align*}
-\sum_{P\in\cY}\mu(P)\ln \mu(P) &\leq -\mu(D)\ln \frac{\mu(D)}{|\cY |} \\
&\leq -(1-\delta)\ln (1-\delta)+\ln N_D(\cU^F) \\
&\leq -(1-\delta)\ln (1-\delta)+|F|(\lwh_{\comb , \mu}(\cU,\delta)+\varepsilon),
\end{align*}
and
\begin{align*}
-\sum_{P\in\cZ}\mu(P)\ln \mu(P) &\leq -\mu(D^c)\ln\frac{\mu(D^c)}{|\cZ |} \\
&\leq -\delta\ln \delta+\delta \ln N_{D^c}(\cU^F) \\
&\leq -\delta\ln \delta+\delta |F|\ln |\cU|.
\end{align*}
Thus $\frac{1}{|F|}H_{\mu}(\cU^F)\leq -(1-\delta)\ln
(1-\delta)-\delta\ln \delta+\lwh_{\comb , \mu}(\cU,
\delta)+\varepsilon+\delta\ln |\cU|$ and hence
\begin{align*}
h_{\mu}^-(\cU)\leq -(1-\delta)\ln (1-\delta)-\delta\ln
\delta+\lwh_{\comb , \mu}(\cU, \delta)+\varepsilon+\delta\ln |\cU|.
\end{align*}
Letting $\varepsilon\to 0$ and $\delta\to 0$ we get
$h_{\mu}^-(\cU)\leq \lwh_{\comb , \mu}(\cU)$.
\end{proof}

Let $k\geq 2$ and let $Z$ be a nonempty finite set. We write $\cW$ for
the cover of $\{0, 1, \dots , k\}^Z = \prod_{z\in Z} \{ 0, 1, \dots ,k \}$
consisting of subsets of the form $\prod_{z\in Z}\{i_z\}^\comp$, where
$1\leq i_z\leq k$ for each $z\in Z$. For a set
$S\subseteq \{0, 1,\dots , k\}^Z$ we denote by $F_S$ the
minimal number of sets in $\cW$ one needs to cover $S$. The
following lemma provides a converse to \cite[Lemma 3.3]{Ind}.

\begin{lemma}\label{L-converse}
Let $k\geq 2$. For every finite set $Z$ and $S\subseteq \{ 0, 1, \dots ,k \}^{Z}$,
if $S|_W \supseteq \{ 1, \dots , k \}^W$ for some nonempty set
$W\subseteq Z$, then $F_S \geq \big( \frac{k}{k-1}\big)^{|W|}$.
\end{lemma}

\begin{proof}
Replacing $S$ by $S|_W$ we may assume that $W=Z$. We
prove the assertion by induction on $|Z|$. The case $|Z|=1$ is
trivial. Suppose that the assertion holds for $|Z|=n$.
Consider the case $|Z|=n+1$. Take $z\in Z$ and set $Y=Z\setminus\{z\}$.
For each $1\leq j\leq k$ write $S_j$ for the set of all elements of $S$
taking value $j$ at $z$. Then $S_j |_Y\supseteq \{ 1, \dots,k\}^Y$, and so
$F_{S_j} \geq \big( \frac{k}{k-1} \big)^{|Y|}$. Now suppose that some
$\cV\subseteq\cW$ covers $S$. Write $\cV_j$ for the set of all
elements of $\cV$ that have nonempty intersection with $S_j$. Then $|\cV_j|\geq
F_{S_j} \geq \big( \frac{k}{k-1} \big)^{|Y|}$. Note that each element of $\cV$ is
contained in at most $k-1$ many of the sets $\cV_1, \dots, \cV_k$. Thus
$(k-1)|\cV |\geq \sum^k_{j=1}|\cV_j|\geq k\big( \frac{k}{k-1} \big)^{|Y|}$, and
hence $|\cV |\geq \big( \frac{k}{k-1} \big)^{|Z|}$, completing the induction.
\end{proof}

\begin{lemma}\label{L-positive}
For a finite Borel cover $\cU$ of $X$, the three quantities
$h_{\mu}^-(\cU)$, $\lwh_{\comb , \mu}(\cU)$, and
$\uph_{\comb , \mu}(\cU)$ are either all zero or all nonzero. If the complements
in $X$ of the members of $\cU$ are pairwise disjoint and $\oA$ is a tuple
consisting of these complements, then we may also add $\lwind_{\mu}(\oA)$ and
$\upind_{\mu}(\oA)$ to the list.
\end{lemma}

\begin{proof}
The first assertion follows from Lemma~\ref{L-local}. If $\oA$ is a tuple as in the
lemma statement, then Lemma~3.3 of \cite{Ind} and Lemma~\ref{L-converse}
yield the equivalence of $\lwh_{\comb , \mu}(\cU)>0$ and $\lwind_{\mu}(\oA)>0$ as well as
the equivalence of $\uph_{\comb ,\mu}(\cU)>0$ and $\upind_{\mu}(\oA)>0$.
\end{proof}

\begin{proposition}\label{P-IE basic}
The following hold:
\begin{enumerate}
\item Let $\oA = (A_1 , \dots , A_k )$ be a tuple of closed subsets of $X$ which
has positive upper $\mu$-independence density. Then there exists a $\mu$-IE-tuple
$(x_1 , \dots , x_k )$ with $x_j \in A_j$ for $j=1, \dots ,k$.

\item $\IE^\mu_2 (X) \setminus \Delta_2 (X)$ is nonempty if and only if
$h_\mu (X) > 0$.

\item $\IE^\mu_1(X)={\rm supp}(\mu)$.

\item $\IE^\mu_k (X)$ is a closed $G$-invariant subset of $X^k$.

\item Let $\pi : X\rightarrow Y$ be a topological $G$-factor map.
Then $\pi^k (\IE^\mu_k (X))=\IE^{\pi_* (\mu )}_k (Y)$.
\end{enumerate}
\end{proposition}

\begin{proof}
(1) Apply Lemma~\ref{L-split} and a compactness argument.

(2) As is well known and easy to show,
$h_\mu (X) > 0$ if and only if there is a two-element
Borel partition of $X$ with positive entropy. We can thus apply (1) and
Lemma~\ref{L-pe to pd} to obtain the ``if'' part. The ``only if'' part
follows from Lemma~\ref{L-positive}.

(3) This follows from Lemma~\ref{L-supp}.

(4) Trivial.

(5) This follows from (1), (3), (4), and Lemmas~\ref{L-hlifting} and
\ref{L-positive}.
\end{proof}

%%%%%%%%%%%%%%%%%%%%%%%%%%%%%%%%%%%%%%%%%%%%%%%%%%%%%%%%%%%%%%%%%%%%%%%%%%

\subsection{IE-tuples and measure IE-tuples}\label{SS-realize} 

Here we will show that the set of IE-tuples of length $k$ is equal
to the closure of the union of the sets $\IE_\mu^k (X)$ over all $G$-invariant Borel 
probability measures $\mu$ on $X$, and furthermore that when $X$ is metrizable
there exists a $G$-invariant Borel probability measure $\mu$ on $X$ 
such that the sets of $\mu$-IE-tuples and IE-tuples coincide.

We will need a version of the Rokhlin tower lemma.
%See \cite[page192]{Varadarajan} for the definition of standard measure spaces.
%Let $(Y,\sY , \nu , G)$ be a measure-preserving dynamical system.
Following \cite{EMAGA}, for a finite set $F\subseteq G$ and 
%$V\in \sY$ 
a Borel subset $V$ of $X$ 
we say that $F\times V$ maps
to an {\it $\varepsilon$-quasi-tower} if there exists a measurable
subset $A\subseteq F\times V$ such that the map $A\to X$ sending
$(s, x)$ to $sx$ is one-to-one and for each $x\in V$ the cardinality of
$\{s\in F : (s, x)\in A\}$ is at least $(1-\varepsilon)|F|$.
%We say that the action is {\it free} or the measure $\mu$ is {\it free} if the
%set of mixed points of each $e\neq g$ in $G$ has measure $0$.
The case $\delta=0$ of the following theorem is a direct consequence 
%appears as
of Theorem~5 on page~59 of \cite{OW}. The general case $\delta>0$ follows from
the proof given there.
%(It seems to me that the general case $\delta>0$ follows from
%Theorem 6 on \cite[page 61]{OW}. But the proof of Theorem 6 is much
%more complicated and I do not understand the relation between the
%constants involved there.)
Note that although the acting groups are generally assumed to be
countable in \cite{OW}, this assumption is not necessary here.

\begin{theorem}\label{T-Rokhlin}
Let $1>\varepsilon>0$ and $\frac{\varepsilon^2}{4}>\delta>0$. 
%Then there exist $k>0$ and
%$\eta_1, \dots, \eta_{k-1}>0$ depending only on $\varepsilon$
%and $\delta$ (not on $G$ or $X$) 
%such that 
Then whenever the action of $G$ is free with respect to $\mu$,
$F_1\subseteq F_2\subseteq \dots \subseteq F_k$ are nonempty
finite subsets of $G$ such that $F_{j+1}$ is
$(F_jF_j^{-1}, \eta_j)$-invariant and $\eta_j|F_j|<\frac{\varepsilon^2}{4}$
for all $1\le j<k$, $(1-\frac{\varepsilon}{2})^k<\varepsilon$, and
$D_1, \dots, D_k$ are Borel subsets of $X$ with $\mu$-measure at least
$1-\delta$, one can find Borel subsets $V_1, \dots, V_k$ such that
\begin{enumerate}
\item each $F_j\times V_j$ maps to an $\varepsilon$-quasi-tower,

\item $F_iV_i\cap F_jV_j=\emptyset$ for $i\neq j$,

\item $\mu \big( \bigcup^k_{j=1}F_jV_j \big) > 1-\varepsilon$,

\item $V_j\subseteq D_j$ for each $j$.
\end{enumerate}
\end{theorem}

For the definitions of the quantities $\hmeas_\mu^+ (\cU )$ and 
$\lwh_{\comb , \mu}(\cU)$ see the discussion after Lemma~\ref{L-supp}.

\begin{lemma} \label{L-compare free}
Suppose that 
$G$ is infinite and the action of $G$ is free
with respect to $\mu$.
Let $\cU$ be a finite Borel cover of $X$.
Then $h_{\mu}^+(\cU)\le \lwh_{\comb , \mu}(\cU)$.
\end{lemma}

\begin{proof}
Let $1>\varepsilon>0$ and $\frac{\varepsilon^2}{4}>\delta>0$. 
Then we can find
nonempty finite subsets $F_1\subseteq F_2\subseteq
\cdots\subseteq F_k$ of $G$ satisfying the conditions 
of Theorem~\ref{T-Rokhlin}
and 
$\frac{1}{|F_j|}\ln N_{\delta}(\cU^{F_j})<\lwh_{\comb , \mu}(\cU,
\delta)+\varepsilon$ for $j=1, \dots , k$. For each $j=1, \dots , k$ take
a $D_j\in \sB (\mu, \delta)$ such that $\frac{1}{|F_j|}\ln
N_{D_j}(\cU^{F_j})<\lwh_{\comb , \mu}(\cU, \delta)+\varepsilon$. Then we
can find Borel sets $V_1, \dots, V_k \subseteq X$ satisfying the conclusion
of Theorem~\ref{T-Rokhlin}.

For $j=1, \dots , k$ pick a Borel partition $\cP_j$ of $D_j$ which is
finer than the restriction of $\cU^{F_j}$ to $D_j$ and has cardinality
$N_{D_j}(\cU^{F_j})$.
For each $P\in \cP_j$ fix a $U_{P, s}\in \cU$ for each $s\in F_j$ such that
$P\subseteq \bigcap_{s\in F_j}s^{-1}U_{P, s}$. Since $F_j\times V_j$
maps to an $\varepsilon$-quasi-tower, we can find a measurable
subset $A_j$ of $F_j\times V_j$ such that $T|_{A_j}: A_j\rightarrow
X$ is one-to-one, where $T:G\times X \to X$ is the map $(s, x)\mapsto sx$,
and $|\{s\in F_j: (s, x)\in A_j\}|\ge (1-\varepsilon)|F_j|$
for each $x\in V_j$. Define a Borel partition $\cY = \{Y_U: U\in
\cU\}$ of $\bigcup_j T(A_j)$ finer than the restriction of $\cU$ to
$\bigcup_j T(A_j)$ by stipulating that, for each $(s, x)\in A_j$ with $x\in P\in
\cP_j$, $sx\in Y_U$ exactly when $U=U_{P, s}$. Take a Borel partition
$\mathcal{Z}=\{Z_U: U\in \cU\}$ of $( \bigcup_j T(A_j) )^\comp$
with $Z_U\subseteq U$ for each $U\in \cU$. Set $P_U=Y_U\cup Z_U$ for each
$U\in \cU$. Then $\cP=\{P_U:U\in \cU\}$ is a Borel partition of $X$ finer
than $\cU$. Note that $\mu(T(A_j))\ge (1-\varepsilon)\mu(F_jV_j)$
for each $j$. Thus $\mu(\bigcup_j T(A_j))> (1-\varepsilon)^2$.

Next we estimate $h_{\mu}(\cP)$.
Suppose that $F$ is a finite subset of $G$ which is\linebreak
$((\bigcup_jF_j)(\bigcup_jF_j)^{-1}, \sqrt{\varepsilon})$-invariant.
% for all
%$1\le j\le k$.
Set $F_x=\big\{ s\in F: sx\in \bigcup_jT(A_j) \big\}$ for each $x\in X$ and put
$W=\{x\in X: |F_x|\ge (1-\sqrt{\varepsilon})|F|\}$. It is easy to
see that $\mu(W^\comp )\le \mu((\bigcup_j
T(A_j))^\comp )/\sqrt{\varepsilon}<2\sqrt{\varepsilon}$. Replacing $W$ by
$W\setminus \bigcup_{s\in F^{-1}F\setminus \{e_G\}}\{x\in X: sx=x\}$ we may assume
that $s_1x\neq s_2x$ for all $x\in W$ and all distinct $s_1, s_2\in F$. Let us estimate
the number $M$ of atoms of $\cP^F$ which have nonempty intersection with
$W$. Write $\mathfrak{H}_j$ for the collection of all subsets of $F_j$ with
cardinality at least $(1-\varepsilon)|F_j|$. For each $x\in W$,
setting $F'_x=F_x\cap \big\{ s\in F: (\bigcup_jF_j)(\bigcup_jF_j)^{-1}s\subseteq F \big\}$,
we have $|F'_x|\ge (1-2\sqrt{\varepsilon})|F|$. Note that if $(s, y)\in A_j$ for some
$1\le j\le k$ and $sy=s'x$ for some $s'\in F'_x$, setting $c=s^{-1}s'$ and
$H=\{h\in F_j:(h, y)\in A_j\}$, we have $y=cx$, $Hc\subseteq F_x$ and $H\in \mathfrak{H}_j$.
Thus for each $w\in W$ we
can find a finite set $C_{j, H} \subseteq G$ for every $H\in
\mathfrak{H}_j$ such that the following hold:
\begin{enumerate}
\item $Hc\cap H'c'=\emptyset$ for all $c\in C_{j, H}, c'\in C_{j', H'}$
unless $H=H'$, $c=c'$, and $j=j'$,

\item $\bigcup_{j, H}HC_{j, H}\subseteq F$ and $\big| \bigcup_{j, H}HC_{j,
H} \big| \ge (1-2\sqrt{\varepsilon})|F|$,

\item $cx\in V_j$ and $H=\{h\in F_j: (h, cx)\in A_j\}$ for each $c\in C_{j, H}$.
\end{enumerate}
Note that the atom of $\cP$ to which $hcx$ for $h\in H$ belongs is
determined by $h$ and the atom of $\cP_j$ to which $cx$ belongs. Thus,
for each fixed choice of sets $C_{j, H}$ satisfying (1) and (2) above,
the number of atoms of $\cP^F$ containing some $x\in W$ with such a
choice of $C_{j, H}$ is at most
\begin{align*}
|\cU|^{2\sqrt{\varepsilon}|F|} \cdot \prod_{j}|\cP_j|^{\sum_{H\in
\mathfrak{H}_j} |C_{j, H}|} &\le |\cU|^{2\sqrt{\varepsilon}|F|}
\cdot \prod_{j} \exp \!\big[ (\lwh_{\comb , \mu}(\cU,
\delta)+\varepsilon)|F_j|{\textstyle\sum_{H\in\mathfrak{H}_j}} |C_{j, H}| \big] \\
&= |\cU|^{2\sqrt{\varepsilon}|F|} \cdot \exp \!\big[ (\lwh_{\comb , \mu}(\cU,
\delta)+\varepsilon){\textstyle\sum_j \big( |F_j|\sum_{H\in \mathfrak{H}_j}} |C_{j,
H}| \big) \big]\\
&\le |\cU|^{2\sqrt{\varepsilon}|F|} \cdot \exp \!\bigg[
\frac{(\lwh_{\comb , \mu}(\cU,\delta)+\varepsilon)|F|}{1-\varepsilon} \bigg] .
\end{align*}
By Stirling's formula, the number of subsets of an $n$-element set
with cardinality at least $(1-\varepsilon)n$ is at most
$e^{f(\varepsilon)n}$ for all $n\ge 0$ with $f(\varepsilon)\to 0$ as
$\varepsilon\to 0$. Fix an element $g_{j, H}\in H$ for each $j$ and $H\in
\mathfrak{H}_j$. Then $C_{j, H}$ is determined by the set $g_{j, H}C_{j,
H}$ in $F$. Thus, for a fixed $Q\subseteq F$, writing $a=\min_j |F_j|$
and summing as appropriate over nonnegative integers $t_{j, H}$, $t_j$, or $t$
subject to the indicated constraints,
the number of choices of sets $C_{j, H}$ satisfying (1) and (2) and
$\bigcup_{j,H}HC_{j, H}=Q$ is at most
\begin{align*}
\lefteqn{\sum_{\sum_{j, H}t_{j, H}|H|=|Q|}\frac{|F|!}{\big( |F|-\sum_{j,
H}t_{j, H} \big) !\prod_{j, H}t_{j,H}!}} \hspace*{20mm} \\
\hspace*{15mm} &\le \sum_{(1-\varepsilon)\sum_{j}t_{j}|F_j|\le
|Q|}\frac{|F|!}{\big( |F|-\sum_{j}t_{j}\big) !\prod_{j}t_{j}!}\cdot
\prod_j\sum_{\sum_{H\in \mathfrak{H}_j}t_{j,
H}=t_j}\frac{|t_j|!}{\prod_{H\in \mathfrak{H}_j}t_{j, H}!}\\
&= \sum_{(1-\varepsilon)\sum_{j}t_{j}|F_j|\le
|Q|}\frac{|F|!}{\big( |F|-\sum_{j}t_{j}\big) !\prod_{j}t_{j}!}\cdot \prod_j
|\mathfrak{H}_j|^{t_j}\\
&\le \sum_{(1-\varepsilon)\sum_{j}t_{j}|F_j|\le
|Q|}\frac{|F|!}{\big( |F|-\sum_{j}t_{j} \big)!\prod_{j}t_{j}!}\cdot \prod_j
e^{f(\varepsilon)t_j|F_j|}\\
&\le \sum_{(1-\varepsilon)\sum_{j}t_{j}|F_j|\le
|Q|}\frac{|F|!}{\big( |F|-\sum_{j}t_{j}\big) !\prod_{j}t_{j}!}\cdot
e^{f(\varepsilon)|F|/(1-\varepsilon)}\\
&\le \sum_{(1-\varepsilon)at\le |Q|}\frac{|F|!}{(|F|-t)!t!}\cdot
\sum_{\sum_jt_j=t}\frac{t!}{\prod_j t_j!}\cdot
e^{f(\varepsilon)|F|/(1-\varepsilon)}\\
&= \sum_{(1-\varepsilon)at\le |Q|}\frac{|F|!}{(|F|-t)!t!}\cdot
k^t\cdot e^{f(\varepsilon)|F|/(1-\varepsilon)}\\
&\le \sum_{(1-\varepsilon)at\le |Q|}\frac{|F|!}{(|F|-t)!t!}\cdot
k^{|F|/((1-\varepsilon)a)}\cdot
e^{f(\varepsilon)|F|/(1-\varepsilon)}\\
&\le e^{f(1/((1-\varepsilon )a)) |F|}\cdot
k^{|F|/((1-\varepsilon)a)}\cdot
e^{f(\varepsilon)|F|/(1-\varepsilon)} .
\end{align*}
%where $t_{j, H}$, $t_j$ and $t$ are nonnegative integers, and
%$a=\min_j |F_j|$.
The number of choices of $Q\subseteq F$ with
$|Q|\ge (1-2\sqrt{\varepsilon})|F|$ is at most
$e^{f(2\sqrt{\varepsilon})|F|}$. Therefore, $M$ is at most
%\[|\cU|^{2\sqrt{\varepsilon}|F|} \cdot e^{(\lwh_{\comb , \mu}(\cU,
%\delta)+\varepsilon)|F|/(1-\varepsilon)}\cdot
%e^{f(\frac{1}{(1-\varepsilon)s})|F|}\cdot
%k^{|F|/((1-\varepsilon)a)}\cdot
%e^{f(\varepsilon)|F|/(1-\varepsilon)}\cdot
%e^{f(2\sqrt{\varepsilon})|F|}.\]
\begin{gather*}
|\cU|^{2\sqrt{\varepsilon}|F|} \cdot \exp \bigg[
\frac{(\lwh_{\comb , \mu}(\cU,\delta)+\varepsilon) |F|}{1-\varepsilon} \bigg] \cdot
\exp \big[ f\big( 1/((1-\varepsilon)a)\big) |F| \big] \hspace*{30mm} \\
\hspace*{40mm} \cdot\,
k^{|F|/((1-\varepsilon)a)}\cdot
\exp \bigg[ \frac{f(\varepsilon)|F|}{1-\varepsilon} \bigg] \cdot
\exp \big[ f(2\sqrt{\varepsilon})|F| \big] .
\end{gather*}
Since the function $x\mapsto -x\ln x$ is concave on $[0, 1]$, we have
\[ \sum_{P\in \cP^F} -\mu(P\cap W)\ln \mu(P\cap W)\le -\mu(W)\ln \frac{\mu(W)}{M} \le -\mu(W)\ln
\mu(W)+\ln M \]
and
\begin{align*}
\sum_{P\in \cP^F}-\mu(P\cap W^\comp )\ln \mu(P\cap W^\comp )
&\le -\mu(W^\comp) \ln \frac{\mu(W^\comp)}{|\cP|^{|F|}-M} \\
&\le -\mu(W^\comp )\ln\mu(W^\comp )+\mu(W^\comp )|F|\ln |\cU| .
\end{align*}
Set $\mathcal{Q}=\{W, W^\comp \}$. Since the function $x\mapsto -x\ln x$ on $[0, 1]$
has maximal value $e^{-1}$, we get
\begin{align*}
\Hmeas (\cP^F) \le \Hmeas (\cP^F\vee \mathcal{Q}) &= \sum_{P\in \cP^F} -\mu(P\cap
W)\ln \mu(P\cap W)+\sum_{P\in \cP^F}-\mu(P\cap W^\comp )\ln \mu(P\cap W^\comp)\\
&\le  -\mu(W)\ln \mu(W)+\ln M-\mu(W^\comp )\ln \mu(W^\comp )+\mu(W^\comp )|F|\ln
|\cU|\\
&\le 2e^{-1}+\ln M+2\sqrt{\varepsilon}|F|\ln |\cU|.
\end{align*}
Since $G$ is infinite, $|F|\to \infty$ as $F$ becomes more and more
invariant. Therefore
\begin{align*}
h_{\mu}^+(\cU) \le h_{\mu}(\cP)
&\le 4\sqrt{\varepsilon}\ln |\cU|+
\frac{\lwh_{\comb , \mu}(\cU,\delta)+\varepsilon}{1-\varepsilon}+
f\big( 1/ ((1-\varepsilon)a)\big) \\
&\hspace*{35mm} \ + \frac{\ln k}{(1-\varepsilon)a}+ \frac{f(\varepsilon)}{1-\varepsilon}+
f(2\sqrt{\varepsilon}).
\end{align*}
Since we may choose $F_1, \dots, F_k$ to be as close as we wish
%in proportional terms
to being invariant, we may let $a\to \infty$. Thus
\begin{align*}
h_{\mu}^+(\cU)
&\le 4\sqrt{\varepsilon}\ln |\cU|+\frac{\lwh_{\comb , \mu}(\cU,\delta)+\varepsilon}{1-\varepsilon}
+ \frac{f(\varepsilon)}{1-\varepsilon}+ f(2\sqrt{\varepsilon})\\
&\le 4\sqrt{\varepsilon}\ln |\cU|+\frac{\lwh_{\comb ,\mu}(\cU)+\varepsilon}{1-\varepsilon} +
\frac{f(\varepsilon)}{1-\varepsilon}+ f(2\sqrt{\varepsilon}).
\end{align*}
Letting $\varepsilon\to 0$ we get $h_{\mu}^+(\cU)\le \lwh_{\comb ,
\mu}(\cU)$, as desired.
\end{proof}

\begin{lemma}\label{L-partition}
%Suppose that $X$ is metrizable.
Let $\mu$ be a Borel probability measure on $X$. Let $C_1 ,\dots ,C_k$ be 
closed 
subsets of $X$. Then for every $k$-element 
Borel partition $\cP = \{ P_1 , \dots , P_k \}$ with $P_i \cap C_i = \emptyset$ 
for $i=1, \dots ,k$ and every $\delta > 0$ there is a $k$-element Borel partition
$\cQ = \{ Q_1 , \dots , Q_k \}$ such that $Q_i \cap C_i = \emptyset$ and 
$\mu (\partial Q_i ) = 0$ for $i=1,\dots ,k$ and $H_\mu (\cQ | \cP ) < \delta$.
\end{lemma}

\begin{proof}
Let $\cP = \{ P_1 , \dots , P_k \}$ be a $k$-element Borel partition with 
$P_i \cap C_i = \emptyset$ for $i=1, \dots ,k$. Let $\delta > 0$.
By the regularity of $\mu$, for $i=1, \dots ,k-1$ we can find a compact set
$K_i \subseteq P_i$ such that $\mu (P_i \setminus K_i ) < \varepsilon$
and an open set $U_i \supseteq P_i$ such that $\mu (U_i \setminus P_i ) < \varepsilon$
and $U_i \cap C_i = \emptyset$. 
Then $U_1, \dots, U_{k-1}$ cover $C_k$. Thus we can find
a closed cover $D_1, \dots, D_{k-1}$ of $C_k$ such that
$D_i\subseteq U_i$ for $i=1, \dots, k-1$.
For each $x\in K_i\cup D_i$ there exists an open neighbourhood  
$V$ of $x$ contained in $U_i$ whose boundary 
has zero measure, for if we take a function $f\in C(X)$ with image in $[0,1]$ 
which is $0$ at $x$ and $1$ on $U_i^\comp$ then only countably many 
of the open sets $\{ y\in X : f(y) < t \}$ for $t\in (0,1)$ can have boundary
with positive measure. By compactness there is a finite union $B_i$ of such
$V$ which covers $K_i\cup D_i$, and $\mu (\partial (B_i )) = 0$.
Then $\mu (B_i \Delta P_i ) < 2\varepsilon$ for 
$i=1, \dots, k-1$.
Now define the partition
$\cQ = \{ Q_1 , \dots , Q_k \}$ by $Q_1 = B_1$, $Q_2 = B_2 \setminus B_1$,
$Q_3 = B_3 \setminus (B_1 \cup B_2 )$, \dots ,
$Q_k = X \setminus (B_1 \cup\cdots\cup B_{k-1} )$. 
Then $Q_i \cap C_i = \emptyset$ and
$\mu (\partial Q_i ) = 0$ for $i=1,\dots ,k$ and 
$H_\mu (\cQ | \cP ) < \delta (\varepsilon )$
where $\delta (\varepsilon ) \to 0$ as $\varepsilon\to 0$, yielding the lemma.
\end{proof}

\begin{lemma}\label{L-measure}
Let $\ox = (x_1 , \dots , x_k )$ be an $IE$-tuple consisting of distinct points
and let $U_1 , \dots ,U_k$ be pairwise disjoint open neighbourhoods of $x_1 , \dots ,x_k$,
respectively. Then there exist a $G$-invariant Borel probability measure $\mu$ on $X$ 
and a $\mu$-IE-tuple $(x_1' , \dots ,x_k' )$ such that $x_i' \in U_i$ for each
$i=1,\dots ,k$. 
\end{lemma}

\begin{proof}
The case $k=1$ follows from \cite[Prop. 3.12]{Ind} and 
Proposition~\ref{P-IE basic}(3). 
So we may assume $k\ge 2$. 

Let $\{ F_n \}_n$ be a F{\o}lner net in $G$. For each $i=1,\dots ,k$ choose a
closed neighbourhood $C_i$ of $x_i$ contained in $U_i$. 
%By Problem~4.10 of \cite{Pat} we may assume that 
%$\lim_{n\to\infty} |F_n F_{n-1}^{-1} \Delta F_n |/|F_n | = 0$.
Since $\ox$ is an IE-tuple there is a $d>0$ such that for each $n$ we can find 
an independence set $I_n \subseteq F_n$ for the tuple $\oC = (C_1 , \dots , C_k )$
such that $|I_n | \geq d|F_n |$. For each $n$ pick an
$x_\sigma \in \bigcap_{s\in I_n} s^{-1} C_{\sigma (s)}$ for every 
$\sigma\in\{ 1,\dots ,k \}^{I_n}$ and define on $X$ the following averages of point masses: 
\[ \nu_n = \frac{1}{k^{|I_n |}} \sum_{\sigma\in\{ 1,\dots ,k \}^{I_n}} \delta_{x_\sigma} , 
\hspace*{8mm} \mu_n = \frac{1}{|F_n |} \sum_{s\in F_n} s\mu_n . \]
Take a weak$^*$ limit point $\mu$ of the net $\{ \mu_n \}_n$. 
By passing to a cofinal subset of the net we may assume that  
$\mu_n$ converges to $\mu$. 

Let $\cP = \{ P_1 , \dots , P_k \}$ be a Borel partition of $X$ such that 
$P_i \cap U_i = \emptyset$ and 
$\mu (\partial P_i ) = 0$ for each $i=1,\dots ,k$. 
Let $E$ be a nonempty finite subset of $G$. We will use subadditivity and concavity as in the
proof of the variational principle in Section~5.2 of \cite{JMO}. 
The function $A\mapsto H_{\nu_{n}} (\cP^A )$ on finite subsets of $G$ is subadditive 
in the sense that if $\boldsymbol{1}_A = \sum\lambda_B \boldsymbol{1}_B$ is a finite 
decomposition of the indicator of a finite set $A\subseteq G$ over a collection of  
sets $B\subseteq A$ with each $\lambda_B$ positive, then 
$H_{\nu_{n}} (\cP^A ) \leq \sum \lambda_B H_{\nu_{n}} (\cP^B )$ 
(see Section~3.1 of
\cite{JMO}).
Observe that $\varepsilon (n) := | E^{-1} F_n \setminus F_n |/ |F_n |$ is bounded
above by $| E^{-1} F_n \Delta F_n |/ |F_n |$ and hence by the F{\o}lner property
tends to zero along the net.
Applying the subadditivity of $H_{\nu_n} (\cdot )$ to the decomposition 
$\boldsymbol{1}_{F_n} = \frac{1}{|E|} \sum_{s\in E^{-1} F_n} \boldsymbol{1}_{Es \cap F_n}$, 
we have 
\begin{align*}
H_{\nu_{n}} (\cP^{F_n} ) 
&\leq \frac{1}{|E|} \sum_{s\in F_n} H_{\nu_{n}} (\cP^{Es} ) 
+ \frac{1}{|E|} \sum_{s\in E^{-1} F_n \setminus F_n} H_{\nu_{n}} (\cP^{Es} ) \\
&\leq \frac{1}{|E|} \sum_{s\in F_n} H_{\nu_{n}} (\cP^{Es} ) + 
\varepsilon (n) |F_n | \ln k .
\end{align*}
Since $P_i \cap C_i = \emptyset$ for each $i$, every atom of $\cP^{I_n}$ 
contains at most $(k-1)^{|I_n|}$ points from the set 
$\{ x_\sigma : \sigma\in \{ 1,\dots ,k \}^{I_n} \}$ and hence has $\nu_{n}$-measure at
most $(\frac{k-1}{k} )^{|I_n |}$, so that 
\begin{align*}
H_{\nu_{n}} (\cP^{I_n} ) &= \sum_{W\in\cP^{I_n}} -\nu_{n} (W) \ln \nu_{n} (W) \\
&\geq \sum_{W\in\cP^{I_n}} \nu_{n} (W) \ln \bigg( \frac{k}{k-1} \bigg)^{|I_n |} \\
&= |I_n | \ln \bigg( \frac{k}{k-1} \bigg)
\end{align*}
and thus
\begin{align*} 
\frac{1}{|F_n |} H_{\nu_{n}} (\cP^{F_n} ) \geq \frac{1}{|F_n |} H_{\nu_{n}} (\cP^{I_n} ) 
\geq d \ln \bigg( \frac{k}{k-1} \bigg) . 
\end{align*}
It follows using the concavity of the function $x\mapsto -x\ln x$ that
\begin{align*}
\frac{1}{|E|} H_{\mu_{n}} (\cP^E ) 
\geq \frac{1}{|F_n |} \sum_{s\in F_n} \frac{1}{|E|} H_{\nu_{n}} (\cP^{Es} ) 
\geq d\ln \bigg( \frac{k}{k-1} \bigg) - \varepsilon (n) \ln k .
\end{align*}
Since the boundary of each $P_i$ has zero $\mu$-measure,
the boundary of each atom of $\cP^E$ has zero $\mu$-measure, and so by \cite[Thm. 17.20]{CDST}
the entropy of $\cP^E$ is a continuous function of the measure with respect to the
weak$^*$ topology, whence
\[ \frac{1}{|E|} H_{\mu} (\cP^E ) = \lim_n \frac{1}{|E|} H_{\mu_{n}} (\cP^E ) \geq 
d\ln (k/(k-1)) . \]
Since this holds for every nonempty finite set $E\subseteq G$, we obtain 
$h_\mu (\cP ) \geq d\ln (k/(k-1))$.  

Now let $\cP = \{ P_1 , \dots , P_k \}$ be any $k$-element Borel partition of $X$ such that
$P_i \cap U_i = \emptyset$ for each $i=1, \dots ,k$. By Lemma~\ref{L-partition}, for
every $\delta > 0$ there is a $k$-element Borel partition
$\cQ = \{ Q_1 , \dots , Q_k \}$ such that $Q_i \cap C_i = \emptyset$ and 
$\mu (\partial Q_i ) = 0$ for $i=1,\dots ,k$ and $H_\mu (\cQ | \cP ) < \delta$,
so that $h_{\mu} (\cP ) \geq h_{\mu} (\cQ ) - \delta \geq d\ln (k/(k-1)) - \delta$
by the previous paragraph. Thus $h_{\mu} (\cP ) \geq d\ln (k/(k-1))$.
This inequality holds moreover for any finite Borel partition $\cP$ that refines
$\cU:=\{ U_1^\comp , \dots , U_k^\comp \}$ as a cover since we may assume that $\cP$
is of the above form by coarsening it if necessary. 
Therefore $h_\mu^+ (\cU ) > 0$.

Suppose that the action of $G$ on $X$ is (topologically) free, i.e., 
for all $x\in X$ and $s\in G$, $sx = x$ implies $s=e$.
Then it is free with respect to $\mu$, and
hence $\lwh_{\comb ,\mu} (\cU ) > 0$ by Lemma~\ref{L-compare free}. 
Therefore by Lemma~\ref{L-positive} and Proposition~\ref{P-IE basic}(1)
there is 
a $\mu$-IE-tuple $(x_1' , \dots ,x_k' )$ contained
in $U_1\times \cdots \times U_k$.

Now suppose that the action of $G$ on $X$ is not free. Take a 
free action of $G$ on a compact Hausdorff space $(Y,G)$, e.g., the universal minimal
$G$-system \cite{Ellis}.
Then the product system $(X\times Y,G)$
is an extension of $(X,G)$ which is free. 
By Proposition~3.9(4) of \cite{Ind} we can find 
a lift $\tilde{\ox}$ of the tuple $\ox$ under this
extension such that $\tilde{\ox}$ is an IE-tuple.
By the previous paragraph
there are a $G$-invariant Borel probability measure $\mu$ on $X\times Y$ 
and a $\mu$-IE-tuple $\tilde{\ox}'$ 
contained in the inverse image of $U_1\times \cdots \times U_k$.
 It then follows by Proposition~\ref{P-IE basic}(5) that the image
$\ox'$ of $\tilde{\ox}'$ is 
a $\nu$-IE-tuple contained in $U_1\times \cdots \times U_k$
for the measure $\nu$ on $X$ induced from $\mu$, completing the proof. 
\end{proof}

\noindent From Lemma~\ref{L-measure} we obtain:

\begin{theorem}\label{T-mu-IE IE closure}
For each $k\geq 1$ the set of IE-tuples of length $k$ is equal to the closure of 
of the union of the sets $\IE_\mu^k (X)$ over all $G$-invariant Borel probability 
measures $\mu$ on $X$.
\end{theorem}

\begin{lemma}\label{L-point}
Suppose that $X$ is metrizable. 
Let $\ox = (x_1 , \dots , x_k )$ be an IE-tuple.
Then there is a $G$-invariant Borel probability measure $\mu$ on $X$ such that 
$\ox$ is a $\mu$-IE-tuple.
\end{lemma}

\begin{proof}
We may assume that $\ox$ consists of distinct points. 
Since $X$ is metrizable, we can find for each $m\in\Nb$ pairwise disjoint open neighbourhoods 
$U_{m,1} , \dots ,U_{m,k}$ of $x_1 , \dots ,x_k$, respectively, so that for each
$i=1,\dots ,k$ the family $\{ U_{m,i} : m\in\Nb \}$ forms a neighbourhood basis for $x_i$. 
For each $m$ take a measure $\mu_m$ and a $\mu$-IE-tuple 
$\ox_m$ as given by Lemma~\ref{L-measure} with respect to
$U_{m,1} , \dots ,U_{m,k}$ and define the $G$-invariant Borel probability measure 
$\mu = \sum_{m=1}^\infty 2^{-m} \mu_m$. Then $\ox_m$ is 
a $\mu$-IE-tuple for each $m$, and so $\ox$ is a $\mu$-IE-tuple
by Proposition~\ref{P-IE basic}(4).
\end{proof}

\begin{theorem}\label{T-mu-IE IE}
Suppose that $X$ is metrizable. Then there is a $G$-invariant Borel 
probability measure $\mu$ on $X$ such that the sets of $\mu$-IE-tuples and IE-tuples coincide.
\end{theorem}

\begin{proof}
For each $k\geq 1$ take a countable dense subset $\{ \ox_{k,i} \}_{i\in I_k}$ of the set of IE-tuples 
of length $k$. By Lemma~\ref{L-point}, for every $k\geq 1$ and $i\in I_k$ there is a 
$G$-invariant Borel probability measure $\mu_{k,i}$ on $X$ such that $\ox_{k,i}$
is a $\mu_{k,i}$-IE-tuple. Set $\mu = \sum_{k=1}^\infty \sum_{i\in I_k} \lambda_{k,i} \mu_{k,i}$
for some $\lambda_{k,i} > 0$ with $\sum_{k=1}^\infty \sum_{i\in I_k} \lambda_{k,i} = 1$.
Then $\mu$ is a $G$-invariant Borel probability measure, and $\ox_{k,i}$
is a $\mu$-IE-tuple for every $k\geq 1$ and $i\in I_k$. Since the set of 
$\mu$-IE-tuples of a given length is closed by Proposition~\ref{P-IE basic}(4)
and $\mu$-IE-tuples are always IE-tuples, we obtain the desired conclusion.
\end{proof}

In the case $G=\Zb$, the conclusion of Theorem~\ref{T-mu-IE IE} for $\mu$-entropy pairs 
and topological entropy pairs was established in \cite{BGH} and then more generally 
for $\mu$-entropy tuples and topological entropy tuples in \cite{LVRA}.

%%%%%%%%%%%%%%%%%%%%%%%%%%%%%%%%%%%%%%%%%%%%%%%%%%%%%%%%%%%%%%%%%%%%%%%%%%%%%

\subsection{The relation between $\mu$-IE-tuples and $\mu$-entropy tuples}

For $G=\Zb$ the notion of a $\mu$-entropy pair was introduced in \cite{BHMMR}
and generalized to $\mu$-entropy tuples in \cite{LVRA}.
%and has been an important tool in the local study of both measure entropy and
%topological entropy (see \cite{ETJ}).
We will accordingly say for $k\geq 2$ that a nondiagonal tuple
$(x_1 , \dots , x_k ) \in X^k$ is a
{\it $\mu$-entropy tuple} if whenever $U_1 , \dots , U_l$
are pairwise disjoint Borel neighbourhoods of the distinct points in
the list $x_1 , \dots , x_k$, every Borel partition of $X$ refining
$\{ U_1^\comp , \dots , U_l^\comp \}$ has positive measure entropy.
In this subsection we aim to show that nondiagonal $\mu$-IE-tuples are the same
as $\mu$-entropy tuples.

Our first task is to establish Lemma~\ref{L-positive inf}. For this we will
use the orbit equivalence technique of Rudolph and Weiss \cite{EMAGA}, which will enable
us to apply a result of Huang and Ye for $\Zb$-actions \cite{LVRA}.
In order to invoke Theorem~2.6 of \cite{EMAGA}, whose hypotheses include ergodicity,
we will need the ergodic decomposition of entropy, which asserts that if
$(Y,\sY , \nu )$ is a Lebesgue space equipped with an action of a countable
discrete amenable group $H$ and $\nu = \int_Z \nu_z \, d\omega (z)$ is the
corresponding ergodic decomposition, then for every finite measurable partition $\cP$
of $Y$ we have $\hmeas_\nu (\cP ) = \int_Z \hmeas_{\nu_z} (\cP )\, d\omega (z)$.
The standard proof of this for $G=\Zb$ using symbolic representations (see for example
Section~15.3 of \cite{ETJ}) also works in the general case.
Given a tuple $\oA = ( A_1 , \dots , A_k )$ of Borel
subsets of $X$ with $\bigcap_{i=1}^k A_i = \emptyset$, we say that a
finite Borel partition $\cP$ of $X$ is {\it $\oA$-admissible} if it refines
$\{ A_1^\comp , \dots , A_k^\comp \}$ as a cover of $X$. For the definitions of the 
quantities $\hmeas_\mu^+ (\cU )$ and $\lwh_{\comb , \mu}(\cU)$ see the discussion 
after Lemma~\ref{L-supp}. As the proof below
involves several different systems, we will explicitly indicate the
action in our notation for the various entropy quantities.

\begin{lemma}\label{L-positive inf}
Suppose that $X$ is metrizable and $G$ is countably infinite.
%and denote by $T$ the action of G$ on $X$.
Let $\oA = ( A_1 , \dots , A_k )$ be a tuple of pairwise disjoint Borel
subsets of $X$. Denote by $\cU$ the Borel cover
$\{ A_1^\comp , \dots , A_k^\comp \}$ of $X$. Suppose that $\hmeas_\mu (\cP ) > 0$
for every $\oA$-admissible finite Borel partition $\cP$ of $X$.
Then $\lwh_{\comb , \mu} (\cU) > 0$.
\end{lemma}

\begin{proof}
Denote by $T$ the action of $G$ on $X$. Take a free weakly mixing action
$S$ of $G$ on a Lebesgue space $(Y,\sY , \nu )$ (for example a Bernoulli action).
We will consider the product
action $T\times S$ on $(X\times Y, \sB\otimes\sY , \mu\times\nu )$ and view $\sB$
and $\sY$ as sub-$\sigma$-algebras of $\sB\otimes\sY$ when convenient.
Since $S$ is free and ergodic, by the Connes-Feldman-Weiss theorem \cite{CFW} there
is an integer action $\hat{R}$ on $(Y,\sY , \nu )$ with the same orbits as
$S$ and we may choose $\hat{R}$ to have zero measure entropy. Now we define an
integer action $R$ on $(X\times Y, \sB\otimes\sY , \mu\times\nu )$ with the same
orbits as $T\times S$ by setting $R(x,y) = (T\times S)_{s(y)} (x,y)$ where
$s(y)$ is the element of $G$ determined by $\hat{R} y = S_{s(y)} y$.
%Note that the orbit change from $T\times S$ to $R$ is $\sY$-measurable in the sense
%of Definition~2.5 in \cite{EMAGA}.

Let $\pi : (X,\sB , \mu )\to (Z,\sZ , \omega )$ be the dynamical factor defined by
the $\sigma$-algebra $\sI_T$ of $T$-invariant sets in $\sB$. We write the disintegration
of $\mu$ over $\omega$ as $\mu = \int_Z \mu_z \, d\omega (z)$ and for every $z\in Z$
put $X_z = \pi^{-1} (z)$ and $\sB_z = \sB \cap X_z$ and denote by $T_z$
the restriction of $T$ to $(X_z , \sB_z , \mu_z )$. Since $S$ is weakly mixing,
the $\sigma$-algebra $\sI_{T\times S}$ of $(T\times S)$-invariant sets in $\sB\otimes\sY$
coincides with $\sI_T$, viewing the latter as a sub-$\sigma$-algebra of $\sB\otimes\sY$.
The dynamical factor $(X\times Y,\sB\otimes\sY , \mu\times\nu )
\to (Z,\sZ , \omega )$ defined by $\sI_{T\times S}$ is the product of $\pi$ and the
trivial factor and gives the ergodic decomposition of $T\times S$
%disintegration of $\mu\times\nu$ over $\omega$
%written as $\mu = \int_Z (\mu_z \times\nu )\, d\omega (z)$
with $\omega$-a.e.\ ergodic components
$(X_z \times Y, \sB_z \otimes\sY , \mu_z \times\nu )$ with action
$T_z \times S$ for $z\in Z$.
The orbit equivalence of $R$ with $T\times S$ respects the ergodic decomposition
and so for $R$ we have $\omega$-a.e.\ ergodic components
$(X_z \times Y, \sB_z \otimes\sY , \mu_z \times\nu )$ with action
$R_z$ for $z\in Z$. Note that for each $z\in Z$ the action $R_z$ is free and
the orbit change from $T_z \times S$ to $R_z$ is $\sY$-measurable in the sense
of Definition~2.5 in \cite{EMAGA}.

Write $\oB$ for the tuple $(A_1 \times Y, \dots , A_k \times Y )$
of pairwise disjoint $\sB$-measurable subsets of $X\times Y$.
Let $\cQ = \{ Q_1 , \dots Q_r \}$ be a $\oB$-admissible finite measurable partition
of $X\times Y$. We will show that there exists a set of
$z\in Z$ of nonzero measure for which
$\hmeas_{\mu_z \times\nu} (T_z \times S,\cQ_z | \sY ) > 0$, where
$\cQ_z = \{ Q_j \cap (X_z \times Y) : j=1, \dots ,r \}$. Suppose to the contrary that
$\hmeas_{\mu_z \times\nu} (T_z \times S,\cQ_z | \sY ) = 0$ for $\omega$-a.e.\ $z\in Z$.
Consider the conditional expectations
$E^{\sB} = \id_{L^1 (X,\mu )} \otimes\nu
: L^1 (X \times Y,\mu\times\nu ) = L^1 (X,\mu )\widehat{\otimes} L^1 (Y,\nu )
\to L^1 (X ,\mu )$ and
$E^{\sB_z} = \id_{L^1 (X_z ,\mu_z )} \otimes\nu
: L^1 (X_z \times Y,\mu_z \times\nu ) = L^1 (X_z ,\mu_z )\widehat{\otimes} L^1 (Y,\nu )
\to L^1 (X_z ,\mu_z )$ for $z\in Z$. As is easy to check using approximations in the
algebraic tensor product, for every $f\in L^1 (X \times Y,\mu\times\nu ) =
L^1 (X,\mu )\widehat{\otimes} L^1 (Y,\nu )$ there is a
full-measure set of $z\in Z$ for which
$E^{\sB_z} (f|_{X_z} )(x) = E^{\sB} (f)(x)$ for $\mu_z$-a.e.\ $x\in X_z$.
For each $j=1, \dots ,r$
set $C_j = \{ x\in X : E^{\sB} (\boldsymbol{1}_{Q_j} )(x) > 0 \}$, which is defined up to
a set of $\mu$-measure zero and hence can be assumed to satisfy the condition that for
every $i=1,\dots ,r$ it is disjoint from $A_i \times Y$ if and only if $Q_j$ is.
Then $\{ C_j : j=1, \dots ,r \}$ is an $\oA$-admissible Borel cover of $X$. Putting
$\cP = \{ C_j \setminus\bigcup_{d=1}^{j-1} C_d : j=1, \dots ,r \}$ we obtain an
$\oA$-admissible measurable partition of $X$.

Now let $z\in Z$.
%We write $\oA_z$ for the tuple
%$(A_1 \cap X_z , \dots , A_k \cap X_z )$ of measurable subsets of $X_z$.
Denote by $\sR$ the
relative Pinsker $\sigma$-algebra of $T_z \times S$ with respect to $\sY$,
i.e., the $\sigma$-algebra generated by all measurable partitions $\cR$ of
$X_z \times Y$ such that $\hmeas_{\mu\times\nu} (T_z \times S, \cR | \sY ) = 0$.
In the $\omega$-a.e.\ situation that $R_z$ is ergodic we have
$\sR = \sP_{T_z} \otimes\sY$ by Theorem~4.10 of \cite{EMAGA}. From the
discussion in the previous paragraph we see that if $z$ is assumed to belong to a
certain set of full measure then for each $j=1,\dots ,r$ the sets
$C_j \cap X_z$ and $\{ x\in X_z : E^{\sB_z} (\boldsymbol{1}_{Q_j \cap X_z} )(x) > 0 \}$
coincide up to a set of $\mu_z$-measure zero. In this case, setting
$\cP_z = \{ P \cap X_z : P\in \cP \}$ we obtain a partition of $X_z$ which is
$\sP_{T_z}$-measurable and hence satisfies
$\hmeas_{\mu_z} (T_z ,\cP_z ) = 0$. It follows using the ergodic decomposition of entropy 
that $\hmeas_\mu (T,\cP ) = \int_Z \hmeas_\mu (T_z ,\cP_z )\, d\omega (z) = 0$,
contradicting our hypothesis. Therefore we must have
$\hmeas_{\mu_z \times\nu} (T_z \times S,\cQ_z | \sY ) > 0$ for all $z$ in a
set $W \subseteq Z$ of nonzero measure.

For every $z$ in a subset of $W$ with the same measure as $W$ the action
$R_z$ is ergodic and free, in which case we can apply Theorem~2.6 of
\cite{EMAGA} along with the fact that $\hat{R}$ has zero entropy to obtain
\[ \hmeas_{\mu_z \times\nu} (R_z ,\cQ_z ) =
\hmeas_{\mu_z \times\nu} (R_z ,\cQ_z | \sY ) =
\hmeas_{\mu_z \times\nu} (T_z \times S,\cQ_z | \sY ) > 0 . \]
The ergodic decomposition of entropy then yields
\[ \hmeas_{\mu\times\nu} (R,\cQ ) = \int_Z \hmeas_{\mu_z \times\nu} (R_z ,\cQ_z )\,
d\omega (z) > 0 . \]
It follows by Theorem~4.6 of \cite{LVRA} that the infimum $c$ of
$\hmeas_{\mu\times\nu} (R,\cQ )$ over all $\oB$-admissible finite measurable
partitions $\cQ$ of $X$ is nonzero.

Denote by $\cV$ the measurable cover
$\{ A_1^\comp \times Y , \dots , A_k^\comp \times Y \}$ of $X\times Y$.
Suppose we are given a $\oB$-admissible finite measurable partition $\cQ$ of $X\times Y$.
Applying the ergodic decomposition of entropy,
Theorem~2.6 of \cite{EMAGA}, and the fact that $\hat{R}$ has zero entropy we get
\begin{align*}
\hmeas_{\mu\times\nu} (T\times S,\cQ )
&= \int_Z \hmeas_{\mu_z \times\nu} (T_z \times S,\cQ_z )\, d\omega (z) \\
&\geq \int_Z \hmeas_{\mu_z \times\nu} (T_z \times S,\cQ_z | \sY )\, d\omega (z) \\
&= \int_Z \hmeas_{\mu_z \times\nu} (R_z ,\cQ_z | \sY )\, d\omega (z) \\
&= \int_Z \hmeas_{\mu_z \times\nu} (R_z ,\cQ_z )\, d\omega (z) \\
&= \hmeas_{\mu\times\nu} (R,\cQ ) \\
&\geq c .
\end{align*}
Therefore $\hmeas^+_{\mu\times\nu} (T\times S,\cV ) \geq c > 0$, and since $T\times S$ is free
it follows by Lemma~\ref{L-compare free} that $\lwh_{\comb ,\mu\times\nu} (T\times S,\cV ) > 0$.
As we clearly have
$\lwh_{\comb ,\mu} (T,\cU ) \geq \lwh_{\comb ,\mu\times\nu} (T\times S,\cV )$, this
establishes the lemma.
\end{proof}

We remark that, in the last paragraph of the above proof,
if $\cQ$ is of the form $\{ P\times Y : P\in\cP \}$
for some finite $\oA$-admissible Borel partition $\cP$ of $X$, then
$\hmeas_\mu (T,\cP ) = \hmeas_{\mu\times\nu} (T\times S,\cQ )$, in which case the
display shows that $\hmeas^+_\mu (T,\cU ) \geq c > 0$.

In order to reduce the general case of discrete amenable groups to
the case of countable ones, we shall need Lemma~\ref{L-min} below.
For this we need the machinery of quasi-tiling developed by
Ornstein and Weiss. The following lemma is contained in the proof
of Theorem 6 in \cite{OW}.

\begin{lemma}\label{quasi-tiling}
Given $1>\varepsilon>0$, if $F_1\subseteq F_2\subseteq \dots
\subseteq F_k$ are nonempty finite subsets of $G$ such that
$F_{i+1}$ is $(F_iF^{-1}_i, \eta_i)$-invariant,
$\eta_i|F_iF^{-1}_i|\le \frac{\varepsilon^2}{4}$ for $i=1, 2,\dots, k-1$, 
and $(1-\frac{\varepsilon}{2})^k<\varepsilon$, then for any
$(F_k, \frac{\varepsilon^2}{4})$-invariant finite nonempty subset
$F$ of $G$ there are translates $\{F_ic_{ij}\}_{i, j}$ contained
in $F$ and subsets $E_{ij}\subseteq F_ic_{ij}$ such that
$E_{ij}\cap E_{i'j'}=\emptyset$ for all $(i, j)\neq (i', j')$, 
$|E_{ij}|/|F_ic_{ij}|\ge 1-\varepsilon$ for all $(i, j)$, and
$|\bigcup_{ij}F_ic_{ij}|/|F|\ge 1-\varepsilon$.
\end{lemma}

The following lemma is a direct consequence of
Lemma~\ref{quasi-tiling}. For any $\varphi$ satisfying the
conditions below, by Proposition 3.22 in \cite{Ind},
$\frac{\varphi(F)}{|F|}$ converges as $F$ becomes more and more
invariant. Note that every subgroup of $G$ is amenable
\cite[Prop.\ 1.12]{Pat}.

\begin{lemma}\label{L-min}
If $\varphi$ is a real-valued function which is defined on the set
of finite subsets of $G$ and satisfies
\begin{enumerate}
\item $0\le \varphi(A)<+\infty$ and $\varphi(\emptyset)=0$,

\item $\varphi(A)\le \varphi(B)$ for all $A\subseteq B$,

\item $\varphi(As)=\varphi(A)$ for all finite $A\subseteq G$ and
$s\in G$,

\item $\varphi(A\cup B)\le \varphi(A)+\varphi(B)$ if $A\cap
B=\emptyset$,
\end{enumerate}
then the limit of $\frac{\varphi(F)}{|F|}$ as $F$ becomes more and
more invariant in $G$ is the minimum of the corresponding limits
of $\frac{\varphi(F)}{|F|}$ as $F$ becomes more and more invariant
in $H$ for $H$ running over the countable subgroups of $G$.
\end{lemma}

\begin{theorem}\label{T-IE E}
For every $k\geq 2$, a nondiagonal tuple in $X^k$ is a $\mu$-IE-tuple if and only if it
is a $\mu$-entropy tuple.
\end{theorem}

\begin{proof}
The fact that a nondiagonal $\mu$-IE-tuple is a $\mu$-entropy
tuple follows from Lemma~\ref{L-positive}. In the case that $X$ is
metrizable and $G$ is countably infinite, Lemmas~\ref{L-positive
inf} and \ref{L-positive} combine to show that a $\mu$-entropy
tuple is a $\mu$-IE-tuple. Suppose now that $X$ is arbitrary. When
$G$ is finite, it is easily seen that the nondiagonal
$\mu$-IE-tuples and $\mu$-entropy tuples are both precisely the
nondiagonal tuples in $\supp (\mu )^k$. When $G$ is countably
infinite, write $X$ as a projective limit of a net of metrizable
spaces $X_j$ equipped with compatible $G$-actions and induced
Borel probability measures $\mu_j$. Then by Proposition~\ref{P-IE
basic}(5) the $\mu$-IE-tuples are the projective limits of the
$\mu_j$-IE-tuples. Since the image of a measure entropy tuple
under a factor map is clearly again a measure entropy tuple as
long as its image is nondiagonal, we conclude from the metrizable
case that every $\mu$-entropy tuple is a $\mu$-IE-tuple. Finally,
when $G$ is uncountably infinite, it follows from Lemma~\ref{L-min}
that the set of $\mu$-entropy tuples for $(X,G)$ is equal to the
intersection over the countable
subgroups $G'$ of $G$ of the sets consisting of the $\mu$-entropy
tuples for $(X,G' )$. It is also easily verified that the set of
$\mu$-IE-tuples for $(X,G)$ contains the intersection over the
countable subgroups $G'$ of $G$ of the sets consisting of
the $\mu$-IE-entropy tuples for $(X,G' )$. We thus obtain the
result.
\end{proof}

To prove the product formula for $\mu$-IE-tuples we will use
the Pinsker von Neumann algebra $\mathfrak{P}_X$, i.e., the $G$-invariant
von Neumann subalgebra of $\LX$ corresponding to the Pinsker $\sigma$-algebra
(see the beginning of the next section).
Denote by $E_X$ the conditional expectation $L^{\infty}(X, \mu)\rightarrow
\mathfrak{P}_X$. 
The following lemma appeared as Lemma~4.3 in \cite{LVRA}. Note
that the assumptions in \cite{LVRA} that $X$ is metrizable and $G=\Zb$ 
are not needed here.

\begin{lemma}\label{L-Pinsker}
Let $\cU=\{U_1, \dots, U_k\}$ be a Borel cover of $X$. 
Then $\prod^k_{i=1}E_X(\chi_{U^c_i})\neq 0$ if and only if $\hmeas_{\mu}(\cP)>0$ for 
every finite Borel partition $\cP$ finer than $\cU$ as a cover.
\end{lemma}

Combining Lemma~\ref{L-Pinsker}, Proposition~\ref{P-IE basic}(3), and
Theorem~\ref{T-IE E}, we obtain the following charaterization
of $\mu$-IE tuples. 

\begin{lemma}\label{L-measure IE char}
A tuple $\ox=(x_1, \dots, x_k)\in X^k$ is a $\mu$-IE tuple if and only
if for any Borel neighbourhoods $U_1 , \dots , U_k$ of $x_1 , \dots , x_k$, respectively,
one has $\prod^k_{i=1}E_X(\chi_{U_i})\neq 0$.
\end{lemma}

\begin{theorem}\label{T-IE product}
Let $(Y,G)$ be another topological $G$-system and $\nu$ a $G$-invariant
Borel probability measure on $Y$. Then for all $k\geq 1$ we have
$\IE_{\mu\times\nu}^k (X\times Y) = \IE_\mu^k (X) \times \IE_\nu^k (Y)$.
\end{theorem}

\begin{proof} 
By Proposition~\ref{P-IE basic}(5) we have 
$\IE_{\mu\times\nu}^k (X\times Y) \subseteq 
\IE_\mu^k (X) \times \IE_\nu^k (Y)$. Thus we just need
to prove 
$\IE_\mu^k (X) \times \IE_\nu^k (Y)\subseteq \IE_{\mu\times\nu}^k (X\times Y)$.

Assume first that both $X$ and $Y$ are metrizable and
$G$ is countable. Then $\mathfrak{P}_{X\times Y}=\mathfrak{P}_X\otimes
\mathfrak{P}_{Y}$ \cite[Theorem 0.4(3)]{Dan} (see also \cite[Theorem 4]{ETWP}
for the ergodic case) and hence $E_{X\times Y}(f\otimes g)=
E_X(f)\otimes E_Y(g)$ for any $f\in L^{\infty}(X, \mu)$ and 
$g\in L^{\infty}(Y, \nu)$. Now the desired inclusion follows from 
Lemma~\ref{L-measure IE char}. 

The proof for the general case follows the argument in the proof of
Theorem~\ref{T-IE E}. 
\end{proof}

In the case $G=\Zb$, the product formula for measure entropy pairs was 
established in \cite{mu}, while for general measure entropy tuples 
it is implicit in Theorem~8.1 of \cite{LVRA}, whose
proof we have essentially followed here granted the general tensor product formula
for Pinsker von Neumann algebras. Notice that our 
IE-tuple viewpoint results in a particularly simple formula.

%%%%%%%%%%%%%%%%%%%%%%%%%%%%%%%%%%%%%%%%%%%%%%%%%%%%%%%%%%%%%%%%%%%%%%%%%%%%%

\section{Combinatorial independence and the Pinsker algebra}\label{S-Pinsker}

Continuing within the realm of entropy, we will assume throughout the
section that $(X,G)$ is a topological dynamical system with $G$ amenable
%$\sB$ is the Borel $\sigma$-algebra of $X$,
and $\mu$ is a $G$-invariant Borel probability measure on $X$.
Recall that the {\it Pinsker $\sigma$-algebra} is the $G$-invariant
$\sigma$-subalge-\linebreak bra of $\sB$ generated by all finite Borel partitions
of $X$ with zero entropy (or, equivalently, all two-element Borel partitions
of $X$ with zero entropy), and it defines the largest factor of the system with
zero entropy (see Chapter~18 of \cite{ETJ}).
The corresponding $G$-invariant von Neumann subalgebra of $\LX$
will be denoted by $\mathfrak{P}_X$ and referred to as the
{\em Pinsker von Neumann algebra}. In Theorem~\ref{T-Pinsker}
we will give various local descriptions of the Pinsker von Neumann algebra
in terms of combinatorial independence, $\ell_1$ geometry, and
c.p.\ approximation entropy.

The notion of c.p.\ (completely positive) approximation entropy was introduced by
Voiculescu in \cite{Voi} for $^*$-automorphisms of hyperfinite von Neumann
algebras preserving a faithful normal state.
We will formulate here a version
of the definition for amenable acting groups. So let $M$ be a
von Neumann algebra, $\sigma$ a faithful normal state on $M$, and $\beta$
a $\sigma$-preserving action of the discrete amenable group $G$ on $M$
by $^*$-automorphisms. For a finite set $\Upsilon\subseteq M$ and
$\delta > 0$ we write $\CPA_\sigma (\Upsilon , \delta )$ for the set of all
triples $(\varphi , \psi , B)$ where $B$ is a finite-dimensional
$C^*$-algebra and $\varphi : M\to B$ and $\psi : B\to M$ are unital
completely positive maps such that $\sigma\circ\psi\circ\varphi = \sigma$
and $\| (\psi\circ\varphi )(a) - a \|_\sigma < \delta$ for all
$a\in\Upsilon$. We then set
\[ \rcp_\sigma (\Upsilon , \delta ) = \inf \{ \rank\, B :
(\varphi , \psi , B)\in\CPA_\sigma (\Upsilon , \delta ) \} \]
if the set on the right is nonempty, which is always
the case if $M$ is commutative or hyperfinite. Otherwise we put
$\rcp_\sigma (\Upsilon , \delta ) = \infty$.
Write $\hcpa_\sigma (\beta , \Upsilon , \delta )$ for the limit
supremum of $\frac{1}{|F|} \ln \rcp_\sigma (\bigcup_{s\in F} \alpha_s (\Upsilon ), \delta )$
as $F$ becomes more and more invariant, and define
\begin{align*}
\hcpa_\sigma (\beta , \Upsilon ) &= \sup_{\delta > 0}
\hcpa_\sigma (\beta , \Upsilon , \delta ) , \\
\hcpa_\sigma (\beta ) &= \sup_{\Upsilon}
\hcpa_\sigma (\beta , \Upsilon ) ,
\end{align*}
where the last supremum is taken over all finite subsets $\Upsilon$ of $M$.
We refer to $\hcpa_\sigma (\beta , \Upsilon )$ as the
{\em c.p.\ approximation entropy} of $\beta$. When $G=\Zb$ and $M$ is commutative and has
separable predual, this coincides with Voiculescu's original definition by the arguments
leading to Corollary~3.8 in \cite{Voi}.

\begin{question} 
Does the above definition always coincide with Voiculescu's when $G=\Zb$?
\end{question}

By Corollary~3.8 in \cite{Voi}, when $X$ is metrizable, $G = \Zb$, and
the action is ergodic, the c.p.\ approximation entropy of the induced action $\alpha$
on $L^\infty (X,\mu )$ agrees with the measure entropy $h_\mu (X)$. The
arguments also work for general amenable $G$. It follows using the ergodic
decomposition of entropy (see the paragraph before Lemma~\ref{L-positive inf}) 
that when $X$ is metrizable
the Pinsker von Neumann algebra is the largest $G$-invariant von Neumann subalgebra
of $\LX$ on which the c.p.\ approximation entropy is zero.

We next define geometric analogues of upper and lower measure independence
density from Section~\ref{S-IE}.
Let $f\in\LX$. Let $p$ be a projection in $\LX$ and let $\lambda\geq 1$. We say that a
set $J\subseteq G$ is an {\em $\ell_1$-$\lambda$-isomorphism set for $f$ relative to $p$}
if $\{ p\alpha^i (f) : i\in J \}$ is $\lambda$-equivalent to the standard basis
of $\ell_1^J$.
For $\delta > 0$ denote by $\mathscr{P} (\mu ,\delta )$ the set of projections $p\in\LX$
such that $\mu (p) \geq 1-\delta$. For every finite subset $F$ of $G$,
$\lambda\geq 1$, and $\delta > 0$ we define
\[ \varphi_{f,\lambda ,\delta} (F) = \min_{p\in\mathscr{P} (\mu ,\delta )}
\max \big\{ |F\cap J| : J \text{ is an } \ell_1 \text{-} \lambda
\text{-isomorphism set for } f \text{ relative to } p \big\} . \]
Write $\upind_\mu (f,\lambda ,\delta )$ for the limit supremum of
$\frac{1}{|F|} \varphi_{f,\lambda ,\delta} (F)$ as $F$ becomes more and more invariant,
and $\lwind_\mu (f,\lambda ,\delta )$ for the corresponding limit infimum.
Set $\upind_\mu (f,\lambda ) = \sup_{\delta > 0} \upind_\mu (f,\lambda ,\delta )$
and $\lwind_\mu (f,\lambda ) = \sup_{\delta > 0} \lwind_\mu (f,\lambda ,\delta )$.
Finally, we define $\upind_\mu (f) = \sup_{\lambda\geq 1} \upind_\mu (f,\lambda )$
and $\lwind_\mu (f) = \sup_{\lambda\geq 1} \lwind_\mu (f,\lambda )$, and refer to these
quantities respectively as the {\em upper $\mu$-$\ell_1$-isomorphism density} and
{\em lower $\mu$-$\ell_1$-isomorphism density} of $f$.
On the topological side, for each $\lambda\geq 1$ the limit of
\[ \frac{1}{|F|} \max \big\{ |F\cap J| :
\{ \alpha^i (f) : i\in J \} \text{ is } \lambda
\text{-equivalent to the standard basis of } \ell_1^J \big\} \]
as $F$ becomes more and more invariant exists (see the end of Section~3 in \cite{Ind}),
and we refer to the supremum of these limits over all $\lambda\geq 1$ as the
{\it $\ell_1$-isomorphism density} of $f$.

Glasner and Weiss proved the next lemma for the real scalar case
\cite[Lemma 2.3]{GW}. The complex scalar version
follows by considering the map
$E\rightarrow (\ell^n_\infty )_{\Rb} \oplus_\infty (\ell^n_\infty )_{\Rb} =
(\ell^{2n}_\infty )_{\Rb}$
sending each $v\in E\subseteq \ell^n_{\infty}$ to the pair consisting of its real
and imaginary parts.

\begin{lemma}\label{L-GW}
For all $b > 0$ and $\delta > 0$ there exist
$c > 0$ and $\varepsilon > 0$ such that, for all sufficiently large $n$, if
$E$ is a subset of the unit ball of $\ell^n_{\infty}$ which is $\delta$-separated and
$|E|\ge e^{bn}$, then there are a $t\in [-1, 1]$ and a set
$J\subseteq \{1, 2, \dots, n\}$ for which
\begin{enumerate}
\item $|J|\ge cn$, and

\item either for every $\sigma\in \{0, 1\}^J$ there is a
$v\in E$ such that for all $j\in J$
\begin{alignat*}{2}
\re (v(j)) &\ge t+\varepsilon \hspace*{6mm} & &\text{if } \sigma(j)=1, \text{ and} \\
\re (v(j)) &\le t-\varepsilon & &\text{if } \sigma(j)=0,
\end{alignat*}
or for every $\sigma\in \{ 0,1 \}^J$ there is a $v\in E$
such that for all $j\in J$ the above holds with $\re (v(j))$
replaced by $\im (v(j))$.
\end{enumerate}
\end{lemma}

%\begin{lemma}\label{L-l1 to indep}
%Let $\varepsilon>0$. Then there exist positive constants $M, d$ and
%$\delta$ such that whenever $X$ is a compact space, $\Omega$ is a
%finite subset of $B_1(C(X))$ with $|\Omega|\ge M$,
%and the linear map $\gamma:\ell^{\Omega}_1\rightarrow
%C(X)$ sending the standard basis elements of $\ell^{\Omega}_1$ to
%elements of $\Omega$ is an isomorphism with
%$\|\gamma^{-1}\|^{-1}>2\varepsilon$, there exist a
%$t\in [-1, 1]$ and a set $J\subseteq \Omega$ such that $|J|\ge d|\Omega|$
%and at least one of the two finite sequences of pairs
%$\{(A_{\re (f)}, B_{\re (f)})\}_{f\in J}$ and
%$\{(A_{\im (f)}, B_{\im (f)})\}_{f\in J}$ is independent,
%%the finite sequence of pairs $(A_f, B_f)$ for $f\in J$ is independent,
%where
%\begin{gather*} %\label{sets:eq}
%A_h:=h^{-1}([t+\delta, 1]), \quad B_h:=h^{-1}([-1, t-\delta)]),
%\end{gather*}
%for $h\in C(X)$.
%\end{lemma}

The following is a consequence of Lemma~3.6 in \cite{Ind}.

\begin{lemma}\label{L-smaller}
There exists a $c > 0$ such that whenever $I$ is a finite set
and $A_{i,1}$, $A_{i,2}$, and $B_i$ for $i\in I$ are subsets of a given set such that
the collection $\{ (A_{i,1} \cup A_{i,2} , B_i ) : i\in I \}$ is independent,
there are a set $J\subseteq I$ with $|J| \geq c|I|$ and a $j\in \{ 1,2\}$ for which
the collection  $\{ (A_{i,j} , B_i ) : i\in J \}$ is independent.
\end{lemma}

\begin{lemma}\label{L-l1 to indep}
For every $\delta > 0$ there exist $c >0$ and $\varepsilon > 0$ such that, for every
compact Hausdorff space $Y$ and finite subset $\Theta$ of the unit ball of $C(Y)$
of sufficiently large cardinality, if the linear map
$\gamma : \ell_1^\Theta \to C(Y)$ sending the standard basis of $\ell_1^\Theta$
to $\Theta$ is an isomorphism with $\| \gamma^{-1} \|^{-1} \geq \delta$,
then there exist closed disks $B_1 , B_2 \subseteq\Cb$ of diameter at most
$\varepsilon /6$ with ${\rm dist} (B_1 , B_2 ) \geq\varepsilon$
and an $I\subseteq\Theta$ with $|I| \geq c|\Theta |$ such that the collection
$\{ (f^{-1} (B_1 ) , f^{-1} (B_2 )): f\in I \}$ is independent.
%in the sense that choosing any one of the pairs for each $f\in\Theta$
%yields a family of sets with nonempty intersection.
\end{lemma}

\begin{proof}
Let $\delta > 0$.
Define a pseudometric $d_{\Theta}$ on $Y$ by
\[ d_{\Theta} (x, y)=\sup_{f\in\Theta}|f(x)-f(y)| \]
and pick a maximal $(\delta /4)$-separated subset $Z$ of $Y$. Then the
open balls $B(z, \delta /4)$ with radius $\delta /4$ and centre
$z$ for $z\in Z$ cover $Y$. A standard partition of unity argument
(see the proof of Proposition~4.8 in \cite{Voi}) yields the bound
$\rc(\Theta, \delta /2)\le |Z|$ for the contractive
$(\delta /2)$-rank of $\Theta$ as defined in \cite{DEBS}. By
Lemma~3.2 of \cite{DEBS} we have $\ln \rc(\Theta, \delta /2)\ge
|\Theta|a\|\gamma\|^{-2}(\|\gamma^{-1}\|^{-1}-\delta /2)^2$ for some universal
constant $a>0$. Thus $|Z|\ge
e^{|\Theta|a\|\gamma\|^{-2}(\|\gamma^{-1}\|^{-1}-\delta /2)^2}\ge
e^{|\Theta|a\delta^2 /4}$. Evaluation of the functions in $\Theta$ on the points of $Y$
yields a map $\psi$ from $Y$ to the unit ball of $\ell^{\Theta}_{\infty}$ such that
$\psi (Z)$ is $(\delta /4)$-separated. By Lemma~\ref{L-GW} there are
$c > 0$ and $\varepsilon > 0$ depending only on $a$ and $\delta$ such that
there exist closed disks $B_1$ and $B_2$ contained in the unit disk of $\Cb$
with ${\rm dist} (B_1 , B_2 ) \geq 4\varepsilon /3$
and an $I\subseteq\Theta$ with $|I| \geq c|\Theta |$ such that the collection
$\{ (f^{-1} (B_1 ) , f^{-1} (B_2 )): f\in I \}$ is independent. Now for some
$N\in\Nb$ depending on $\varepsilon$ we can cover each of $B_1$ and $B_2$
with $N$ disks of diameter at most $\varepsilon /6$. By repeated application
of Lemma~\ref{L-smaller} we can then replace each of $B_1$ and $B_2$ with one
of the smaller disks to obtain the result (with a smaller $c$).
\end{proof}

\begin{lemma}\label{L-rcp lower bd}
Let $\delta > 0$ and $\lambda > 0$.
Let $\Omega = \{ f_1 , \dots , f_n \}$ be a subset of the unit ball of
$L^\infty (X,\mu )$ and suppose that for all $g_1 , \dots , g_n$ in the unit ball of
$L^\infty (X,\mu )$ with $\max_{1\leq i\leq n} \| g_i - f_i \|_\mu < \delta$ there exists
an $I \subseteq \{ 1,\dots ,n \}$ of cardinality at least $dn$ for which the linear map
$\ell_1^I \to \spn \{ g_i : i\in I \}$ sending the standard basis element with index $i\in I$
to $g_i$ has an inverse with norm at most $\lambda$.
Then
\[ \ln\rcp_\mu (\Omega , \delta ) \geq an \]
for some constant $a>0$ which depends only on $\lambda$.
\end{lemma}

\begin{proof}
Let $(\varphi , \psi , B)\in \CPA_\mu (\Omega , \delta )$. Then there exists an
$I \subseteq \{ 1,\dots ,n \}$ of cardinality at least $dn$ for which the linear map
$\ell_1^I \to \spn \{ (\psi\circ\varphi )(f_i ) : i\in I \}$ sending the standard basis
element with index $i\in I$ to $g_i$ has an inverse with norm at most $\lambda$.
It follows using the operator norm contractivity of $\varphi$ and $\psi$
that for any scalars $c_i$ for $i\in I$ we have
\[ \bigg\| \sum_{i\in I} c_i \varphi (f_i ) \bigg\| \geq
\bigg\| \sum_{i\in I} c_i (\psi\circ\varphi )(f_i ) \bigg\|
\geq \lambda^{-1} \sum_{s\in I} | c_i | , \]
so that the subset
$\{ \varphi (f_i ) : i\in I \}$ of $B$ is $\lambda$-equivalent
to the standard basis of $\ell_1^I$. Lemma~3.1 of \cite{EID} then guarantees
the existence of a constant $a>0$ depending only on $\lambda$ such that
$\ln\rank (B) \geq an$, yielding the result.
\end{proof}

\begin{lemma}\label{L-rcp upper bd}
Let $\delta > 0$.
Let $\Omega = \{ f_1 , \dots , f_n \}$ be a subset of the unit ball of
$L^\infty (X,\mu )$ and for each $i=1,\dots ,n$ let $\cP_i$ be a finite Borel partition of
$X$ such that $\esssup_{x,y\in P} |f_i (x) - f_i (y)| < \delta$ for every $P\in\cP_i$.
Suppose that $\Hmeas (\cP ) \leq n\delta^2$ where $\cP = \bigvee_{i=1}^n \cP_i$. Then
\[ \ln\rcp_\mu (\Omega , \sqrt{\delta^2 + 4\delta} ) \leq 2n\delta \]
if $n$ is sufficiently large as a function of $\delta$.
\end{lemma}

\begin{proof}
For a finite Borel partition $\cQ$ of $X$ we write $I(\cQ )$ for the information function
$-\sum_{Q\in\cQ} \boldsymbol{1}_Q \ln \mu (Q)$. Then $\Hmeas (\cP ) = \int_X I(\cP )\, d\mu$,
and so by our assumption the set $D$ on which the nonnegative function $I(\cP )/n$ takes
values less than $\delta$ has measure at least $1-\delta$. Then $\mu (P) \geq e^{-n\delta}$
for all $P\in\cP$ such that $\mu (P\cap D) \neq\emptyset$. Let $B$ be the linear span of
$\{ \boldsymbol{1}_{P\cap D} : P\in\cP \text{ and } \mu (P\cap D) \neq\emptyset \} \cup
\{ \boldsymbol{1}_{D^\comp} \}$.
Then $B$ is a unital $^*$-subalgebra of $L^\infty (X,\mu )$ and $\dim B \leq e^{n\delta} + 1$.
Taking the $\mu$-preserving conditional expectation $\varphi : L^\infty (X,\mu ) \to B$ and
the inclusion $\psi : B\to L^\infty (X,\mu )$ it is readily checked that
$(\varphi , \psi , B)\in \CPA_\mu (\Omega , \sqrt{\delta^2 + 4\delta} )$ so that
$\rcp_\mu (\Omega , \sqrt{\delta^2 + 4\delta} ) \leq e^{n\delta} + 1$, from which the desired
conclusion follows.
\end{proof}

Denote by $\Omega$ the pure state space of $\LX$ equipped with the
relative weak$^*$ topology, under which it is compact. When appropriate we will view elements
of $\LX$ as continuous functions on $\Omega$.
The action $\alpha$ of $G$ on $L^\infty (X,\mu )$ gives rise to a topological dynamical
system $(\Omega , G)$ with the action of $G$ defined by
$(s,\sigma )\mapsto\sigma\circ\alpha_{s^{-1}}$. Since $\mu$ defines a state on $\LX$
it gives rise to a $G$-invariant Borel probability measure on $\Omega$, which we will
also denote by $\mu$. For a projection
$p\in L^\infty (X,\mu )$ we write $\Omega_p$ for the clopen subset of
$\Omega$ supporting $p$.

%Unless otherwise stated, the norm on $\LX$ will be the operator norm.

\begin{theorem}\label{T-Pinsker}
Let $f\in\LX$. Let $\{ F_n \}_{n\in\Lambda}$
be a F{\o}lner net in $G$. Then the following are equivalent:
\begin{enumerate}
\item $f\notin\mathfrak{P}_X$,

\item there is a $\mu$-IE-pair $(\sigma_1 , \sigma_2 )\in \Omega\times\Omega$
such that $f(\sigma_1 ) \neq f(\sigma_2 )$,

\item there are $d>0$, $\delta > 0$, and $\lambda > 0$ such that,
for all $n$ greater than some $n_0 \in\Lambda$, 
whenever $g_s$ for $s\in F_n$ are elements of $\LX$
satisfying $\| g_s - \alpha_s (f) \|_\mu < \delta$
for every $s\in F_n$ there exists an $I \subseteq F_n$
of cardinality at least $d|F_n |$ for which the linear map $\ell_1^I \to
\spn \{ g_s : s\in I \}$ sending the standard basis element with index $s\in I$
to $g_s$ has an inverse with norm at most $\lambda$,

\item the same as {\rm (3)} with ``for all $n$ greater than some $n_0 \in\Lambda$'' 
replaced by ``for all $n$ in a cofinal subset of $\Lambda$'',

\item $\lwind_\mu (f) > 0$,

\item $\upind_\mu (f) > 0$,

\item $\hcpa_\mu (\alpha , \{ f \} ) > 0$,

\item $\hcpa_\mu (\beta ) > 0$ for the restriction $\beta$ of $\alpha$ to the
von Neumann subalgebra of $\LX$ dynamically generated by $f$.
\end{enumerate}
When the action is ergodic and either $X$ is metrizable or $G$ is countable, we can add:
\begin{enumerate}
\item[(9)] there is a $\delta > 0$ such that every $g\in\LX$ satisfying
$\| g - f \|_\mu < \delta$ has positive
$\ell_1$-isomorphism density with respect to the operator norm.
\end{enumerate}
When $f\in C(X)$ we can add:
\begin{enumerate}
\item[(10)] $f\notin C(Y)$ whenever $\pi : X\to Y$ is a topological $G$-factor map
such that $h_{\pi_* (\mu )} (Y) = 0$,

\item[(11)] there is a $\mu$-IE-pair $(x_1 , x_2 )\in X\times X$
such that $f(x_1 ) \neq f(x_2 )$.
\end{enumerate}
\end{theorem}

\begin{proof}
(1)$\Rightarrow$(2). Since the $\alpha$-invariant von Neumann subalgebra of $\LX$
generated by $f$ is also dynamically generated by the set of spectral projections
of $f$ over closed subsets of the complex plane, we can
find a clopen set $Z\subseteq\Omega$ corresponding to a spectral projection
of $f$ over $A$ for some set $A\subseteq\Cb$
such that the two-element clopen partition $\cZ = \{ Z , Z^\comp \}$ satisfies
$\hmeas_\mu (\Omega , \cZ ) > 0$. Using Lemma~\ref{L-pe to pd}
we can find a closed set $B\subseteq\Cb$ with $B\cap A = \emptyset$ such that the pair
$(Z,Z' )$ has positive $\mu$-independence density, where $Z'$ is the subset of $\Omega$
supporting the spectral projection of $f$ over $B$. By Proposition~\ref{P-IE basic}(1) there
is a $\mu$-IE-pair $(\sigma_1 , \sigma_2 )\in\Omega\times\Omega$ such that $\sigma_1 \in Z$
and $\sigma_2 \in Z'$. Then $f(\sigma_1 ) \in A$ while $f(\sigma_2 ) \in B$,
establishing (2).

(2)$\Rightarrow$(3). Let $(\sigma_1 ,\sigma_2 ) \in \Omega\times\Omega$ be a
$\mu$-IE-pair such that $f(\sigma_1 ) \neq f(\sigma_2 )$. Choose disjoint closed disks
$B_1 , B_2 \subseteq\Cb$ such that ${\rm diam}(B_1)={\rm diam}(B_2)\le
\frac{1}{10}{\rm dist}(B_1, B_2)$, $f(\sigma_1 )\in \interior (B_1 )$, and
$f(\sigma_2 )\in \interior (B_2 )$ and set $\varepsilon = \frac{1}{10}{\rm dist} (B_1 , B_2 )$. Choose
clopen neighbourhoods $A_1$ and $A_2$ of $\sigma_1$ and $\sigma_2$, respectively, such that
$f(A_1 ) \subseteq B_1$ and $f(A_2 ) \subseteq B_2$. Write $\oA$ for
the pair $(A_1 , A_2 )$. Since
$(\sigma_1 ,\sigma_2 )$ is a $\mu$-IE-pair there exists by Proposition~\ref{P-alternative}
a $\delta > 0$ such that $\lwind_{\mu}' (\oA , \delta ) > 0$. Take an $\eta > 0$ such that
whenever $h$ is an element of $L^\infty (X,\mu )$ for which
$\| h \|_\mu < \eta$ the set
$\{ x\in X : |h(x)| \leq \varepsilon \}$ has measure at least $1-\delta$.

Now let $n\in\Lambda$ and suppose that we are given $g_s \in \LX$ for $s\in F_n$
such that $\| g_s - \alpha_s (f) \|_\mu < \eta$
for every $s\in F_n$. For each $s\in F_n$ set
$D_s = \{ \sigma\in\Omega : |g_s (\sigma ) - \alpha_s (f)(\sigma )| \leq\varepsilon \}$,
which has measure at least $1 - \delta$ by our choice of $\eta$, and for
$s\in G\setminus F_n$ set $D_s = \Omega$.
Put $d = \lwind_{\mu}' (\oA , \delta ) /2$.
Assuming that $n > n_0$ for a suitable $n_0 \in\Lambda$,
there exist an independence set $I\subseteq F_n$
for $\oA$ relative to the map $s\mapsto D_s$ such that
$|I| \geq d |F_n |$.
The standard Rosenthal-Dor argument \cite{Dor} then shows that the linear map
$\ell_1^I \to\spn \{ g_s : s\in I \}$ sending the standard basis element with index
$s\in I$ to $g_s$ has an inverse with norm at most $\varepsilon^{-1}$,
yielding (3).

(3)$\Rightarrow$(4). Trivial.

(4)$\Rightarrow$(7). Apply Lemma~\ref{L-rcp lower bd}.

(3)$\Rightarrow$(5). We may assume that $\| f \| = 1$.
Let $d$, $\delta$, and $\lambda$ be as given by (3).
Then for any $p\in\mathscr{P} (\mu , \delta^2 )$ and $s\in G$ we have
$\| p\alpha_s (f) - \alpha_s (f) \|_\mu \leq \| p-1 \|_\mu \| f \|
\leq \delta $.
It follows that $\varphi_{f,\lambda ,\delta^2 } (F_n ) \geq d|F_n |$ for
every $n\in\Nb$, and hence $\lwind_\mu (f) \geq \lwind_\mu (f,\lambda ,\delta^2 )
\geq d > 0$.

(5)$\Rightarrow$(6). Trivial.

(6)$\Rightarrow$(4). We may assume that $G$ is infinite and $\| f \| = 1$. By (6)
there are a $\lambda\geq 1$ and a $\delta > 0$
such that $\upind_\mu (f,\lambda ,\delta ) > 0$.
Then there is a $d > 0$ and a cofinal set $L\subseteq\Lambda$ such that
$\varphi_{f,\lambda ,\delta} (F_n ) \geq d |F_n |$ for all $n\in L$.
Let $b$ be a positive number to be further specified below, and set
$\delta' = \delta b$.
Let $c>0$ and $\varepsilon > 0$ be as given by Lemma~\ref{L-l1 to indep} with respect
to $\delta = \lambda^{-1}$.
Take an $\eta > 0$ such that whenever $h$ is an element of
$L^\infty (X,\mu )$ for which $\| h \|_\mu < \eta$ the set
$\{ x\in X : |h(x)| \leq\varepsilon /12 \}$ has measure at least $1-\delta'$.

Now let $n\in L$, and suppose we are given $g_s \in \LX$ for $s\in F_n$ such that
$\| g_s - \alpha_s (f) \|_\mu < \eta$
for every $s\in F_n$. By our choice of $\eta$, for every $s\in F_n$ there
is a projection $p_s \in\mathscr{P} (\mu , \delta' )$ such that
$\| p_s (g_s - \alpha_s (f)) \| \leq\varepsilon /12$. Denote by $\cS$ the set
of all $\sigma\in \{ 1,2 \}^{F_n}$ such that $|\sigma^{-1} (2)| \leq b|F_n |$.
Setting $p_{s,1} = p_s$ and $p_{s,2} = p_s^\perp$
we define the projection $r = \sum_{\sigma\in\cS} \prod_{s\in F_n } p_{s,\sigma (s)}$.
Then
\[ \mu (r^\perp ) b|F_n | \leq \sum_{s\in F_n} \mu (p_s^\perp ) \leq |F_n | \delta' \]
and so $\mu (r^\perp ) \leq b^{-1} \delta' = \delta$. Hence there is an
$K\subseteq F_n$ with $|K|\geq d|F_n |$ such that $K$ is an
$\ell_1$-$\lambda$-isomorphism set for $f$ relative to $r$.

By our choice of $c$ and $\varepsilon$, assuming that $|F_n |$ is sufficiently large
we can find closed disks $B_1 , B_2 \subseteq\Cb$ of diameter at most $\varepsilon /6$
with ${\rm dist} (B_1 , B_2 ) \geq \varepsilon$ and a
$J\subseteq K$ with $|J| \geq c|K|$ such that the collection
\[ \big\{ \big( (\alpha_s (f)|_{\Omega_r} )^{-1} (B_1 ) ,
(\alpha_s (f)|_{\Omega_r} )^{-1} (B_2 ) \big) : s\in J \big\} \]
of pairs of subsets of $\Omega_r$ is independent.
Define the subsets $C_{s,1} = (g_s |_{\Omega_r} )^{-1} (B_1' )$
and $C_{s,2} = (g_s |_{\Omega_r} )^{-1} (B_2' )$ of $\Omega_r$, where
$B_1'$ (resp.\ $B_2'$) is the closed disk with the same centre as $B_1$
(resp.\ $B_2$) but with radius bigger by $\varepsilon / 12$.
Since $\max_{s\in J} \| p_s (g_s - \alpha_s (f)) \| \leq\varepsilon /12$,
for each $\sigma\in \{ 1,2 \}^J$ we can find by the
definition of $r$ a set $J_\sigma \subseteq J$ with
$|J\setminus J_\sigma | \leq b|F_n |$ such that
$\bigcap_{s\in J_\sigma } (\Omega_{p_s} \cap C_{s,\sigma (s)} ) \neq \emptyset$,
and we define $\rho_\sigma \in \{ 0,1,2 \}^J$ by
\[ \rho_\sigma (s) = \left\{ \begin{array}{l@{\hspace*{8mm}}l}
\sigma (s) &
\text{if } s\in J_\sigma , \\
0 &
\text{otherwise} .
\end{array} \right. \]
Since $\max_{\sigma\in\{ 1,2 \}^J} |\rho_\sigma^{-1} (0)| \leq 2^{b|F_n |}$, for every
$\rho\in \{ 0,1,2 \}^J$ the number of
$\sigma\in\{ 1,2 \}^J$ for which $\rho_\sigma = \rho$ is at most
$2^{b|F_n |}$, and so the set
$\cR = \big\{ \rho_\sigma : \sigma \in \{ 1,2 \}^J \big\}$ has
cardinality at least $2^{|J|} / 2^{d|F_n |} \geq 2^{(cd-b)|F_n |}$.
It follows by Lemma~\ref{L-density} that for a small enough $b$ not depending on $n$
there exists a $t > 0$ for which
%as it ranges over $L$
we can find an $I\subseteq J$ with
$|I|\geq t|J| \geq tcd|F_n |$ such that
$\cR |_I \supseteq \{ 1,2 \}^I$.
Then the collection $\{ (C_{s,1} , C_{s,2} ) : s\in I \}$ is independent, and
since ${\rm dist} (B_1' , B_2' ) \geq 5\varepsilon /6 >
2\max (\diam (B_1' ), \diam (B_2' ))$
the standard Rosenthal-Dor argument \cite{Dor} shows that the linear map
$\ell_1^I \to\spn \{ g_s : s\in I \}$ sending the standard basis element with index $s\in I$
to $g_s$ has an inverse with norm at most $10\varepsilon^{-1}$.
We thus obtain (4).

(7)$\Rightarrow$(8). It suffices to note that if $N$ is an $G$-invariant
von Neumann subalgebra of $\LX$ then for every finite subset $\Theta\subseteq N$ we
have $\hcpa_{\mu |_N} (N , \Theta ) = \hcpa_\mu (\LX , \Theta )$,
i.e., for computing c.p.\ approximation entropy it doesn't matter
whether $\Theta$ is considered as a subset of $N$ or $\LX$. This follows from
the fact that there is a $\mu$-preserving conditional expectation from $\LX$
onto $N$ \cite[Prop.\ V.2.36]{Tak1}. See the proof of Proposition~3.5 in \cite{Voi}.

(8)$\Rightarrow$(1). Suppose that $f\in\mathfrak{P}_X$. Let $\Upsilon$ be a finite
subset of the von Neumann subalgebra of $\LX$ generated by $f$ and let $\delta > 0$.
Take a finite Borel partition $\cP$ of $X$ such that the characteristic functions of the atoms
of $\cP$ are spectral projections of $f$ and
$\sup_{g\in\Omega} \esssup_{x,y\in P} | g(x) - g(y) | < \delta$
for each $P\in\cP$. Then $\hmeas_\mu (X,\cP ) = 0$ by our assumption, and thus, since
we may suppose $G$ to be infinite (for otherwise the system has completely positive entropy),
we obtain $\hcpa_\mu (\beta , \Upsilon, \sqrt{\delta^2+4\delta} ) \leq 2\delta$ by Lemma~\ref{L-rcp upper bd}.
Hence (8) fails to hold. Thus (8) implies (1).

Assume now that $G$ is countable and the action is free and ergodic and let us show that (9)
is equivalent to the other conditions.

(3)$\Rightarrow$(9). Let $d$, $\delta$, and $\lambda$ be as given by (3). Let $g$ be an
element of $\LX$ such that $\| g - f \|_\mu < \delta$. Then
$\| \alpha_s (g) - \alpha_s (f) \|_\mu < \delta$
for all $s\in G$, and so for every $n\in\Nb$ there is an $I\subseteq F_n$
of cardinality at least $d|F_n |$ for which $\{ \alpha_s (g) : s\in I \}$ is
$\| g \|\lambda$-equivalent in the operator norm to the standard basis of $\ell_1^I$.
Thus $g$ has positive $\ell_1$-isomorphism density.

(9)$\Rightarrow$(8). Suppose that $G$ is countable. 
We will first treat the case that the action of $G$ on $X$ is free. Suppose contrary
to (8) that $\hcpa_\mu (\beta ) = 0$. Since $\alpha$ is free and ergodic so is $\beta$,
and since $G$ is countable the von Neumann subalgebra of $\LX$ dynamically 
generated by $f$ has separable predual.
We can thus apply the Jewett-Krieger theorem
for free ergodic measure-preserving actions of countable discrete amenable groups 
on Lebesgue spaces (see \cite{FUG}, which
shows the finite entropy case; the general case was announced
in \cite{SEM} but remains unpublished) to
obtain a topological $G$-system $(Y,G)$ with a unique invariant Borel probability
measure $\nu$ such that $\beta$ can be realized as the action of $G$ on
$L^\infty (Y,\nu )$ arising from the action of $G$ on $Y$. Now let $\delta > 0$ be as
given by (9). Take a function
$g\in C(Y) \subseteq L^\infty (Y,\nu )$ such that
$\| g - f \|_\mu < \delta$. Since the system $(Y,G)$ has zero topological entropy by the
variational principle \cite{JMO}, it follows by Theorem~5.3 of \cite{DEBS}
(which is stated for $\Zb$-systems but is readily seen to cover actions of
general amenable groups) that the
function $g$ has zero $\ell_1$-isomorphism density,
%with respect to the operator norm,
contradicting our choice of $\delta$. We thus obtain (9)$\Rightarrow$(8)
in the case that the action is free.

Suppose now that the action of $G$ on $X$ is not free.
Take a free weakly mixing measure-preserving acion of $G$ on a Lebesgue space
$(Z,\sZ ,\omega )$ (e.g., a Bernoulli shift). Then the product action on
$X\times Z$ is free and ergodic. Write $E$ for the conditional expectation of
$L^\infty (X\times Z , \mu\times\omega )$ onto $L^\infty (X,\mu )$. With $\delta > 0$ as
given by (9), for every $g\in L^\infty (X\times Z,\mu\times\omega )$ such that
$\| E(g) - f \|_\mu < \delta$ the function $E(g)$ has positive $\ell_1$-isomorphism
density, which implies that $g$ has positive $\ell_1$-isomorphism density
since $E$ is contractive and $G$-equivariant. Thus the function
$f\otimes\boldsymbol{1}$ in $L^\infty (X\times Z,\mu\times\omega )$ also satisfies (9)
for the same $\delta$. By the previous paragraph we obtain (8) for
$f\otimes\boldsymbol{1}$. But this is equivalent to (8) for $f$ itself,
yielding (9)$\Rightarrow$(8) when $G$ is countable.

Suppose that $G$ is uncountable and $X$ is metrizable. In this case we will
actually show (9)$\Rightarrow$(7). For every 
$s\in G$ write $\cE_s$ for the orthogonal complement in $L^2 (X,\mu )$
of the subspace of vectors fixed by $s$. Then the span of $\bigcup_{s\in G} \cE_s$ is dense 
in $L^2 (X,\mu )\ominus\Cb\boldsymbol{1}$ by ergodicity, and since $L^2 (X,\mu )$ is separable
there is a countable set $J\subseteq G$ such that the span of $\bigcup_{s\in J} \cE_s$  
is dense in $L^2 (X,\mu )\ominus\Cb\boldsymbol{1}$. It follows that the subgroup $H$
generated by $J$ does not fix any vectors in $L^2 (X,\mu )\ominus\Cb\boldsymbol{1}$.
This means that the action of $H$ on $X$ is ergodic, as is the action of any
subgroup of $G$ containing $H$. By Lemma~\ref{L-min} condition (9) holds
for the action of every subgroup of $G$ containing $H$, and thus for the action of
a countable such subgroup we get (9)$\Rightarrow$(8) by the two previous paragraphs 
and hence (9)$\Rightarrow$(7).
But if (7) fails for the action of $G$ then it fails for the action of every 
subgroup of $G$ containing some fixed countable subgroup $W$ of $G$ and in particular
for the action of the countable subgroup generated by $H$ and $W$, yielding a contradiction.

Finally, we suppose that $f\in C(X)$ and demonstrate the equivalence of (11) and (12)
with the other conditions.

(2)$\Rightarrow$(11). The inclusion $C(\supp(\mu ))\subseteq\LX$ gives rise at the spectral
level to a topological $G$-factor map $\Omega\to \supp (\mu )$, and so the implication
follows from Proposition~\ref{P-IE basic}(5).

(11)$\Rightarrow$(10). Use Proposition~\ref{P-IE basic}(5).

(10)$\Rightarrow$(11). Suppose that $f(x_1)=f(x_2)$ for
every $(x_1, x_2)\in \IE^\mu_2 (X)$. Set $E=\{(x, y)\in X\times X:f(x)=f(y)\}$.
Then $E$ is a closed equivalence relation on $X$.
Thus $\bigcap_{s\in G} sE$ is a $G$-invariant closed equivalence
relation on $X$ and hence gives rise to a topological $G$-factor $Y$ of $X$. In
particular, $f\in C(Y)$. Denote the factor map $X\rightarrow Y$ by $\pi$.
Our assumption says that $\IE^\mu_2 (X)\subseteq E$. Since $\IE^\mu_2 (X)$ is
$G$-invariant, $\IE^\mu_2 (X)\subseteq \bigcap_{s\in G} sE$. This means that
$(\pi\times \pi)(\IE^\mu_2 (X))\subseteq \triangle_Y$.
By (2) and (5) of Proposition~\ref{P-IE basic}, $h_{\pi_*(\mu)} (Y) = 0$.

(11)$\Rightarrow$(3). Apply the same argument as for (2)$\Rightarrow$(3).
\end{proof}

Theorem~\ref{T-Pinsker} shows that for general $X$
the Pinsker von Neumann algebra
is the largest $G$-invariant von Neumann subalgebra of $\LX$ on which
the c.p.\ approximation entropy is zero.

\begin{remark}
One interesting consequence of Theorem~\ref{T-Pinsker} is the following. 
In the case that $G$ is countable, 
if a weakly mixing measure-preserving action of $G$ on a 
Lebesgue space $(Y,\sY , \nu )$ does not have completely positive entropy, 
then it has a metrizable topological model
$(Z,G)$ for which the set $\IE^k (Z)$ of topological IE-tuples has 
zero $\nu^k$-measure for each $k\geq 2$. 
Indeed weak mixing implies that the product action of $G$ on $Y^k$ is ergodic 
with respect to $\nu^k$, so that for a topological model $(Z,G)$ and $k\geq 2$
the set $\IE^k (Z)$ has $\nu^k$-measure either zero or one. If
for every metrizable topological model $(Z,G)$ we had $\nu^k (\IE^k (Z)) = 1$ for
some $k\geq 2$, it would follow that every element of $L^\infty (Y,\nu )$
has positive $\ell_1$-isomorphism density, since such an element is a
continuous function for some metrizable topological model by the countability
of $G$ and hence separates a topological IE-pair.
But then $(Y,\sY , \nu , G)$ would have completely positive entropy by
Theorem~\ref{T-Pinsker}. Actually the weak mixing assumption 
can be weakened to the requirement that there be no sets of measure strictly
between zero and one with finite $G$-orbit.

We also point out that, in a related vein, if the topological system $(X,G)$
does not have completely positive entropy, then for a $G$-invariant Borel probability 
measure on $X$ the set $\IE^k (X)$ has zero product measure for each $k\geq 2$, unless
some nontrivial quotient of $(X,G)$ has points with positive induced measure. The reason
is that if $\IE^k (X)$ for some $k\geq 2$ has positive product measure then so does 
$\IE^k (Y)$ with respect to the induced measure for every quotient $(Y,G)$ of 
$(X,G)$, and if every point in such a quotient $(Y,G)$ has zero induced measure then 
the diagonal in $Y^k$ has zero product measure and hence does
not contain $\IE^k (Y)$, implying that $(Y,G)$ has positive topological entropy.
In particular, we see that if $(X,G)$ is minimal and does not have completely positive
entropy and $X$ is connected (and hence has no nontrivial finite quotients) then for 
every $G$-invariant Borel probability measure on $X$ the set
$\IE^k (X)$ has zero product measure for each $k\geq 2$. 
\end{remark}

At the extreme end of completely positive entropy where the Pinsker von Neumann
algebra reduces to the scalars, the picture topologizes
and we have the following result. Recall that a topological system is said to
have {\it completely positive entropy} if every nontrivial factor has positive
topological entropy, {\it uniformly positive entropy} if every nondiagonal
element of $X\times X$ is an entropy pair, and {\it uniformly positive entropy
of all orders} if for each $k\geq 2$ every nondiagonal element of $X^k$ is an
entropy tuple (see \cite[Chap.\ 19]{ETJ} and \cite{LVRA}).

\begin{theorem}\label{T-cpe}
%Let $\oX = (X, ,\mu ,G)$ be a measure-preserving dynamical system. Let
Suppose that $X$ is metrizable or $G$ is countable.
Let $\boldsymbol{\Omega} = (\Omega , G)$ be the topological dynamical system
associated to $\oX = (X,\sB ,\mu ,G)$ as above. Then the following are equivalent:
\begin{enumerate}
\item $\oX$ has completely positive entropy,

\item every nonscalar element of $\LX$ has positive $\ell_1$-isomorphism density,

\item $\boldsymbol{\Omega}$ has completely positive entropy,

\item $\boldsymbol{\Omega}$ has uniformly positive entropy,

\item $\boldsymbol{\Omega}$ has uniformly positive entropy of all orders.
\end{enumerate}
\end{theorem}

\begin{proof}
(1)$\Rightarrow$(5). Every Borel partition of $\Omega$ is $\mu$-equivalent
to a clopen partition and thus every nontrivial such partition has positive entropy
by (1). It follows that, for each $k\geq 2$, every nondiagonal tuple in $\Omega^k$
is a $\mu$-entropy tuple and hence a $\mu$-IE-tuple by Theorem~\ref{T-IE E}.
Since $\mu$-IE-tuples are obviously IE-tuples and the latter are easily seen to be
entropy tuples when they are nondiagonal, we obtain (5).

(5)$\Rightarrow$(4)$\Rightarrow$(3). These implications hold for any topological $G$-system,
the first being trivial and the second being a consequence of the properties of entropy for
open covers with respect to taking extensions.

(3)$\Rightarrow$(2). Apply Corollary~5.5 of \cite{DEBS} as extended to actions
of discrete amenable groups.

(2)$\Rightarrow$(1). By (2) there do not exist any nonscalar $G$-invariant
projections in $\LX$, i.e., the system $\oX$ is ergodic. We can thus apply
(9)$\Rightarrow$(1) of Theorem~\ref{T-Pinsker}.
\end{proof}

For $G=\Zb$ the equivalence of (1), (3), (4), and (5) in Theorem~\ref{T-cpe}
can also be obtained from Section~3 of \cite{GW}.
%Theorem~3.4 of \cite{LVRA} or Theorem~3.1 of \cite{GW}.

One might wonder whether a similar type of topologization occurs at the other
extreme of zero entropy. Glasner and Weiss
showed however in \cite{SEUPEM} that every free ergodic $\Zb$-system has a minimal
topological model with uniformly positive entropy.

Using Theorem~\ref{T-Pinsker} and viewing joinings as equivariant unital positive
maps, we can give a linear-geometric proof of the disjointness of zero entropy systems
from completely positive entropy systems, which for measure-preserving actions of
discrete amenable groups on Lebesgue spaces was established in \cite{ETWP}
(see also Chapter~6 of \cite{ETJ}).
Recall that a {\em joining} between two measure-preserving $G$-systems $(Y,\sY ,\nu ,G)$ and
$(Z,\sZ ,\omega ,G)$ is a $G$-invariant probability measure on
$(Y\times Z, \sY\otimes\sZ )$ with $\nu$ and $\omega$ as marginals. The two systems
are said to be {\em disjoint} if $\nu\times\omega$ is the only joining between them.

\begin{proposition}\label{P-joining Pinsker}
Let $(Y,\sY ,\nu ,G)$ and $(Z,\sZ ,\omega ,G)$
be measure-preserving $G$-systems. Let $\varphi : L^\infty (Y,\nu ) \to L^\infty (Z,\omega )$
be a $G$-equivariant unital positive linear map such that $\omega\circ\varphi = \nu$. Then
$\varphi (\mathfrak{P}_X ) \subseteq \mathfrak{P}_Y$.
\end{proposition}

\begin{proof}
Since $\varphi$ is unital and positive it is operator norm contractive and
for every $f\in L^\infty (Y,\nu )$ we have
\[ \| \varphi (f) \|_\omega = \omega (\varphi (f)^* \varphi (f))^{1/2}
\leq \omega (\varphi (f^* f))^{1/2} = \nu (f^* f)^{1/2} = \| f \|_\nu , \]
that is, $\varphi$ is also contractive for the norms
$\| \!\cdot\! \|_\nu$ and $\| \!\cdot\! \|_\omega$. Thus if condition (3) in
Theorem~\ref{T-Pinsker} holds for a given $f\in L^\infty (Z,\omega )$ with
witnessing constants $d$, $\delta$, and $\lambda$ then it also holds
for every element of $\varphi^{-1} (\{ f \} )$ with the same witnessing constants.
The equivalence (1)$\Leftrightarrow$(3) in Theorem~\ref{T-Pinsker} now
yields the proposition.
\end{proof}

A joining $\eta$ between two measure-preserving systems
$\boldsymbol{Y} = (Y,\sY ,\nu ,G)$ and\linebreak $\boldsymbol{Z} = (Z,\sZ ,\omega ,G)$
gives rise as follows to a $G$-equivariant unital positive linear map
$\varphi : L^\infty (Y,\nu ) \to L^\infty (Z,\omega )$
such that $\omega\circ\varphi = \nu$ (this is a special case of
a construction for correspondences between von Neumann algebras \cite{Popa}).
Define the operator $S : L^2 (Z,\omega ) \to L^2 (Y\times Z,\eta )$
by $(S\xi )(y,z) = \xi (z)$ for all $\xi\in L^2 (Z,\omega )$ and $(y,z)\in Y\times Z$ and
the representation $\pi : L^\infty (Y,\nu ) \to \cB (L^2 (Y\times Z,\eta ))$ by
$(\pi (f) \zeta )(y,z) = f(y)\zeta (y,z)$ for all $f\in L^\infty (Y,\nu )$,
$\zeta\in L^2 (Y\times Z,\eta )$, and $(y,z)\in Y\times Z$. Then for $f\in L^\infty (Y,\nu )$
we set $\varphi (f) = S^* \pi (f) S$. It is easily checked that $S^* \pi (f) S$
commutes with every element of the commutant $L^\infty (Z,\omega )'$, so that
$\varphi (f) \in L^\infty (Z,\omega )'' = L^\infty (Z,\omega )$.
Now define the representation $\rho : L^\infty (Z,\omega ) \to \cB (L^2 (Y\times Z,\eta ))$
by $(\rho (g) \zeta )(y,z) = g(z)\zeta (y,z)$ for all $g\in L^\infty (Z,\omega )$,
$\zeta\in L^2 (Y\times Z,\eta )$, and $(y,z)\in Y\times Z$.
Then for $f\in L^\infty (Y,\nu )$ and $g\in L^\infty (Z,\omega )$ we have, with
$\boldsymbol{1}$ denoting the unit in the appropriate $L^\infty$ algebra,
\begin{align*}
\eta (\pi (f)\rho (g)) &= \langle \pi (f)\rho (g) ,
\boldsymbol{1}\otimes \boldsymbol{1} \rangle_\eta
= \langle \pi (f) \rho (g) S\boldsymbol{1}, S\boldsymbol{1} \rangle_\eta \\
&= \langle \pi (f) Sg\boldsymbol{1}, S\boldsymbol{1} \rangle_\eta
= \langle S^* \pi (f) Sg\boldsymbol{1}, \boldsymbol{1} \rangle_\omega \\
&= \omega (\varphi (f)g) .
\end{align*}
In the case that the image of $\varphi$ is the scalars, we see that $\eta$ gives rise to
the product state $\varphi\otimes\omega$ on
$L^\infty (Y,\nu ) \otimes L^\infty (Z,\omega )$ under composition with the
representation $f\otimes g \mapsto \pi (f)\rho (g)$, and furthermore $\varphi = \nu$
by the assumption on the marginals in the definition of joining.

\begin{corollary}\label{C-disjoint entropy}
Let $\boldsymbol{Y} = (Y,\sY ,\nu ,G)$ and $\boldsymbol{Z} = (Z,\sZ ,\omega ,G)$
be measure-preserving $G$-systems. Suppose that $\boldsymbol{Y}$ has zero entropy
and $\boldsymbol{Z}$ has completely positive entropy.
Then $\boldsymbol{Y}$ and $\boldsymbol{Z}$ are disjoint.
\end{corollary}

\begin{proof}
As above, a joining $\eta$ between $\boldsymbol{Y}$ and $\boldsymbol{Z}$ gives rise to a
$G$-equivariant unital positive linear map $\varphi : L^\infty (Y,\nu ) \to L^\infty (Z,\omega )$
such that $\omega\circ\varphi = \nu$.
By Proposition~\ref{P-joining Pinsker} the image of such a map $\varphi$ must be
the scalars. Hence there is only the one joining $\nu\times\omega$.
\end{proof}

%%%%%%%%%%%%%%%%%%%%%%%%%%%%%%%%%%%%%%%%%%%%%%%%%%%%%%%%%%%%%%%%%%%%%%%

\section{Measure IN-tuples}\label{S-IN}

In this section $(X,G)$ is an arbitrary topological dynamical system
and $\mu$ a $G$-invariant Borel probability measure on $X$. We will
define $\mu$-IN-tuples and establish some properties in analogy
with $\mu$-IE-tuples. Here the role of measure entropy is played by
measure sequence entropy. The combinatorial phenomena responsible for the
properties of $\mu$-IE-tuples in Proposition~\ref{P-IE basic} apply equally
well to the sequence entropy framework,
and so it will essentially be a matter of recording the
analogues of various lemmas from Section~\ref{S-IE}. We will also
show that nondiagonal $\mu$-IN-tuples are the
same as $\mu$-sequence entropy tuples and derive the measure IN-tuple 
product formula.

For $\delta > 0$ we say that a finite tuple $\oA$ of subsets of $X$
has {\em $\delta$-$\mu$-independence density over arbitrarily large finite sets}
if there exists a $c>0$ such that for every $M>0$ there is a finite set
$F\subseteq G$ of cardinality at least $M$ which possesses the property that
every $D\in\mathscr{B}' (X,\delta )$ has a $\mu$-independence set $I\subseteq F$
relative to $D$ with $|I|\ge c|F|$. We say that $\oA$ has
{\em positive sequential $\mu$-independence density} if for some
$\delta > 0$ it has $\delta$-$\mu$-independence density over arbitrarily large
finite sets.

Arguing as in the proof of Lemma~\ref{L-split} yields:

\begin{lemma}\label{L-split sh}
Let $\oA = (A_1 , \dots , A_k )$ be a tuple of subsets of $X$ which has
positive sequential $\mu$-independence density.
Suppose that $A_1 = A_{1,1} \cup A_{1,2}$.
Then at least one of the tuples $\oA_1 = (A_{1,1} , A_2 , \dots , A_k )$ and
$\oA_2 = (A_{1,2} , A_2 , \dots , A_k )$ has
positive sequential $\mu$-independence density.
\end{lemma}

In \cite{Ind} we defined a tuple $\ox = (x_1 , \dots, x_k )\in X^k$
to be an IN-tuple (or an IN-pair in the case $k=2$) if for every product
neighbourhood $U_1 \times\cdots\times U_k$ of $\ox$ the $G$-orbit of the tuple
$(U_1, \dots, U_k)$ has arbitrarily large finite independent subcollections.
Here is the measure-theoretic analogue:

\begin{definition}
We call a tuple $\ox = (x_1 , \dots , x_k )\in X^k$ a {\em
$\mu$-IN-tuple} (or {\em $\mu$-IN-pair} in the case $k=2$)
if for every product neighbourhood $U_1 \times\cdots\times U_k$ of
$\ox$ the tuple $(U_1 , \dots , U_k )$ has positive sequential
$\mu$-independence density. We denote the set of
$\mu$-IN-tuples of length $k$ by $\IN^\mu_k (X)$.
\end{definition}

Obviously every $\mu$-IN-tuple is a IN-tuple.

The following analogue of Lemma~\ref{L-pe to pd} follows
immediately from Lemma~\ref{L-mind}.

\begin{lemma}\label{L-pse to mind}
Let $\cP=\{P_1, P_2\}$ be a two-element Borel partition of $X$
such that $h_{\mu}(\cP; \mathfrak{s})>0$
for some sequence $\mathfrak{s}$ in $G$. Then there exists
$\varepsilon>0$ such that whenever $A_1\subseteq P_1$ and $A_2\subseteq P_2$
are Borel sets with $\mu(P_1\setminus A_1), \mu (P_2\setminus A_2)<\varepsilon$
the pair $\oA=(A_1, A_2)$ has positive sequential $\mu$-independence density.
\end{lemma}

Fix a sequence $\mathfrak{s}=\{s_j\}_{j\in \Nb}$ in $G$. Recalling the notation
$\varphi_{\oA, \delta}$ and $\varphi'_{\oA, \delta}$ from Subsection~\ref{SS-density},
for $\delta>0$ we set
\begin{align*}
\upind_\mu(\oA, \delta;\mathfrak{s}) &= \limsup_{n\to
\infty}\frac{1}{n}\varphi_{\oA, \delta}(\{s_1,\dots ,s_n\}), \\
\upind'_\mu(\oA, \delta;\mathfrak{s}) &= \limsup_{n\to
\infty}\frac{1}{n}\varphi'_{\oA, \delta}(\{s_1,\dots ,s_n\}), \\
\upind_\mu (\oA; \mathfrak{s} ) &= \sup_{\delta
> 0} \upind_\mu (\oA , \delta;\mathfrak{s} ).
\end{align*}
By Lemma~\ref{L-prime}, we have
\[ \upind_\mu (\oA;\mathfrak{s} ) = \sup_{\delta
> 0} \upind'_\mu (\oA , \delta;\mathfrak{s} ). \]
Clearly $\oA$ has positive sequential $\mu$-independence density
if and only if $\upind_{\mu} (\oA;\fs )>0$ for some sequence $\mathfrak{s}$ in $G$.

Let $\cU$ be a finite Borel cover of $X$. Recall that $\Hmeas (\cU )$ denotes the
infimum of the entropies $\Hmeas (\cP )$ over all finite Borel partitions $\cP$ of $X$
that refine $\cU$. For $\delta>0$ we set
\begin{align*}
\uph_{\comb , \mu}(\cU, \delta; \mathfrak{s}) &= \limsup_{n\to \infty}
\frac{1}{n}\ln N_{\delta} \bigg( \bigvee^n_{j=1}s^{-1}_j\cU \bigg) , \\
\uph_{\comb , \mu}(\cU; \mathfrak{s}) &=
\sup_{\delta>0}\uph_{\comb , \mu}(\cU, \delta; \mathfrak{s}),\\
h_{\mu}^-(\cU; \mathfrak{s}) &=
\limsup_{n\to \infty}\frac{1}{n}H \bigg( \bigvee^n_{j=1}s^{-1}_j\cU \bigg) , \\
h_{\mu}^+(\cU; \mathfrak{s}) &=
\inf_{\cP\succeq \cU}h_{\mu}(\cP;\mathfrak{s}) ,
\end{align*}
where the last infimum is taken over finite Borel partitions refining $\cU$.
Both $h_{\mu}^-(\cU; \mathfrak{s})$ and $h_{\mu}^+(\cU; \mathfrak{s})$
appeared in \cite{HMY} for the case of $G=\Zb$. We have
$h_{\mu}^-(\cU; \mathfrak{s})\le h_{\mu}^+(\cU;\mathfrak{s})$ trivially.

The next lemma is the analogue of Lemma~\ref{L-hlifting} and follows directly from
Lemma~\ref{L-Hlifting}.

\begin{lemma}\label{L-shlifting}
Let $\pi:X\to Y$ be a factor of $X$. For any finite Borel cover
$\cU$ of $Y$, one has
$$ h_{\mu}^-(\pi^{-1}\cU;\mathfrak{s})=h_{\pi_*(\mu)}^-(\cU;\mathfrak{s}).$$
\end{lemma}

The argument in the proof of Lemma~\ref{L-local} can also be used to show:

\begin{lemma}\label{L-local se}
We have $\delta\cdot \uph_{\comb , \mu}(\cU, \delta; \mathfrak{s})
\le h_{\mu}^-(\cU; \mathfrak{s})\le
\uph_{\comb , \mu}(\cU; \mathfrak{s})$.
\end{lemma}

Next we come to the analogue of Lemma~\ref{L-positive}.

\begin{lemma}\label{L-positive se}
For a finite Borel cover $\cU$ of $X$, the quantities
$h_{\mu}^-(\cU;\mathfrak{s})$ and $\uph_{\comb , \mu}(\cU;\mathfrak{s})$
are either both zero or both nonzero. If the complements
in $X$ of the members of $\cU$ are pairwise disjoint and $\oA$ is a tuple
consisting of these complements, then we may also add the third
quantity $\upind_{\mu}(\oA;\mathfrak{s})$ to the list.
\end{lemma}

\begin{proof}
The first assertion is a consequence of Lemma~\ref{L-local}.
For a tuple $\oA$ as in the lemma statement, Lemma~3.3 of \cite{Ind} and
Lemma~\ref{L-converse} show that
$\uph_{\comb , \mu}(\cU;\mathfrak{s})>0$ if and only if
$\upind_{\mu}(\oA;\mathfrak{s})>0$.
\end{proof}

\begin{proposition}\label{P-IN basic}
The following hold:
\begin{enumerate}
\item Let $\oA = (A_1 , \dots , A_k )$ be a tuple of closed subsets of $X$ which
has positive sequential $\mu$-independence density.
Then there exists a $\mu$-IN-tuple
$(x_1 , \dots , x_k )$ with $x_j \in A_j$ for $j=1, \dots ,k$.

\item $\IN^\mu_2 (X) \setminus \Delta_2 (X)$ is nonempty if and only if the system
$(X, \sB , \mu ,G)$ is nonnull.

\item $\IN^\mu_1(X)={\rm supp}(\mu)$ when $G$ is an infinite group.

\item $\IN^\mu_k (X)$ is a closed $G$-invariant subset of $X^k$.

\item Let $\pi:X\rightarrow Y$ be a topological $G$-factor map.
Then $\pi^k (\IN^\mu_k(X))=\IN^{\pi_*(\mu)}_k(Y)$.
\end{enumerate}
\end{proposition}

\begin{proof}
(1) Apply Lemma~\ref{L-split sh} and a compactness argument.

(2) As is well known and easy to show,
$(X, \mu)$ is nonnull if and only if there is a two-element
Borel partition of $X$ with positive sequence entropy
with respect to some sequence in $G$. We thus obtain the ``if'' part by (1) and
Lemma~\ref{L-pse to mind}. For the ``only if'' part apply Lemma~\ref{L-positive se}.

(3) This follows from Lemma~\ref{L-supp}.

(4) Trivial.

(5) This follows from (1), (3), (4) and Lemmas~\ref{L-shlifting} and
\ref{L-positive se}.
\end{proof}

The concept of measure sequence entropy tuple originates in \cite{HMY},
which deals with the case $G=\Zb$. The
definition works equally well for general $G$. Thus for $k\geq 2$ we say that a
nondiagonal tuple $(x_1 , \dots , x_k ) \in X^k$ is a
{\it sequence entropy tuple for $\mu$} if whenever $U_1 , \dots , U_l$
are pairwise disjoint Borel neighbourhoods of the distinct points in
the list $x_1 , \dots , x_k$, every Borel partition of $X$ refining
the cover $\{ U_1^\comp , \dots , U_l^\comp \}$ has
positive measure sequence entropy with respect to some sequence in $G$.
To show that nondiagonal $\mu$-IN-tuples are the same as $\mu$-sequence
entropy tuples, it suffices by Lemma~\ref{L-positive se} to prove that if
$\cU$ is a cover of $X$ consisting of the complements of neighbourhoods of
the points in a $\mu$-sequence entropy tuple then
$h_{\mu}^-(\cU;\mathfrak{s})>0$ for some sequence $\mathfrak{s}$ in $G$.
For $G=\Zb$ this was done by Huang, Maass, and Ye in Theorem~3.5 of \cite{HMY}.
Their methods readily extend to the general case, as we will now indicate.

Given a unitary representation $\pi : G\to\cB (\cH )$, the Hilbert space $\cH$
orthogonally decomposes into
two $G$-invariant closed subspaces $\cH_\wm$ and $\cH_\cpct$ such that
$\pi$ is weakly mixing on $\cH_\wm$ and the $G$-orbit of every vector in $\cH_\cpct$
has compact closure \cite{Gode}. For our $\mu$-preserving action of $G$
on $X$, considering its associated unitary representation of $G$
on $L^2 (X,\mu )$ there exists by Theorem~7.1 of \cite{Zimmer}
a $G$-invariant von Neumann subalgebra $\mathfrak{D}_X \subseteq \LX$ such that
$L^2 (X,\mu )_\cpct = L^2 (\mathfrak{D}_X ,\mu |_{\mathfrak{D}_X} )$.
The following lemma generalizes part of Theorem~2.3 of \cite{HMY} with
essentially the same proof. In \cite{HMY} $X$ is assumed to be metrizable, but
that is not necessary here.

\begin{lemma}\label{L-seq rel}
Let $\cP$ be a finite Borel partition of $X$. Then there is a sequence
$\fs$ in $G$ such that $\hmeas_\mu(\cP; \fs) \geq \Hmeas (\cP | \mathfrak{D}_X )$.
\end{lemma}

\begin{proof}
First we show that, given a finite Borel partition $\cQ$ of $X$ and an
$\varepsilon > 0$, the set of all $s\in G$ such that
$\Hmeas (s^{-1} \cP | \cQ ) \geq \Hmeas (\cP | \mathfrak{D}_X ) - \varepsilon$
is thickly syndetic. Write $\cP = \{ P_1 , \dots , P_k \}$ and
$\cQ = \{ Q_1 ,\dots , Q_l \}$ and denote by $E$ the $\mu$-preserving conditional
expectation onto $\mathfrak{D}_X$. Since
$\boldsymbol{1}_A - E(\boldsymbol{1}_A ) \in L^2 (X,\mu )_\wm$ for
every Borel set $A\subseteq X$ and thick syndeticity is preserved under taking
finite intersections, for each $\eta > 0$ the set of all $s\in G$ such that
$\sup_{1\leq i\leq k,1\leq j\leq l} | \langle U_s (\boldsymbol{1}_{P_i} -
E(\boldsymbol{1}_{P_i} )),\boldsymbol{1}_{Q_j} \rangle | < \eta$
is thickly syndetic. It follows that for all $s$ in some thickly syndetic set
we have, using the concavity of the function $x\mapsto -x\ln x$,
\begin{align*}
\Hmeas (s^{-1} \cP | \cQ ) + \varepsilon
&\geq \sum_{i=1}^k \sum_{j=1}^l -\langle U_s E(\boldsymbol{1}_{P_i} ) ,
\boldsymbol{1}_{Q_j} \rangle \ln \bigg(
\frac{\langle U_s E(\boldsymbol{1}_{P_i} ) , \boldsymbol{1}_{Q_j} \rangle}{\mu (Q_j)} \bigg) \\
&\geq \sum_{i=1}^k \int_X -U_s E(\boldsymbol{1}_{P_i} )
\ln (U_s E(\boldsymbol{1}_{P_i} ))\, d\mu \\
&= \Hmeas (\cP | \mathfrak{D}_X ) ,
\end{align*}
as desired.

We can now recursively construct a sequence $\fs = \{s_1 = e, s_2 , s_3 , \dots \}$
in $G$ such that $\Hmeas (s_n^{-1} \cP | \bigvee_{i=1}^{n-1} s_i^{-1} \cP ) \geq
\Hmeas (\cP | \mathfrak{D}_X ) - 2^{-n}$ for each $n>1$. Using the identity
$\Hmeas (\bigvee_{i=1}^n s_i^{-1} \cP ) = \Hmeas (\bigvee_{i=1}^{n-1}
s_i^{-1} \cP ) + \Hmeas (s_n^{-1} \cP | \bigvee_{i=1}^{n-1} s_i^{-1} \cP )$ we then
get
\[ \hmeas_\mu (\cP; \fs) = \limsup_{n\to\infty} \frac1n \sum_{k=1}^n
\Hmeas (s_k^{-1} \cP | \textstyle{\bigvee}_{i=1}^{k-1} s_i^{-1} \cP ) \geq
\Hmeas (\cP | \mathfrak{D}_X ) . \]
\end{proof}

Using Lemma~\ref{L-seq rel} we can now argue as in the proof of Theorem~3.5 of \cite{HMY}
to deduce that $\hmeas_{\mu}^-(\cU;\mathfrak{s})>0$ for some sequence $\mathfrak{s}$ in $G$
whenever $\cU$ is a cover of $X$ whose elements are the complements of neighbourhoods of
the points in a $\mu$-sequence entropy tuple (it can be checked that the
metrizability hypothesis on $X$ in \cite{HMY} is not necessary in this case).
In \cite{HMY} the authors use the fact that $\mathfrak{D}_X$-measurable partitions have
zero measure sequence entropy for all sequences, which for $G=\Zb$ and metrizable $X$
is contained in \cite{Kush}. In our more general setting we can appeal to
Theorem~\ref{T-null} from the next section. We thus obtain the desired result:

\begin{theorem}\label{T-IN SE}
For every $k\geq 2$, a nondiagonal tuple in
$X^k$ is a $\mu$-IN-tuple if and only if it is a $\mu$-sequence entropy tuple.
\end{theorem}

To establish the product formula for $\mu$-IN-tuples we will make use of the maximal
null von Neumann algebra $\mathfrak{N}_X \subseteq \LX$, which corresponds to the largest
factor of the system with zero sequence entropy for all sequences 
(see the beginning of the next section).
Denote by $E'_X$ the conditional expectation $L^{\infty}(X, \mu)\rightarrow
\mathfrak{N}_X$. 
The following lemma is the analogue of Lemma~\ref{L-N Pinsker} and 
appeared as Lemma~3.3 in \cite{HMY}. 
Note that the assumptions in \cite{HMY} that 
$X$ is metrizable and $G=\Zb$ are not needed here.

\begin{lemma}\label{L-N Pinsker}
Let $\cU=\{U_1, \dots, U_k\}$ be a Borel cover of $X$. 
Then $\prod^k_{i=1}E'_X(\chi_{U^c_i})\neq 0$ if and only if for every
finite Borel partition $\cP$ finer than $\cU$ as a cover one has
$\hmeas_{\mu}(\cP; \mathfrak{s})>0$ for some sequence $\mathfrak{s}$ in $G$. 
\end{lemma}

Combining Lemma~\ref{L-N Pinsker}, Proposition~\ref{P-IN basic}(3), and
Theorem~\ref{T-IN SE}, we obtain the following analogue
of Lemma~\ref{L-measure IE char}. 

\begin{lemma}\label{L-measure IN char}
When $G$ is infinite, 
a tuple $\ox=(x_1, \dots, x_k)\in X^k$ is a $\mu$-IN tuple if and only
if for any Borel neighbourhoods $U_1 , \dots , U_k$ of $x_1 , \dots , x_k$, respectively,
one has $\prod^k_{i=1}E'_X(\chi_{U_i})\neq 0$.
\end{lemma}

The following is the analogue of Theorem~\ref{T-IE product}. 

\begin{theorem}\label{T-IN product}
Let $(Y,G)$ be another topological $G$-system and $\nu$ a $G$-invariant
Borel probability measure on $Y$. Then for all $k\geq 1$ we have
$\IN_{\mu\times\nu}^k (X\times Y) = \IN_\mu^k (X) \times \IN_\nu^k (Y)$.
\end{theorem}

\begin{proof} 
When $G$ is finite, both sides are empty. So
we may assume that $G$ is infinite. 
By Proposition~\ref{P-IN basic}(5) we have 
$\IN_{\mu\times\nu}^k (X\times Y) \subseteq 
\IN_\mu^k (X) \times \IN_\nu^k (Y)$. Thus we just need
to prove 
$\IN_\mu^k (X) \times \IN_\nu^k (Y)\subseteq \IN_{\mu\times\nu}^k (X\times Y)$.

Since the tensor product of a weakly mixing unitary representation
of $G$ and any other unitary representation of $G$ is weakly mixing, we
have $L^2 (X\times Y,\mu\times\nu )_\cpct=L^2 (X, \mu)_\cpct
\otimes L^2(Y,\nu)_\cpct$.
It follows that 
$\mathfrak{D}_{X\times Y}=\mathfrak{D}_X\otimes \mathfrak{D}_Y$.
By Theorem~\ref{T-null} from the next section we have $\mathfrak{N}_{X}=
\mathfrak{D}_X$. Thus $\mathfrak{N}_{X\times Y}=\mathfrak{N}_X\otimes 
\mathfrak{N}_Y$ and hence
$E'_{X\times Y}(f\otimes g)=
E'_X(f)\otimes E'_Y(g)$ for any $f\in L^{\infty}(X, \mu)$ and 
$g\in L^{\infty}(Y, \nu)$. Now the desired inclusion follows from 
Lemma~\ref{L-measure IN char}. 
\end{proof}

In the case $G=\Zb$, the product formula for measure sequence
entropy tuples is implicit in Theorem~4.5 of \cite{HMY}, and we have
essentially applied the argument from there granted the fact
that for general $G$ the maximal null factor is the same as the maximal 
isometric factor, as shown by Theorem~\ref{T-null}.

%%%%%%%%%%%%%%%%%%%%%%%%%%%%%%%%%%%%%%%%%%%%%%%%%%%%%%%%%%%%%%%%%%%%

\section{Combinatorial independence and the maximal null factor}\label{S-null}

We will continue to assume that $(X,G)$ is an arbitrary topological dynamical
system and $\mu$ is a $G$-invariant Borel probability measure on $X$.
In analogy with the Pinsker $\sigma$-algebra in the context of entropy,
the $G$-invariant $\sigma$-subalgebra of $\sB$ generated by all finite Borel
partitions of $X$ with zero sequence entropy for all sequences
(or, equivalently, all two-element Borel partitions of $X$ with zero sequence
entropy for all sequences) defines the largest
factor of the system with zero sequence entropy for all sequences (see \cite{HMY}).
The corresponding $G$-invariant von Neumann subalgebra of $\LX$
will be denoted by $\mathfrak{N}_X$ and referred to as the
{\em maximal null von Neumann algebra}. The system $(X,\sB , \mu , G)$ is said
to be {\it null} if $\mathfrak{N}_X = \LX$ (i.e., if it has zero measure sequence
entropy for all sequences) and {\it completely nonnull} if $\mathfrak{N}_X = \Cb$.
Kushnirenko showed that an ergodic $\Zb$-action on a Lebesgue space is
isometric if and only if $\mathfrak{N}_X = \LX$ \cite{Kush}. As Theorem~\ref{T-null}
will demonstrate more generally, $\mathfrak{N}_X$ always coincides with
$\mathfrak{D}_X$, as defined prior to Lemma~\ref{L-seq rel}.

Our main goal in this section is to establish Theorem~\ref{T-null}, which gives
various local descriptions of the maximal null factor in analogy with
Theorem~\ref{T-Pinsker}. To a large extent
the same arguments apply and we will simply refer to the
appropriate places in the proof Theorem~\ref{T-Pinsker}. On the other hand,
several conditions appear in Theorem~\ref{T-null} which have no analogue in
the entropy setting, reflecting the fact that there is a particularly strong
dichotomy between nullness and nonnullness.
This dichotomy hinges on the orthogonal decomposition of $L_2 (X,\mu )$ into the
$G$-invariant closed subspaces $L_2 (X,\mu )_\wm$ and $L_2 (X,\mu )_\cpct$
(as described prior to Lemma~\ref{L-seq rel}) and the relationship
between compact orbit closures and finite-dimensional subrepresentations
recorded below in Proposition~\ref{P-cpct finite}.

To define the sequence analogue of c.p.\ approximation
entropy, let $M$ be a von Neumann algebra, $\sigma$ a faithful normal state on
$M$, and $\beta$ a $\sigma$-preserving action of the discrete group $G$
on $M$ by $^*$-automorphisms. Let $\fs = \{ s_n \}_n$ be a sequence in $G$.
Recall the quantities $\rcp_\sigma (\cdot , \cdot )$ from the beginning of
Section~\ref{S-Pinsker}. 
For a finite set $\Upsilon\subseteq M$ and $\delta > 0$ we set
\[ \hcpa_\sigma^\fs (\beta , \Upsilon , \delta ) = \limsup_{n\to\infty}
\frac1n \ln \rcp_\sigma \bigg( \bigcup_{i=1}^n \beta_{s_i} (\Upsilon ), \delta \bigg) \]
and define
\begin{align*}
\hcpa_\sigma^\fs (\beta , \Upsilon ) &= \sup_{\delta > 0}
\hcpa_\sigma^\fs (\beta , \Upsilon , \delta ) , \\
\hcpa_\sigma^\fs (\beta ) &= \sup_{\Upsilon}
\hcpa_\sigma^\fs (\beta , \Upsilon )
\end{align*}
where the last supremum is taken over all finite subsets $\Upsilon$ of $M$.
We call $\hcpa_\sigma^\fs (\beta , \Upsilon )$ the {\em sequence c.p.\
approximation entropy} of $\beta$.

In analogy with the upper $\mu$-$\ell_1$-isomorphism density from Section~\ref{S-Pinsker},
given a sequence $\fs = \{ s_n \}_n$ in $G$, $f\in\LX$, $\lambda\geq 1$, and $\delta > 0$
we set
\[ \upind_\mu (f,\lambda ,\delta ; \fs ) =
\limsup_{n\to\infty} \frac1n \varphi_{f,\lambda ,\delta} (\{ s_1 , \dots , s_n \} ) \]
and define
\begin{align*}
\upind_\mu (f,\lambda ; \fs ) &= \sup_{\delta > 0} \upind_\mu (f,\lambda ,\delta ; \fs ) , \\
\upind_\mu (f; \fs ) &= \sup_{\lambda\geq 1} \upind_\mu (f,\lambda ; \fs ) .
\end{align*}
We could also define the lower version but this is less significant for our applications,
in which we would always be able to pass to a subsequence.

To establish (10)$\Rightarrow$(5) in Theorem~\ref{T-null}
we will need the relationship between relatively compact orbits and
finite-dimensional invariant subspaces given by Proposition~\ref{P-cpct finite},
which is presumably well known. For this we record a couple of lemmas.

\begin{lemma}\label{L-AP to compact}
Suppose that $G$ acts on a Banach space $V$ by
isometries. Then the action factors through a compact Hausdorff group
(for a strongly continuous action on $V$ and a homomorphism
from $G$ into this group) if and only if the norm closure of
the orbit of each vector is compact.
\end{lemma}

\begin{proof}
The ``only if'' part is obvious. Suppose that
the action is compact. Denote by $E$ the closure of the image of $G$ in the space
$\cB (V)$ of bounded linear operators on $V$ with respect to the strong operator topology.
Then $E$ is prescisely the closure of $\{(sv)_{v\in V}:s\in G\}$ in
$\prod_{v\in V}\overline{Gv}$. Thus
$E$ is a compact Hausdorff space. Note that multiplication
on the unit ball of $\cB (V)$ is jointly continuous for
the strong operator topology. It follows easily that $E$ is a compact Hausdorff
group of isometric operators on $V$ and that the action of $E$ on $V$ is strongly
continuous. This yields the ``if'' part.
\end{proof}

A {\it compactification} of $G$ is a pair $(\Gamma, \varphi)$ where
$\Gamma$ is a compact Hausdorff group and $\varphi$ is
a homomorphism from $G$ to $\Gamma$ with dense image.
The Bohr compactification $\overline{G}$ of $G$ is the spectrum of the space
of almost periodic bounded functions on $G$ and has the universal property
that every compactification of $G$ factors through it (see \cite{AH}).

\begin{lemma}\label{L-AP to compact 2}
Suppose that $G$ acts on a von Neumann algebra
$M$ by $^*$-automorphisms. Let $\sigma$ be a $G$-invariant faithful
normal state on $M$ such that the induced unitary representation of $G$ on
$L^2 (M,\sigma )$ has the property that the norm closure of
the orbit of each vector is compact. Then the action
factors through an ultraweakly continuous action of $\overline{G}$ on $M$.
\end{lemma}

\begin{proof}
Denote the unitary on $L^2 (M,\sigma )$ corresponding to $s\in G$ by $U_s$.
By Lemma~\ref{L-AP to compact} the unitary
representation $s\mapsto U_s$ of $G$ factors through a strongly continuous
unitary representation of $\overline{G}$. Denote the unitary on
$L^2 (M,\sigma )$ corresponding to $t\in \overline{G}$ by $U_t$. Note that the
action of $s\in G$ on $M$ is conjugation by $U_s$. It
follows that the conjugation by $U_t$ for each $t\in \overline{G}$
preserves $M$.
\end{proof}

For any ultraweakly continuous action of a locally compact group $\Gamma$ on
a von Neumann algebra as automorphisms, there is a $\Gamma$-invariant
ultraweakly dense unital $C^*$-subalgebra of the von Neumann algebra on
which the action of $\Gamma$ is strongly continuous \cite[Lemma 7.5.1]{Ped}.
For any strongly continuous action of a compact group on a Banach space
as isometries, the subspace of
elements whose orbit spans a finite-dimensional subspace is
dense \cite[Theorem III.5.7]{BD}. Thus we have:

\begin{proposition}\label{P-cpct finite}
Under the hypotheses of Lemma~\ref{L-AP to compact 2},
there are a $\overline{G}$-invariant ultraweakly dense unital $C^*$-subalgebra
$A$ of $M$ on which the action of $\overline{G}$ is strongly continuous and a
norm dense $^*$-subalgebra $B$ of $A$ such that the orbit of every
element in $B$ spans a finite-dimensional subspace.
\end{proposition}

The following is a local version of Theorem~5.2 of \cite{Huang} and is a consequence
of the proof given there in conjunction with the Rosenthal-Dor $\ell_1$ theorem, which
asserts that a bounded sequence in a Banach space has either a weakly Cauchy
subsequence or a subsequence equivalent to the standard basis of $\ell_1$ \cite{ell1,Dor}.

\begin{lemma}\label{L-tame-cpct}
Let $f$ be a function in $\LX$ whose $G$-orbit does not contain an infinite subset
equivalent to the standard basis of $\ell_1$. Then $f\in L^2 (X,\mu )_\cpct$.
\end{lemma}

\noindent The converse of lemma~\ref{L-tame-cpct} is false. Indeed by \cite{SEUPEM}
every free ergodic $\Zb$-system has a minimal topological model with uniformly positive
entropy, which means in particular that there are $L^\infty$ functions whose $G$-orbit
has a positive density subset equivalent to the standard basis of $\ell_1$.

In the following theorem $(\Omega , G)$ is the topological
$G$-system associated to $(X,\sB ,G,\mu )$ described before Theorem~\ref{T-Pinsker}.

\begin{theorem}\label{T-null}
Let $f\in\LX$. Then the following are equivalent:
\begin{enumerate}
\item $f\notin\mathfrak{N}_X$,

\item there is a $\mu$-IN-pair $(\sigma_1 , \sigma_2 )\in \Omega\times\Omega$
such that $f(\sigma_1 ) \neq f(\sigma_2 )$,

\item there are $d>0$, $\delta>0$, and $\lambda>0$ such that for any $M>0$
there is some finite subset $F\subseteq G$ with $|F|\geq M$ such that
whenever $g_s$ for $s\in F$ are elements of $\LX$
satisfying $\| g_s - \alpha_s (f) \|_\mu < \delta$
for every $s\in F$ there exists an $I \subseteq F$
of cardinality at least $d|F|$ for which the linear map $\ell_1^I \to
\spn \{ g_s : s\in I \}$ sending the standard basis element with index $s\in I$
to $g_s$ has an inverse with norm at most $\lambda$,

\item $\upind_\mu (f;\fs) > 0$ for some sequence $\fs$ in $G$,

\item $f\notin L^2 (X,\mu )_\cpct$,

\item $\hcpa_\mu^\fs (\alpha , \{ f \} ) > 0$ for some sequence $\fs$ in $G$,

\item $\hcpa_\mu^\fs (\beta ) > 0$ for some sequence $\fs$ in $G$ where $\beta$ is the
restriction of $\alpha$ to the von Neumann subalgebra of $M$ dynamically generated by $f$.

\item there is a $\delta > 0$ such that every $g\in\LX$ satisfying
$\| g -f \|_\mu < \delta$ has an infinite $\ell_1$-isomorphism set,

\item there is a $\delta > 0$ such that every $g\in\LX$ satisfying
$\| g -f \|_\mu < \delta$ has arbitrarily large $\lambda$-$\ell_1$-isomorphism sets
for some $\lambda > 0$,

\item there is a $\delta > 0$ such that every $g\in\LX$ satisfying
$\| g -f \|_\mu < \delta$ has noncompact orbit closure in the operator norm.
\end{enumerate}
When $f\in C(X)$ we can add:
\begin{enumerate}
\item[(11)] $f\notin C(Y)$ whenever $\pi : X\to Y$ is a topological $G$-factor map
such that $\pi_* (\mu )$ is null,

\item[(12)] there is a $\mu$-IE-pair $(x_1 , x_2 )\in X\times X$
such that $f(x_1 ) \neq f(x_2 )$.
\end{enumerate}
\end{theorem}

\begin{proof}
(1)$\Rightarrow$(2) Argue as for (1)$\Rightarrow$(2) in Theorem~\ref{T-Pinsker}
using Lemma~\ref{L-pse to mind} and Proposition~\ref{P-IN basic}(1) instead of
Lemma~\ref{L-pe to pd} and Proposition~\ref{P-IE basic}(1).

(2)$\Rightarrow$(3). Apply the same argument as for (2)$\Rightarrow$(3) in
Theorem~\ref{T-Pinsker}, replacing $\lwind_\mu (\oA , \delta )$ by
$\lwind_\mu (\oA , \delta ; \fs )$ for a suitable sequence $\fs$ in $G$.

(3)$\Leftrightarrow$(4). Use the arguments for (6)$\Rightarrow$(4) and
(3)$\Rightarrow$(6) in the proof of Theorem~\ref{T-Pinsker}.

(3)$\Rightarrow$(6). Argue as for (4)$\Rightarrow$(7) in Theorem~\ref{T-Pinsker}.

(6)$\Rightarrow$(7). As in the case of complete positive approximation entropy,
if $N$ is an $G$-invariant von Neumann subalgebra of $\LX$ and $\fs$ is a sequence in
$G$ then for every finite subset $\Theta\subseteq N$ we have
$\hcpa_{\mu |_N}^\fs (N , \Theta ) = \hcpa_\mu^\fs (M , \Theta )$,
which follows from the fact that there is a $\mu$-preserving conditional expectation
from $\LX$ onto $N$ \cite[Prop.\ V.2.36]{Tak1} (cf.\ Proposition~3.5 in \cite{Voi}).

(7)$\Rightarrow$(1). This can be deduced from Lemma~\ref{L-rcp upper bd} in the same
way that (8)$\Rightarrow$(1) of Theorem~\ref{T-Pinsker} was.

(6)$\Rightarrow$(5). Suppose that $f\in L^2 (X,\mu )_\cpct$. Let $\delta > 0$.
Then the $G$-orbit $\{ \alpha_s (f) : s\in G \}$ contains a finite $\delta$-net $\Omega$
for the $L^2$-norm. Take a finite Borel partition $\cP$ of $X$
such that $\sup_{g\in\Omega} \esssup_{x,y\in P} | g(x) - g(y) | < \delta$ for each
$P\in\cP$. Let $B$ be the $^*$-subalgebra of $M$ generated by $\cP$ and let
$\varphi$ be the $\mu$-preserving condition expectation of $L^\infty (X,\mu )$ onto $B$.
Now for every $s\in G$ we can find a $g\in\Omega$ such that
$\| \alpha_s (f) - g \|_\mu \leq\delta$ so that
\[ \| \varphi (\alpha_s (f)) - \alpha_s (f) \|_\mu
\leq \| \varphi (\alpha_s (f) - g) \|_\mu + \| \varphi (g) - g \|_\mu +
\| g - \alpha_s (f) \|_\mu < 3\delta . \]
Taking the inclusion $\psi : B\hookrightarrow L^\infty (X,\mu )$ it follows that
for every finite set $F\subseteq G$ we have
$(\varphi , \psi , B)\in\CPA_\mu (\{ \alpha_s (f) : s\in F \} , 3\delta )$
and hence $\rcp_\mu (\{ \alpha_s (f) : s\in F \} , 3\delta ) \leq \dim B$. Since $\delta$
is arbitrary we conclude that $\hcpa_\mu^\fs (\alpha , \{ f \} ) = 0$ for every
sequence $\fs$ in $G$.

(5)$\Rightarrow$(1). Since $f\notin L^2 (X,\mu )_\cpct$ the restriction of $\alpha$ to
the von Neumann algebra $N$ dynamically generated by $f$ has nonzero weak mixing
component at the unitary level,
and so there exists a finite partition $\cP$ of $X$ that is $N$-measurable but not
$\mathfrak{D}_X$-measurable (where $\mathfrak{D}_X$ is as defined prior to
Lemma~\ref{L-seq rel}). By Lemma~\ref{L-seq rel} there is a sequence $\fs$ in $G$ such that
$\hmeas_\mu^\fs (X,\cP ) \geq \Hmeas (\cP | \mathfrak{D}_X ) > 0$, from which we infer
that $f\notin\mathfrak{N}_X$.

(5)$\Rightarrow$(8). This follows by observing that if
$\{ g_k \}_{k\in\Nb}$ were a sequence in $\LX$ converging to $f$ in the $\mu$-norm
such that each $g_k$ lacks an infinite $\ell_1$-isomorphism set, then we would have
$g_k \in L^2 (X,\mu )_\cpct$ for each $k$ by Lemma~\ref{L-tame-cpct} and hence
$f\in L^2 (X,\mu )_\cpct$.

(8)$\Rightarrow$(9). Trivial.

(9)$\Rightarrow$(10). It is easy to see that if an element $g$ of $\LX$ has arbitrarily large
$\lambda$-$\ell_1$-isomorphism sets for some $\lambda > 0$ then its $G$-orbit fails
to have a finite $\varepsilon$-net for some $\varepsilon > 0$ depending on $\lambda$
and $\| g \|$.

(10)$\Rightarrow$(5). Suppose contrary to (5) that $f\in L^2 (X,\mu )_\cpct$. Then the
restriction of $\alpha$ to the von Neumann algebra $N$ dynamically generated by $f$
has the property that the norm closure of the $G$-orbit of each vector in $L^2 (N,\mu )$
is compact. By Proposition~\ref{P-cpct finite} this contradicts (10).

Suppose now that $f\in C(X)$. To prove (2)$\Rightarrow$(12), observe that the
inclusion\linebreak $C(\supp (\mu ))\subseteq \LX$ gives rise to a topological $G$-factor map
$\Omega\to \supp (\mu )$, so that we can apply Proposition~\ref{P-IN basic}(5).
For (12)$\Rightarrow$(3) apply the same argument as for (2)$\Rightarrow$(3).
For (11)$\Rightarrow$(12) argue as for (11)$\Rightarrow$(12) in
Theorem~\ref{T-Pinsker}, this time using Proposition~\ref{P-IN basic}.
Finally, for (12)$\Rightarrow$(11) use Proposition~\ref{P-IN basic}(5).
\end{proof}

As pointed out at the beginning of the section and as used in the proof of
Theorem~\ref{T-IN SE}, Theorem~\ref{T-null} shows that a measure-preserving system is
isometric if and only if it is null, which in the case of a $\Zb$-action on a
Lebesgue space is a result of Kushnirenko \cite{Kush}. Note that Theorem~\ref{T-null}
does not depend in any way on Theorem~\ref{T-IN SE}. In conjunction with
Theorem~\ref{T-Pinsker}, Theorem~\ref{T-null} gives
a geometric explanation for the well-known fact that isometric
measure-preserving systems have zero entropy.

Condition (8) in Theorem~\ref{T-null} is the analogue of tameness from topological
dynamics \cite{tame, Ind}. Its equivalence with the other conditions shows that
tameness as distinct from nullness is a specifically topological-dynamical phenomenon.
This equivalence relies in part, via Lemma~\ref{L-tame-cpct}, on the local argument
used by Huang in the case $G=\Zb$ to prove that if $X$ is metrizable and the
system $(X,G)$ is tame then every $G$-invariant Borel probability measure on $X$ is
measure null \cite[Theorem 5.2]{Huang}. The following example illustrates that
the converse of Huang's result fails in an extreme way.

\begin{example}\label{E-tame non-measure null}
By Lemma~7.2 of \cite{Ind}, when $G$ is Abelian, every nontrivial
metrizable weakly mixing system $(X, G)$ is completely untame. We
will show how to construct a weakly mixing uniquely ergodic
subshift $(X,\Zb )$ with the invariant measure supported at a
fixed point. We indicate first how to construct weakly mixing
subshifts $(X,\Zb)$ with $X\subseteq \{0, 1\}^{\Zb}$. We shall
construct two elements $p$ and $q$ in $\{0, 1\}^{\Zb}$ so that
$(p, q)$ is a transitive point for $X\times X$ where $X$ is the
orbit closure of $p$, and determine two increasing sequences
$0=a_1<a_2<\dots$ and $0=a'_1<a'_2<\dots$ of nonnegative integers with
$a_n\le a'_n<a_{n+1}$ for all $n$. Set $p(k)=q(k)=0$ unless
$a_n\le k\le a'_n$ for some $n$. Set $p(0)=1$ and $q(0)=0$.
Suppose that we have determined $a_1,\dots , a_m$ and $a'_1,\dots,
a'_m$ and $p(k)$ and $q(k)$ for all $k\le a'_m$. Take $a_{m+1}$ to
be any integer bigger than $\max(m, a'_{m})$. If $m+1 \equiv 1
\mod 3$, we take $a'_{m+1}=a_{m+1}+2m$ and set $q$ to be $0$ on
the interval $[a_{m+1}, a'_{m+1}]$ while setting $p$ on $[a_{m+1},
a'_{m+1}]$ to be the shift of $q$ on $[-m, m]$. If $m+1 \equiv 2 \mod
3$, we take $a'_{m+1}=a_{m+1}+2m$ and set $p$ to be $0$ on the
interval $[a_{m+1}, a'_{m+1}]$ while setting $q$ on $[a_{m+1},
a'_{m+1}]$ to be the shift of $p$ on $[-m, m]$. If $m+1 \equiv 0 \mod
3$, consider the set $S$ consisting of the sequences of values of $p$ 
over the finite subintervals of $(-\infty, a'_{m}]$. Consider pairs of elements 
in $S$ of the same length which don't appear as the sequence of values of $(p,q)$ 
on some finite subinterval of $(-\infty, a'_m]$. Choose one such pair
$(f, g)$ with the smallest length $d$. Set $a'_{m+1}=a_{m+1}+d-1$
and set $p$ and $q$ to be $f$ and $g$, respectively, on the interval $[a_{m+1},
a'_{m+1}]$. Then it is clear that $(p, q)$ is a
transitive point for $X\times X$ where $X$ is the orbit closure
of $p$.

In general, note that if $U$ is an open subset of $X$ such that there is
an infinite subset $H$ of $G$ for which the sets $hU$ for $h\in H$ are
pairwise disjoint, then $\mu(U)=0$ for any invariant Borel
probability measure $\mu$ on $X$.
Denote by $Y$ the complement of the union of all such $U$. Then
every invariant Borel probability measure $\mu$ of $X$ is supported on
$Y$. We claim that in the construction above, by choosing $a_{m+1}$
large enough at each step, we can arrange for $Y$ to consist of only the point $0$.
Then $(X, \Zb)$ is uniquely ergodic and the invariant measure is supported
at $0$. Note that $Y$ is always an invariant closed subset of $X$.
Let $V$ be the subset of $X$ consisting of elements with
value $1$ at $0$.
It suffices to find an infinite subset $H=\{h_1, h_2, \dots \}$ of $\Zb$
such that the sets $hV$ for $h\in H$ are pairwise disjoint.
Set $h_1=0$. Suppose that we have determined $a_1,\dots, a_m$ and
$a'_1,\dots, a'_m$ and $h_1,\dots, h_m$ and
$p(k)$ and $q(k)$ for all $k\le a'_m$. Take $h_{m+1}>h_m+a'_m-a_1$
and $a_{m+1}>a'_m+h_{m+1}-h_1$.
\end{example}

The following theorem addresses the extreme case of complete nonnullness,
where we see the same kind of topologization as in the entropy setting
of Theorem~\ref{T-cpe}. For the
definitions of the topological-dynamical properties of complete nonnullness,
complete untameness, uniform nonnullness of all orders, and uniform untameness
of all orders, see Sections~5 and 6 of \cite{Ind}.

\begin{theorem}\label{T-WM equivalence}
Let $\oX = (X,\sX ,\mu ,G)$ be a measure-preserving dynamical system. Let
$\boldsymbol{\Omega} = (\Omega , G)$ be the associated topological dynamical system
on the spectrum $\Omega$ of $L^\infty (X,\mu )$. Then the following are equivalent:
\begin{enumerate}
\item $\oX$ is weakly mixing,

\item $\oX$ is completely nonnull,

\item for every nonscalar $f\in\LX$ there is a $\lambda\geq 1$ such that for
every $m\in\Nb$ there exists a set $I\subseteq G$ of cardinality $m$ such that
$\{ \alpha_s (f) : s\in I \}$ is $\lambda$-equivalent to the standard basis
of $\ell_1^m$,

\item every nonscalar element of $\LX$ has an infinite $\ell_1$-isomorphism set,

\item $\boldsymbol{\Omega}$ is completely nonnull,

\item $\boldsymbol{\Omega}$ is completely untame,

\item $\boldsymbol{\Omega}$ is uniformly nonnull of all orders,

\item $\boldsymbol{\Omega}$ is uniformly untame of all orders.
\end{enumerate}
\end{theorem}

\begin{proof}
(1)$\Rightarrow$(8). Use Theorems~8.2, 8.6, and 9.10 of \cite{Ind}.

(8)$\Rightarrow$(7)$\Rightarrow$(5) and (8)$\Rightarrow$(6).
These implications hold for any topological system (see Sections~5 and 6 of \cite{Ind}).

(6)$\Rightarrow$(4). Apply Propositions~6.4 and 6.6 of \cite{Ind}.

(5)$\Rightarrow$(3). Apply Proposition~5.4 and Theorem~5.8 of \cite{DEBS}.

(4)$\Rightarrow$(2), (3)$\Rightarrow$(2), and (2)$\Rightarrow$(1). Apply Theorem~\ref{T-null}.
\end{proof}

In analogy with Proposition~\ref{P-joining Pinsker},
if $(Y,\sY ,\nu ,G)$ and $(Z,\sZ ,\omega ,G)$
are measure-preserving $G$-systems and $\varphi : L^\infty (Y,\nu ) \to L^\infty (Z,\omega )$
is a $G$-equivariant unital positive linear map such that $\omega\circ\varphi = \nu$, then
$\varphi (\mathfrak{N}_X ) \subseteq \mathfrak{N}_Y$. One can deduce this using the
characterization of functions in the maximal null von Neumann algebra in terms of either
$\ell_1$-isomorphism sets or compact orbit closures in $L^2$.
In particular we see that isometric systems are disjoint from weakly mixing systems.
Of course it is well known more generally that distal systems are disjoint from weak mixing
systems (see Chapter~6 of \cite{ETJ}).


\begin{thebibliography}{999}

\bibitem{AH}
E. M. Alfsen and P. Holm. A note on compact representations and almost
periodicity in topological groups. {\it Math.\ Scand.} {\bf 10} (1962), 127--136.

\bibitem{BerRos}
V. Bergelson and J. Rosenblatt. Mixing actions of groups.
{\it Illinois J. Math.} {\bf 32} (1988), 65--80.

\bibitem{BGH}
F. Blanchard, E. Glasner, and B. Host. A variation on the variational
principle and applications to entropy pairs. 
{\it Ergod.\ Th.\ Dynam.\ Sys.} {\bf 17} (1997), 29--43.

\bibitem{BHMMR}
F. Blanchard, B. Host, A. Maass, S. Martinez, and D. J. Rudolph.
Entropy pairs for a measure. {\it Ergod.\ Th.\ Dynam.\ Sys.} {\bf 15} (1995), 621--632.

\bibitem{BD}
T. Br\"{o}cker and T. tom Dieck. {\it Representations of Compact Lie Groups.}
Translated from the German manuscript.
Corrected reprint of the 1985 translation.
Graduate Texts in Mathematics, 98. Springer-Verlag, New York, 1995.

\bibitem{CFW}
A. Connes, J. Feldman, and B. Weiss. An amenable equivalence relation is generated by
a single transformation. {\it Ergod.\ Th.\ Dynam.\ Sys.} {\bf 1} (1981), 431--450.

\bibitem{Dan}
A. I. Danilenko. Entropy theory from the orbital point of view.  
{\it Monatsh.\ Math.}  {\bf 134}  (2001), 121--141.

\bibitem{Dor}
L. E. Dor. On sequences spanning a complex $l_1$-space.
{\it Proc.\ Amer.\ Math.\ Soc.} {\bf 47} (1975), 515--516.

\bibitem{Ellis}
R. Ellis. Universal minimal sets. {\it Proc.\ Amer.\ Math.\ Soc.} {\bf 11}
(1960), 540--543.

\bibitem{Fur}
H. Furstenberg. {\it Recurrence in Ergodic Theory and Combinatorial Number Theory.}
Princeton University Press, Princeton, N.J., 1981.

\bibitem{mu}
E. Glasner. A simple characterization of the set of $\mu$-entropy pairs
and applications. {\it Israel J. Math.} {\bf 102} (1997), 13--27.

\bibitem{ETJ}
E. Glasner. {\it Ergodic Theory via Joinings.} American Mathematical
Society, Providence, RI, 2003.

\bibitem{tame}
E. Glasner. On tame dynamical systems. {\it Colloq.\ Math.} {\bf 105} (2006), 283--295.

\bibitem{ETWP}
E. Glasner, J.-P. Thouvenot, and B. Weiss. Entropy theory without a past.
{\it Ergod.\ Th.\ Dynam.\ Sys.} {\bf 20} (2000), 1355--1370.

\bibitem{SEUPEM}
E. Glasner and B. Weiss. Strictly ergodic, uniform positive entropy models.
{\it Bull.\ Soc.\ Math.\ France} {\bf 122} (1994), 399--412.

\bibitem{GW}
E. Glasner and B. Weiss. Quasi-factors of zero entropy systems.
{\it J. Amer.\ Math.\ Soc.} {\bf 8} (1995), 665--686.

\bibitem{Interplay}
E. Glasner and B. Weiss. On the interplay between measurable and
topological dynamics.  {\it Handbook of Dynamical Systems.} Vol. 1B,
597--648. Elsevier, Amsterdam, 2006.

\bibitem{Gode}
R. Godement. Les fonctions de type positif et la th\'eorie des groupes.
{\it Trans.\ Amer.\ Math.\ Soc.} {\bf 63} (1948), 1--84.

\bibitem{Huang}
W. Huang. Tame systems and scrambled pairs under an Abelian group action.
{\it Ergod.\ Th.\ Dynam.\ Sys.} {\bf 26} (2006), no. 5, 1549--1567.

\bibitem{HMRY}
W. Huang, A. Maass, P. P. Romagnoli, and X. Ye. Entropy pairs and a
local Abramov formula for a measure theoretical entropy of open
covers. {\it Ergod.\ Th.\ Dynam.\ Sys.} {\bf 24} (2004),  
1127--1153.

\bibitem{HMY}
W. Huang, A. Maass, and X. Ye. Sequence entropy pairs and complexity pairs for
a measure. {\it Ann.\ Inst.\ Fourier (Grenoble)} {\bf 54} (2004), 1005--1028.

\bibitem{LVRA}
W. Huang and X. Ye. A local variational relation and applications.
{\it Israel J. Math.} {\bf 151} (2006), 237--279.

\bibitem{HYZ}
W. Huang, X. Ye, and G. Zhang. A local variational principle for
conditional entropy. {\it Ergod.\ Th.\ Dynam.\ Sys.} {\bf 26}
(2006), 219--245.

\bibitem{LETCDAGA}
W. Huang, X. Ye, and G. Zhang. Local entropy theory for a countable discrete 
amenable group action. Preprint, 2007.

\bibitem{KM}
M. G. Karpovsky and V. D. Milman. Coordinate density of sets of vectors.
{\it Discrete Math.} {\bf 24} (1978), 177--184.

\bibitem{Katok}
A. Katok. Lyapunov exponents, entropy and periodic orbits for diffeomorphisms.
{\it Inst.\ Hautes {\'E}tudes Sci.\ Publ.\ Math.} {\bf 51} (1980), 137--173.

\bibitem{CDST}
A. S. Kechris. {\it Classical Descriptive Set Theory.}
Graduate Texts in Mathematics, 156. Springer, New York, 1995. 

\bibitem{EID}
D. Kerr. Entropy and induced dynamics on state spaces.
{\it Geom.\ Funct.\ Anal.} {\bf 14} (2004), 575--594.

\bibitem{DEBS}
D. Kerr and H. Li. Dynamical entropy in Banach spaces.
{\it Invent.\ Math.} {\bf 162} (2005), 649--686.

\bibitem{Ind}
D. Kerr and H. Li. Independence in topological and $C^*$-dynamics.
To appear in {\it Math.\ Ann.}

\bibitem{Kush}
A. G. Kushnirenko. On metric invariants of entropy type.
{\it Russian Math.\ Surveys} {\bf 22} (1967), 53--61.

\bibitem{LW}
E. Lindenstrauss and B. Weiss. Mean topological dimension.
{\it Israel J. Math.} {\bf 115} (2000), 1--24.

\bibitem{JMO}
J. Moulin Ollagnier. {\it Ergodic Theory and Statistical Mechanics.}
Lecture Notes in Math., 1115. Springer, Berlin, 1985.

\bibitem{OW}
D. S. Ornstein and B. Weiss. Entropy and isomorphism theorems for
actions of amenable groups.  {\it J. Analyse Math.} {\bf 48} (1987),
1--141.

\bibitem{Pat}
A. L. T. Paterson. {\it Amenability.}
Mathematical Surveys and Monographs, 29.
American Mathematical Society, Providence, RI, 1988.

\bibitem{Ped}
G. K. Pedersen. {\it $C^*$-algebras and their Automorphism Groups}.
London Mathematical Society Monographs, 14. Academic Press, Inc.,
London-New York, 1979

\bibitem{Popa}
S. Popa. Correspondences. INCREST preprint, 1986.

\bibitem{Roma}
P. P. Romagnoli. A local variational principle for the topological
entropy. {\it Ergod.\ Th.\ Dynam.\ Sys.} {\bf 23} (2003), 1601--1610.

\bibitem{FUG}
A. Rosenthal. Finite uniform generators for ergodic, finite entropy,
free actions of amenable groups. {\it Probab.\ Theory Related Fields}
{\bf 77} (1988), 147--166.

\bibitem{ell1}
H. P. Rosenthal. A characterization of Banach spaces containing $l_1$.
{\it Proc.\ Nat.\ Acad.\ Sci.\ USA} {\bf 71} (1974), 2411--2413.

\bibitem{EMAGA}
D. J. Rudolph and B. Weiss. Entropy and mixing for amenable group actions.
{\it Ann.\ Math.} {\bf 151} (2000), 1119--1150.

\bibitem{Sauer}
N. Sauer. On the density of families of sets.
{\it J. Combinatorial Theory Ser.\ A} {\bf 13} (1972), 145--147.

\bibitem{Shelah}
S. Shelah. A combinatorial problem; stability and order for models
and theories in infinitary languages. {\it Pacific J. Math.}
{\bf 41} (1972), 247--261.

\bibitem{Tak1}
M. Takesaki. {\it Theory of Operator Algebras I.}
Encyclopaedia of Mathematical Sciences, 124.
Springer-Verlag, Berlin, 2002.

\bibitem{Tao}
T. Tao. The dichotomy between structure and randomness, arithmetic progressions,
and the primes. To appear in the Proceedings of the Madrid ICM 2006.
%arxiv:math.NT/0512114.

\bibitem{Voi}
D. V. Voiculescu. Dynamical approximation entropies
and topological entropy in operator algebras. {\it Comm.\ Math.\ Phys.}
{\bf 170} (1995), 249--281.

\bibitem{Wal}
P. Walters. {\it An Introduction to Ergodic Theory.} Graduate Texts in
Mathematics, 79. Springer-Verlag, New York, Berlin, 1982.

\bibitem{SEM}
B. Weiss. Strictly ergodic models for dynamical systems.
{\it Bull.\ Amer.\ Math.\ Soc. (N.S.)} {\bf 13} (1985), 143--146.

\bibitem{Zimmer}
R. Zimmer. Ergodic actions with generalized discrete spectrum.
{\it Illinois J. Math.} {\bf 20} (1976), 555--588.
\end{thebibliography}
\end{document}